\documentclass
{article}

\usepackage[utf8]{inputenc}
\usepackage[margin=2.5cm]{geometry}
\usepackage{amsmath,amssymb,amsthm}
\usepackage{stmaryrd}
\usepackage{mathtools}
\usepackage{dsfont}
\usepackage{hyperref}
\usepackage[dvipsnames]{xcolor}
\usepackage{ulem}
\usepackage{multirow}
\usepackage{cases}

\usepackage{wrapfig}

\usepackage{silence}
\usepackage{constants}
\renewconstantfamily{normal}{symbol=c}
\newconstantfamily{assumptions}{symbol=C}

\usepackage{graphicx}
\usepackage{tikz}
\usetikzlibrary{math}
\usepackage{pgfplots}
\pgfplotsset{compat=1.17}
\usepackage{caption}
\usepackage{subcaption}

\hypersetup{
    colorlinks=true,
    linkcolor=blue,
    filecolor=magenta,      
    urlcolor=Blue,
    citecolor=Blue,
}

\usepackage[
colorinlistoftodos,prependcaption]{todonotes}
\usepackage{setspace}
\newcounter{mycomment}

\newtheorem{theorem}{Theorem}[section]
\newtheorem{lemma}[theorem]{Lemma}
\newtheorem{proposition}[theorem]{Proposition}
\newtheorem{corollary}[theorem]{Corollary}
\newtheorem{remark}[theorem]{Remark}

\DeclarePairedDelimiter\norm{\lVert}{\rVert}

\newcommand{\bA}{\mathbb{A}}
\newcommand{\bE}{\mathbb{E}}

\newcommand{\bB}{\mathbb{B}}

\newcommand{\bM}{\mathbb{M}}
\newcommand{\bN}{\mathbb{N}}

\newcommand{\bP}{\mathbb{P}}

\newcommand{\bR}{\mathbb{R}}

\newcommand{\bH}{\mathbb{H}}

\newcommand{\bS}{\mathbb{S}}
\newcommand{\bZ}{\mathbb{Z}}
\newcommand{\bX}{\mathbb{X}}
\newcommand{\bV}{\mathbb{\hspace{-.07mm}V\hspace{-.2mm}}}

\newcommand{\cA}{\mathcal{A}}

\newcommand{\cH}{\mathcal{H}}

\newcommand{\cN}{\mathcal{N}}

\newcommand{\cX}{\mathcal{X}}

\newcommand{\V}{\mathbb{V}}

\newcommand{\dint}{\mathrm{d}}
\newcommand{\aaa}{\theta_1}
\newcommand{\bbb}{\theta_2}

\newcommand{\dist}{\operatorname{dist}}
\newcommand{\diam}{\operatorname{diam}}

\newcommand{\eqrefb}[1]{{\normalfont(\ref*{#1})}}

\begin{document}

\author{
Gilles Bonnet\footnotemark[1] \footnotemark[2],\;
Anna Gusakova\footnotemark[3]}

\renewcommand{\thefootnote}{\fnsymbol{footnote}}

\footnotetext[1]{
    Bernoulli Institute, University of Groningen, Nijenborgh 4, NL-9747 AG Groningen, Netherlands.
    Email: g.f.y.bonnet@rug.nl
}
\footnotetext[2]{
    CogniGron (Groningen Cognitive Systems and Materials Center), University of Groningen, Nijenborgh 4, NL-9747 AG Groningen, Netherlands.
}
\footnotetext[3]{
    University of Münster, Institut für Mathematische Stochastik, Orl\'eans-Ring 10, 48149 Münster, Germany. Email: gusakova@uni-muenster.de
}

\title{\textbf{Concentration inequalities for Poisson $U$-statistics 
}}

\maketitle

\begin{abstract}

In this article we obtain concentration inequalities for Poisson $U$-statistics $F_m(f,\eta)$ of order $m\ge 1$ with kernels $f$ under general assumptions on $f$ and the intensity measure $\gamma \Lambda$ of underlying Poisson point process $\eta$. The main result are new concentration bounds of the form 
\[
  \mathbb{P}(|F_m ( f , \eta) -\mathbb{E} F_m ( f , \eta)| \ge t)\leq 2\exp(-I(\gamma,t)), 
  \]
  where $I(\gamma,t)$ is of optimal order in $t$,  namely it satisfies $I(\gamma,t)=\Theta(t^{1\over m}\log t)$ as $t\to\infty$ and $\gamma$ is fixed. The function $I(\gamma,t)$ is given explicitly in terms of parameters of the assumptions satisfied by $f$ and $\Lambda$. One of the key ingredients of the proof is bounding the centred moments of $F_m(f,\eta)$. We discuss the optimality of obtained concentration bounds and consider a number of applications related to Gilbert graphs and Poisson hyperplane processes in constant curvature spaces.


\end{abstract}

\tableofcontents


\section{Introduction}

\paragraph{Purpose of the article:}
Let $\eta$ be a Poisson process of intensity measure $\gamma \Lambda$ 
 on a measure space $(\bX,\cX,\Lambda)$, $\gamma>0$, where $\Lambda$ is a $s$-finite measure (i.e.\ a countable sum of finite measures).
Let $m\in\bN$.
We consider a Poisson $U$-statistic of (measurable) \textit{kernel} $f\colon\bX^m\to \bR$, which is a random variable
$$
F_m=F_m(f,\eta) = \int f(x_1,\ldots,x_m) \, \eta^{(m)}(d x_1, \ldots , d x_m),
$$
where $\eta^{(m)}$ denotes the factorial measure of $\eta$.
When $\Lambda$ is non-atomic measure, points of $\eta$ have no multiplicity and thus $F_m$ can be written in the following simpler form
$$
F_m= \sum_{(x_1,\ldots,x_m)\in \eta^m_{\neq}} f(x_1,\ldots,x_m) ,
$$
where $\eta^m_{\neq}$ is the collection of $m$-tuples of distinct points of $\eta$.
In this article we establish concentration bounds for $F_m$ around its mean $\bE F_m$, under additional assumptions,  
which are formulated in terms of bounds for the integrals of the special type of the kernel $f$ with respect to $\Lambda$ (see \eqref{eq:fBound}). This integrals appear in the formulas for the centred moment of Poisson $U$-statistic $F_m$. 
The exact formulation of the above assumption require some additional notation and terminology and is therefore postponed to the Section \ref{sec:Assumptions}, where we also compare it with other natural sets of assumptions, such as for example the \textit{local} U-statistics as studied by Bachmann and Reitzner in \cite{BR18}.

We are aiming for the concentration bounds of the following form. We find \textit{rate functions} $I_+ , I_- \colon [0,\infty) \times [0,\infty) \to [0,\infty)$, with $\lim_{t\to\infty}I_+(\gamma,t)=\lim_{t\to\infty}I_-(\gamma,t)=\infty$ for any fixed $\gamma$, and such that
\begin{align*}
\bP(F_m-\bE F_m\leq -t)&\leq \exp(-I_-(\gamma,t)) ,\\
\bP(F_m-\bE F_m\geq t)&\leq \exp(-I_+(\gamma,t)) .
\end{align*}
The rate functions depends on several quantities associated to the kernel $f$, the most prominent being the \textit{order} $m$.
At the end of this introduction, we provide a qualitative description of the rate functions, and give pointers to specific results in later sections for their precise expressions.

\paragraph{Background on Poisson $U$-statistics and concentration bounds:}
Poisson $U$-statistics cover a wide class of functionals of Poisson processes, which play an important role in stochastic geometry.
Classical examples include the
total edge length of random graphs, different characteristics of intersection process of Poisson hyperplane process,
and volumes of random simplices.
This wide range of applications drags a lot of attention and stimulated the intensive investigation of the properties of Poisson $U$-statistics.

The classical settings when $U$-statistics are defined for binomial processes instead of Poisson processes are well studied and there is an extensive literature on this topic, see e.g.\ \cite{Lee90} for more detail and further references.
In contrast to this, the study of Poisson $U$-statistic appeared to be a challenging task, but some significant progress was made during the last years.
In particular the combination of modern techniques, such as Malliavin calculus on Poisson spaces together with Stein’s method for normal approximation, Wiener-It\^{o} chaos expansion and Fock space representation appeared to be very fruitful leading to the number of striking results including that exact formulas for the moments and abstract limit theorems \cite{ET14, L-RP13b, L-RP13, LPST14, PT13, RS13}. For further background material on this topic and for references we refer reader to \cite{LP} and \cite{PeccatiReitzner}.

Compared with the number of limit theorems available in the literature there are only few results concerning concentration inequalities specific for Poisson $U$-statistics. On the other hand the questions regarding concentration bounds for general Poisson functionals have been considered in the past and one can identify three different approaches to this problem. The first approach relies on a modified log-Sobolev \cite{wu2000new} and $\Phi$-Sobolev inequalities \cite{Chaf}. This technique was first applied by Wu \cite{wu2000new}, leading to a collection of concentration bounds under some rather restrictive conditions, and was further extended  in \cite{bachmann2016concentration} and \cite{GST21}. In \cite{BR18} this approach was in particularly used to derive concentration bounds for local Poisson $U$-statictics over $\bR^d$ with bounded kernel. The obtained concentration bounds have been used to analyse random graphs models \cite{bachmann2016concentration, BR18} and random polytopes \cite{GST21}. Another approach using general covariance identities for exponential
functions of Poisson processes was applied in \cite{BHP07, GieringerLast,HP02}. The resulting concentration inequalities were particularly useful to study concentration properties of
geometric functionals associated with the Poisson Boolean model \cite{GieringerLast}. In \cite{R13} Reitzner proved an analogue of Talagrand’s inequality for the convex distance on the Poisson space, which in turn was used in \cite{RST17} to derive concentration inequalities for general Poisson functionals and, in particular, for $U$-statistics around its median (see also \cite{HR-B03, PeccatiReitzner}). The obtained bounds provide a good rate for local $U$-statistics and $U$-statistics of order $1$ and they have been applied to random geometric graphs models \cite{RST17} and the Poisson flat processes \cite{PeccatiReitzner}. Recently using the general
transport inequality  the concentration bounds for so-called convex functionals of Poisson processes were proved by Gozlan, Herry and Peccati \cite{NRP21}. They also lead to concentration inequalities for Poisson
$U$-statistics around the median. In the very recent work \cite{ST23} Schulte and Th\"ale have used the cumulants method in order to establish moderate deviation principle and derive Bernstein and Cram\'er type concentration inequalities for Poisson $U$-statistics. The application of this method relies on bounds for the cumulants of Poisson $U$-statistic.

\paragraph{Description of our results:}
As it is seen from the above overview, almost all known concentration bounds for Poisson $U$-statistics are based on concentration bounds for general Poisson functionals. For this reason the estimates are often not sensitive to the specific structure of Poisson $U$-statistics and provide a good bound only in some special cases.
Our aim is to use the properties of Poisson $U$-statistics in order to obtain a new set of concentration inequalities under some mild assumptions (see assumptions \eqref{eq:fBound} in the next section), including local $U$-statistics as a particular case, as we will show in Lemma \ref{lem:relAssumptions}. 
The proof relies on the known formulas for the centred moments of Poisson $U$-statistic and Markov inequality. 
It should be noted that the obtained concentration bounds are optimal (up to the constants) and provide an improvement of known results by the logarithmic factor in case of $m\ge 2$. The detailed analysis of the bounds and their comparison with earlier results is postponed to Section \ref{sec:discussion}. As an application we will consider some functionals of random geometric graph models and Poisson hyperplane process in Euclidean and hyperbolic spaces. 

Our main results are summarized in the following theorem, which we present here in a simplified form. An extended form of this theorem with the explicit constants is given at the beginning of Section \ref{sec:discussion}.
The statement of this theorem involves assumption \eqref{eq:fBound} and the notation $f_1$ which are both introduced in the next section.

\begin{theorem}[Main result, simplified version]  \label{tm:Summary} \,
Let $F_m$ be a Poisson $U$-statistic with kernel $f$ satisfying $\|f_1\|_{L^2(\Lambda)} >0$, where $f_1$ is defined by \eqref{eq_fk}.
If 
\eqref{eq:fBound} is satisfied with $q\in[0,1]$, then there exist $C,c>0$ depending only on $f$ and $\Lambda$, such that
\begin{align*}
    \bP(|F_m-\bE F_m| \geq t)
    &\leq 2\exp \left( - c t^{\frac{1}{m}} \min \left( \left( \frac{t}{\gamma^m} \right)^{2-\frac{1}{m}} ,\left( 1 + \log_+ \left( \frac{t}{\gamma^m} \right)\right)^{1-q} \right)  \right), 
    & \qquad t \geq 0, \gamma \geq C^{-1},
\end{align*}
where $\log_+(\cdot) = {\bf 1}(\log(\cdot)\geq 0) \log(\cdot)$ denotes the positive part of the logarithm.
\end{theorem}

\paragraph{Structure of the paper:} The rest of the paper is structured as follows.
In Section \ref{sec:UStat} we present the basic notions related to Poisson $U$-statistics, introduce the assumption \eqref{eq:fBound} of our main theorem, as well as other sets of assumptions and explain how they relate to each other.
In the same section we will present concentration bound for the lower tail and for Poisson $U$-statistic of order $1$ based on known estimates. 
Section \ref{sec:large} is devoted to concentration and anticoncentration bounds under some stronger assumptions.
In Section \ref{sec:ConcentrationViaMoments} we prove the concentration bounds for Poisson $U$-statistics satisfying assumptions \eqref{eq:fBound} by using Markov's inequality and estimates for the centred moments. 
The discussion of the quality of the obtained results and the comparison with the known bounds is postponed to Section \ref{sec:discussion}.
Finally we consider applications to random geometric graphs in Section \ref{sec:geometricGraphs}, and to Poisson hyperplane process in Euclidean and hyperbolic spaces in Section \ref{sec:Applications}.


\section{Poisson \texorpdfstring{$U$}{U}-statistic} \label{sec:UStat}

\subsection{Definition and basic properties}

Let $(\bX,\cX)$ be a measurable space equipped with a measure $\gamma\Lambda$, where $\gamma>0$. 
Let $\bN_0 = \bN \cup \{0\}$ and $\overline{\bN}_0 = \bN \cup \{0,\infty\}$.
We say that a measure is $s$-finite if it can be written as a countable sum of $\sigma$-finite measures.
By $\mathbf{N}(\bX)$ we denote the space of {$s$-finite $\overline{\bN}_0$-valued } 
measures on $\bX$ and $\cN(\bX)$ is defined as the smallest $\sigma$-algebra on $\mathbf{N}(\bX)$ such that the mappings $\xi\mapsto\xi(B)$, $\xi\in\mathbf{N}(\bX)$, $B\in\cX$ are measurable. 
From now on, we assume that $\Lambda$ is a $s$-finite measure.
A \textit{Poisson process} $\eta$ with intensity measure $\lambda=\gamma\Lambda$ is a measurable mapping from some fixed probability space $(\Omega, \cA, \bP)$ to $(\mathbf{N}(\bX),\cN(\bX))$ with the following two properties: for any $B\in\cX$ the random variable $\eta(B)$ is Poisson distributed with mean $\gamma\Lambda(B)$; for any $n\in\bN$ and pairwise disjoint sets $B_1,\ldots,B_n\in\cX$ the random variables $\eta(B_1),\ldots, \eta(B_n)$ are independent. The parameter $\gamma$ is referred to as the \textit{intensity} of the process $\eta$. For more information regarding point processes and their properties we refer reader to \cite{LP,SW}.

A \textit{Poisson functional} is a random variable $F$, which almost surely satisfies $F=g(\eta)$ for some measurable function $g\colon\mathbf{N}(\bX)\to\bR$. The function $g$ is called a representative of $F$. If $\bP_{\eta}$ denotes the distribution of $\eta$ we will write $L^p(\bP_{\eta})$, $p\ge 0$ for the space of Poisson functionals satisfying $\bE|F|^p<\infty$. 
Moreover, as usual, we denote by $L^p(\Lambda^k)$ the space of all measurable functions $f\colon\bX^k\to\bR\cup\{\pm\infty\}$ with
$$
\int_{\bX^k}|f(x_1,\ldots,x_k)|^p\Lambda(\dint x_1)\cdots\Lambda(\dint x_k)<\infty.
$$
The corresponding norm is denoted by $\|\cdot\|_{L^p(\Lambda^k)}$. 
Given a Poisson functional $F$ with representative $g$ and $x\in\bX$ we define the so-called difference or add-one cost operator $D_x$ as follows
$$
D_x F(\eta)=g(\eta+\delta_x)-g(\eta),
$$
where $\delta_x$ denotes the Dirac measure at $x$.

One of the important special cases of Poisson functional is \textit{Poisson $U$-statistic}. Let $m\in\bN$ and $f\colon\bX^m \to\bR$ be a measurable symmetric function satisfying $f\in L^1(\Lambda^m)$. Symmetric means that it is invariant with respect to permutations of the coordinates. The corresponding Poisson $U$-statistic is the following random variable
$$
F_m = F_m(f,\eta) 
= \int f(x_1,\ldots,x_m) \, \eta^{(m)}(d x_1, \ldots , d x_m) ,
$$
where $\eta^{(m)}$ denotes the factorial measure of $\eta$.
The function $f$ is called the \textit{kernel} and $m$ the \textit{order} of $F_m$. In what follows we will always assume that $\Lambda^m(\{x\in\bX^m\colon f(x)\neq 0\})>0$.

The particular interest to Poisson $U$-statistics is motivated by their nice structure. Thus, using Malliavin calculus, one can show that a $U$-statistic of order $m$ consists of a finite sum of Wiener-Ito integrals \cite{RS13}, which allows to derive convenient formulas for the variance and higher centred moments of $F_m(f,\eta)$ in terms of $f$. If $f\in L^1(\Lambda^m)$, using the Slivnyak-Mecke formula \cite[Corollary 3.2.3]{SW} we have
\begin{equation}\label{eq_UstatExp}
\bE F_m(f,\eta)=\gamma^m\int_{\bX^m}f(x_1,\ldots,x_m)\Lambda(\dint x_1)\cdots\Lambda(\dint x_m).
\end{equation}
Further if $f\in L^1(\Lambda^m)$ is a symmetric function, then for $1\leq k\leq m$ we define the function $f_k \colon \bX^k \to \bR$ by 
\begin{equation}\label{eq_fk}
f_k(y_1,\ldots,y_k)\coloneqq  \binom{m}{k}\int_{\bX^{m-k}}f(y_1,\ldots, y_k, z_1,\ldots, z_{m-k}) \Lambda(\dint z_1)\cdots \Lambda(\dint z_{m-k})
\end{equation}
if the integral on the right-hand side is well defined, and by $f_k(y_1,\ldots,y_k)\coloneqq 0$ otherwise. 
For $k=m$, this definition should be understood as $f_m := f$.
It should be noted that, for $f\in L^1(\Lambda^m)$, the above integral is well defined for almost all tuples $(y_1,\ldots,y_k)$, the functions $f_k\colon\bX^k\to\bR$ are symmetric, and $f_k\in L^1(\Lambda^k)$.
Using this notation, the variance of a Poisson $U$-statistic $F_m(f,\eta)\in L^1(\bP_{\eta})$, satisfying $f_k\in L^2(\Lambda^k)$ for all $1\leq k\leq m$, can be written as follows (see \cite[Proposition 12.12]{LP})
\begin{equation}\label{eq_UstatVar}
\V F_m(f,\eta)=\sum_{k=1}^m\gamma^{2m-k}k! \|f_k\|^2_{L^2(\Lambda^k)}.
\end{equation}
Note that the condition $\bE F_m\neq 0$ implies $\|f_1\|_{L^2(\Lambda)}>0$. In particular in case of the non-negative kernel $f$ the condition $\bE F_m\neq 0$ is equivalent to $F_m$ is not almost surely zero. Thus, in case of non-negative kernel $\|f_1\|_{L^2(\Lambda)}=0$ only in the trivial case.

\subsection{Sets of subpartitions and moments of Poisson \texorpdfstring{$U$}{U}-statistic} \label{sec:subpartitions}

In order to present formula \eqref{eq:MomentUStat} below for the $\ell$-th centred moments $\bE[(F_m(f,\eta)-\bE F_m(f,\eta))^\ell]$, $\ell\ge 2$, we need to introduce some additional combinatorial notation and terminology first.

For an integer $n\in\bN$, a \textit{subpartition} $\sigma$ of $[n] \coloneqq \{1,\ldots,n\}$ is a family of disjoint non-empty subsets of $[n]$. The set $\Pi_n^*$ of all subpartitions of $[n]$ is given by
\begin{align*}
    \Pi_n^*
    &\coloneqq  \{ \sigma=\{B_1 , \ldots , B_k\} : k\geq 1 \,;\, \emptyset \neq B_i \subset [n] \text{ and } B_i\cap B_j =\emptyset \text{ for any } i \neq j  \}.
\end{align*}
The subsets $B_i$ are called the \textit{blocks} of the subpartition. The number of blocks in a subpartition $\sigma$ is denoted by $|\sigma|$ and the cardinality of their union $\bigcup_{B\in\sigma} B$ is denoted by $\|\sigma\|$.

We will now define several collections of subpartitions.
A \textit{partition} of $[n]$ is a subpartition $\sigma$ of $[n]$ for which the union of the blocks is the entire set $[n]$.
Thus,
\begin{align*}
    \Pi_n
    &\coloneqq  \{ \sigma : \sigma\in\Pi_n^* \,,\, \|\sigma\|=n \}
\end{align*}
is the set of all partitions of $[n]$.
The collection of subpartitions for which each block has cardinality two or higher is denoted by 
\begin{align*}
    \Pi_{n,\geq 2}^* &\coloneqq  \{ \sigma \in \Pi^*_n : |B| \geq 2 \text{ for any } B\in \sigma  \} .
\end{align*}
Let $\ell,n_1,\ldots,n_\ell \in \bN$ and set $n\coloneqq \sum_{i\in[\ell]} n_i$.
For any $i\in[\ell]$, define 
$$
    R_i 
    \coloneqq  R_i(n_1,\ldots,n_\ell)
    \coloneqq  \{ j\in\bN : n_1 +\cdots + n_{i-1} < j \leq n_1 +\cdots + n_{i} \}.
$$
The letter $R$ stands for ``Row'' and is motivated by the fact that the set $[n]$ can be represented by a diagram consisting of $\ell$ rows of dots, with the $i$-th row having $n_i$ dots and therefore corresponding precisely to the set~$R_i$, see Figure \ref{fig:Diag1}.
We will consider the partition
$$\pi \coloneqq  \pi(n_1,\ldots,n_\ell) \coloneqq  \{R_i : i \in [\ell] \} \in \Pi_n. $$
We define
\begin{equation*}
    \begin{aligned}
        \Pi^*(n_1,\ldots, n_\ell) 
        &\coloneqq  \{ \sigma \in \Pi^*_n : |B\cap R| \leq 1  \text{ for any } B\in \sigma \text{ and } R \in \pi \}, 
        \\
        \Pi^{**}(n_1,\ldots, n_\ell) 
        &\coloneqq  \left\{ \sigma \in \Pi^*_n : 
        \begin{aligned}
            & |B\cap R| \leq 1  \text{ for any } B\in \sigma \text{ and } R \in \pi  , \\ 
            & \text{and for any } R\in\pi \text{ there exists } B\in \sigma \text{ such that } |B| \geq 2 \text{ and } |B\cap R| =1 
        \end{aligned} \right\}.
    \end{aligned}
\end{equation*}
The first collection consist of subpartitions for which any block intersects any row at most once, while the second has the additional constraint that every row is intersected by at least one block of size at least $2$.
We also consider the intersections of these collections with the two previous ones.
\begin{align*}
    \Pi(n_1,\ldots, n_\ell) 
    &\coloneqq  \Pi^\star (n_1,\ldots, n_\ell) \cap \Pi_n
    = \{ \sigma \in \Pi_n : |B\cap R| \leq 1 \text{ for any } B\in \sigma \text{ and } R \in \pi \} ,
    \\
    \Pi_{\geq 2}^*(n_1,\ldots, n_\ell) 
    &\coloneqq  \Pi^*(n_1,\ldots, n_\ell) \cap \Pi_{n,\geq 2}^* 
    = \{ \sigma \in \Pi^*_n : |B| \geq 2 \text{ and } |B\cap R| \leq 1  \text{ for any } B\in \sigma \text{ and } R \in \pi \} .
    \\\Pi_{\geq 2}^{**}(n_1,\ldots,n_\ell)
    &\coloneqq \Pi^{**}(n_1,\ldots, n_\ell) \cap \Pi_{n,\geq 2}^* 
    \\&= \{ \sigma \in \Pi_{\geq 2}^{*}(n_1,\ldots,n_\ell) : \text{for any } R\in\pi \text{ there exists } B\in\sigma \text{ such that }  | B \cap R| = 1 \}.
\end{align*}
Summarizing the above description, we say that for any $n_1,\ldots,n_\ell\in\bN$ with $\sum_{i=1}^{\ell}n_i=n$ the set $\Pi_{\geq 2}^{**}(n_1,\ldots,n_\ell)$ consists of all subpartitions $\sigma \in \Pi^*_n$, such that:
\begin{itemize}
    \item[(i)] for any $B\in\sigma$ we have $|B|\ge 2$;
    \item[(ii)] for any $B\in\sigma$ and any $i\in[\ell]$ we have $|B\cap R_i|\leq 1$;
    \item[(iii)] for any $i\in[\ell]$ there exists $B\in\sigma$ such that $|B\cap R_i|=1$.
\end{itemize}
In case, when $n_1=\ldots=n_{\ell}=m$ we use simplified notation, namely $\Pi(m;\ell)$, $\Pi^*(m;\ell)$, $\Pi_{\geq 2}^{*}(m;\ell)$ and $\Pi_{\geq 2}^{**}(m;\ell)$. In particular, when $n_1=\ldots=n_{\ell}=m$, $n=m\ell$, it is convenient to illustrate the partitions from the set $\Pi_{\geq 2}^{**}(m;\ell)$ with the help of the following diagram (see e.g. Figure \ref{fig:Diag1}), where we represent the set of numbers $[n]$ as a set of nodes in the integer lattice $\bZ^2$. Since any number $k\in [n]$ admits a unique representation of the form $k=(i-1)m+j$, $i\in[\ell]$, $j\in[m]$, we identify the number $k$ with the node with coordinates $(i,j)$. In this representation the condition (ii) can be equivalently reformulated as: each block $B$ of a partition $\sigma$ may intersect each row in maximum one node, and the condition (iii) means, that: each row is intersected by at least one block $B$ of the partition $\sigma$.

\begin{figure}[t]
    \centering
    \begin{subfigure}[b]{0.24\textwidth}
        \centering
        \begin{tikzpicture}
            \foreach \x in {0,...,2}
                \foreach \y in {0,...,3} 
                    {\filldraw (\x,\y) circle (3pt); }
            \draw (0,2) ellipse (10pt and 40pt);
            \draw (1,.5) ellipse (10pt and 25pt);
            \draw[rotate around={-45:(1.5,2.5)}] (1.5,2.5) ellipse (10pt and 30pt);
        \end{tikzpicture}
        \caption{$\sigma\in \Pi_{\geq 2}^{**}(3,3,3,3)$}
    \end{subfigure}
    \hfill
    \begin{subfigure}[b]{0.24\textwidth}
        \centering
        \begin{tikzpicture}
            \foreach \x in {0,...,2}
                \foreach \y in {0,...,3} 
                    {\filldraw (\x,\y) circle (3pt); }
            \draw (0,2) ellipse (10pt and 40pt);
            \draw (1,0.5) ellipse (10pt and 25pt);
            \draw[rotate around={-45:(1.5,2.5)}] (1.5,2.5) ellipse (10pt and 30pt);
            \draw (2,1) ellipse (10pt and 10pt);
            \draw[dashed] (2,1) ellipse (8pt and 8pt);
        \end{tikzpicture}
        \caption{$\sigma \notin \Pi_{\geq 2}^{**}(3,3,3,3)$}
    \end{subfigure}
    \hfill
    \begin{subfigure}[b]{0.24\textwidth}
        \centering
        \begin{tikzpicture}
            \foreach \x in {0,...,2}
                \foreach \y in {0,...,3} 
                    {\filldraw (\x,\y) circle (3pt); }
            \draw (0,2) ellipse (10pt and 40pt);
            \draw (1,.5) ellipse (10pt and 25pt);
            \draw (1.5,2) ellipse (30pt and 10pt);
            \draw[dashed] (1.5,2) ellipse (27pt and 7pt);
        \end{tikzpicture}
        \caption{$\sigma \notin \Pi_{\geq 2}^{**}(3,3,3,3)$}
    \end{subfigure}
    \hfill
    \begin{subfigure}[b]{0.24\textwidth}
        \centering
        \begin{tikzpicture}
            \foreach \x in {0,...,2}
                \foreach \y in {0,...,3} 
                    {\filldraw (\x,\y) circle (3pt); }
            \draw (0,2) ellipse (10pt and 40pt);
            \draw[rotate around={-45:(1.5,2.5)}] (1.5,2.5) ellipse (10pt and 30pt);
        \draw[dotted] (-0.3,-0.4) rectangle (2.3,0.4);
        \end{tikzpicture}
        \caption{$\sigma \notin \Pi_{\geq 2}^{**}(3,3,3,3)$}
    \end{subfigure}
    \caption{Illustration of subpartitions of $\{1,\ldots,12\}$ which belong to $\Pi_{\geq 2}^{**}(3,3,3,3)$ (subfigure (a)) or do not belong to $\Pi_{\geq 2}^{**}(3,3,3,3)$ (subfigures (b)-(d)). Condition (i) is violated by the dashed circle in (b). The dashed ellipse in (c) contradicts condition (ii). Condition (iii) do not hold in (d) because the bottom row do not intersect any block.}
    \label{fig:Diag1}
\end{figure}
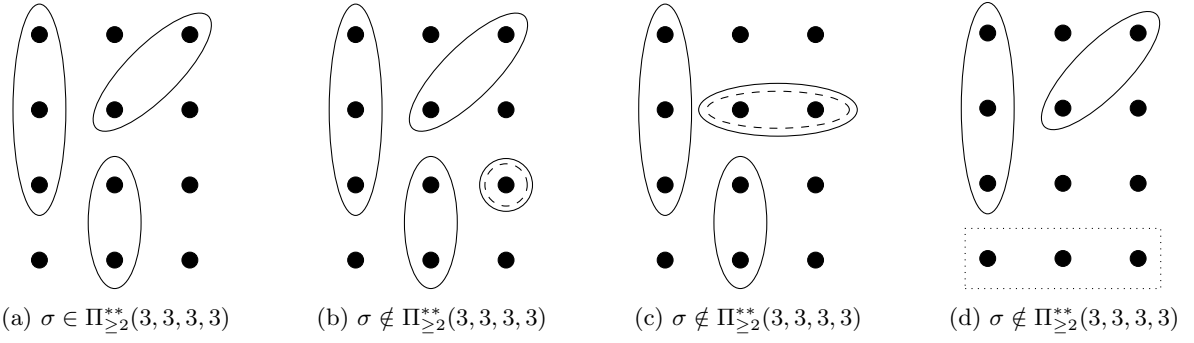

Observe, that one can always complete a subpartition $\sigma$ by adding all the singletons $\{i\} \subset [n]$, which do not belong to any block of $\sigma$. This construction is  the bijection
\begin{align} \label{eq0317}
    \begin{split}
        h\colon\Pi_{n,\geq 2}^* & \to \Pi_n \\
    \sigma & \mapsto \widetilde{\sigma} \coloneqq  \sigma \cup \{\{i\} : i \in [n] 
    \setminus \cup_{B\in\sigma} B \},
    \end{split}
\end{align}
which provides the identifications $\Pi_{n,\geq 2}^* \cong \Pi_n$ and $\Pi_{\geq 2}^*(n_1,\ldots, n_\ell)  \cong \Pi(n_1,\ldots, n_\ell) $.
The inverse map removes all singletons
\begin{align*}
    h^{-1}\colon\Pi_n & \to \Pi_{n,\geq 2}^* \\
    \widetilde{\sigma} & \mapsto \sigma \coloneqq  
    \{ B \in \widetilde{\sigma} : |B|\geq 2 \} .
\end{align*}
Note also, that $|\widetilde{\sigma}| = |\sigma| + n - \|\sigma\| $.
This observation motivates our final definition of a collection of subpartitions:
\begin{align*}
    \Pi_{\geq 2}^{**}(m;\ell, k)
    & \coloneqq  \{\sigma\in \Pi_{\geq 2}^{**}(m;\ell) : m\ell +|\sigma|-\|\sigma\|=|h(\sigma)|=k \} .
\end{align*}

For any measurable real function $f$ defined on $\bX^m$ and for any integer $\ell$, let   $f^{\otimes \ell}$ be the tensor product that maps $(x_1,\ldots,x_{m\ell}) \in \bX^{m\ell}$ to the product $f(x_1\ldots,x_m) \cdots f(x_{(\ell-1)m+1},\ldots,x_{\ell m})$.
For any integer $\ell$ and any $\sigma\in \Pi_{\geq 2}^{**}(m;\ell)$ we consider the real function $(f^{\otimes \ell})_\sigma$ defined on $\bX^{|\sigma|+ m\ell - \|\sigma\|}$ by identifying the variable $x_i$ and $x_j$ if they belong to the same block of $\sigma$.

We are finally ready to formulate the following result \cite[Proposition 12.13]{LP}: for a $U$-statistic $F_m(f,\eta)$ of order $m\ge 1$ with $f\in L^1(\Lambda^m)$ and for any $\ell\ge 2$ such that 
$
 \int (|f|^{\otimes \ell})_{\sigma} \dint \Lambda^{|\sigma|}<\infty
$
holds for all $\sigma\in \Pi(m;\ell)$, the $\ell$-th centred moment of $F_m$ is given by
\begin{align} \label{eq:MomentUStat}
    \bE[(F_m-\bE F_m)^\ell]
    &= \sum_{\sigma\in \Pi_{\geq 2}^{**}(m;\ell)}\gamma^{m\ell +|\sigma|-\|\sigma\|} \int (f^{\otimes \ell})_{\sigma} \dint \Lambda^{m\ell +|\sigma|-\|\sigma\|} ,
\end{align}
where we recall that $\gamma\Lambda$ is the intensity measure of the underlying Poisson process $\eta$. 
Note that the cumulants of $F_m$ can be expressed by a similar formula, where the summation is taken over different sets of subpartitions, see \cite[Theorem 3.7]{ST23}, for example.

\subsection{Assumptions}\label{sec:Assumptions}
Recall that $\eta$ is a Poisson point process on $(\bX,\cX)$ with intensity measure $\gamma\Lambda$, $\gamma>0$ and $F_m(f,\eta)$ is the corresponding Poisson $U$-statistic with kernel $f\colon\bX^m\to\bR$.
In this article we will consider the following assumptions. 

\subsubsection{Statements of the assumptions}

The \textbf{assumptions of the first type} are the following:
We assume that there exist constants $\beta_0\in[1,\infty)$, $\beta_1,\beta_2 \in(0,\infty)$ and $q\in[0,1]$, such that 
\begin{equation}
    \tag{A1} \label{eq:fBound}
    \begin{aligned}
        & f\in L^1(\Lambda^m),\, \int (|f|^{\otimes \ell})_{\sigma} \dint \Lambda^{k}<\infty \text{ and }
        \\&\Big|\int (f^{\otimes \ell})_{\sigma} \dint \Lambda^{k}\Big|
        \leq \beta_0 \beta_1^{k}\beta_2^{\ell} \Big(\prod_{J\in h(\sigma)}|J|!\Big)^q
        \text{ for all $\ell,k\in\bN$, $\ell\ge 2$, $\sigma\in \Pi_{\geq 2}^{**}(m;\ell,k)$},
    \end{aligned}
\end{equation}
where we recall that the map $h$ is defined by \eqref{eq0317}.
Let us point out that in case when $f$ is non-negative the condition $\int (|f|^{\otimes \ell})_{\sigma} \dint \Lambda^{k}<\infty $ can be dropped. Also note that assuming \eqref{eq:fBound} we have 
\begin{align}
\|f_k\|^2_{L^2(\Lambda^k)}&=\int_{\bX^{k}}\left(\binom{m}{k}\int_{\bX^{m-k}}f(y_1,\ldots, y_k, z_1,\ldots, z_{m-k}) \Lambda(\dint z_1)\cdots \Lambda(\dint z_{m-k})\right)^2\Lambda(\dint y_1)\cdots \Lambda(\dint y_{k})\notag\\
&\leq \binom{m}{k}^2\Big|\int (f^{\otimes 2})_{\sigma}\dint \Lambda^{2m-k}\Big|\leq \binom{m}{k}^2\beta_02^{qk}\beta_1^{2m-k}\beta_2^2<\infty,\label{eq:fkA2Bound1}
\end{align}
where $\sigma\in\Pi_{\ge 2}^{**}(m;2,2m-k)$ is taken to be $\sigma=\{B_1,\ldots, B_k\}$ with $B_i=\{\{i\}, \{m+i\}\}$, $1\leq i\leq k$. 
Thus, for any constant $\Cl{c47}>0$, if $\gamma \beta_1\ge \Cr{c47}$, we obtain from \eqref{eq_UstatVar} and \eqref{eq:fkA2Bound1} that
\begin{align}
    \gamma^{2m-1}\|f_1\|^2_{L^2(\Lambda)}\leq \V F_m(f,\eta)
    &\leq 2^q\beta_0\left( \sum_{k=1}^m \big(
    2^q\Cr{c47}^{-1}\big)^{k-1}k! \binom{m}{k}^2 \right) \beta_2^2(\gamma\beta_1)^{2m-1}\notag\\
    &\leq 2^q\beta_0m^2\big(
    2^q\Cr{c47}^{-1}m+1\big)^{m-1} \beta_2^2(\gamma\beta_1)^{2m-1}, \label{eq_UstatVarBounds}
\end{align}
where in the second line we used the inequalities ${m\choose k}\leq m{m-1\choose k-1}$  and $m!/(m-k)!\leq m^k$ for $1\leq k\leq m$.

\medskip

The \textbf{assumptions of the second type} are the following:  We assume that there are constants $\alpha_1,\alpha_2\in(0,\infty)$ such that
\begin{align} \label{assumptions} \tag{A2} 
    \Lambda(\bX)=\alpha_1, 
    \quad \text{and} 
    \quad |f(x_1,\ldots,x_m)| \leq \alpha_2 \text{ for any } (x_1,\ldots,x_m)\in\bX^m.
\end{align}
This assumption means that the expected number of points of $\eta$ is finite 
and the kernel $f$ is a  bounded function.

\medskip
The \textbf{assumptions of the third type} are the following: We assume that there is non-increasing function $g\colon\bR_+\to[0,1]$ and constant $M\in (0,\infty)$ such that
\begin{equation} \label{assumptionsGeneral} \tag{A3}
    \begin{aligned}        
        & \bX \text{ is equipped with a metric $\dist(\cdot,\cdot)$}, \quad f\in L^1(\Lambda^m), \quad f\geq 0\text{ and }\\
        &  f(x_1,\ldots,x_m) \leq M g^{m-1}(\operatorname{diam}(x_1,\ldots,x_m)) ,
        \\&C(g,\Lambda) \coloneqq  \sup_{x \in \bX} \int_{\bX} g(\operatorname{dist}(x,y)) \Lambda(d y) < \infty.
    \end{aligned}
\end{equation}
As it will be shown in Lemma \ref{lem:relAssumptions}, \eqrefb{assumptionsGeneral} generalizes the following type of assumptions, which were introduced by Bachmann and Reitzner in \cite{BR18}.
For given constants $0<\rho\leq \Theta\rho$ and $0<M_1\leq M_2$, they are stated as
\begin{equation}
    \label{assumptionsBR} \tag{$A_{\text{\sc{BR}}}$}
    \begin{aligned}        &\bX=\bR^d,
        \ \Lambda \text{ is locally finite without atoms},\ F<\infty \text{ a.s.}\text{ and }\\
        &\ f(x_1,\ldots,x_m)
        \begin{cases}
            > 0 \text{ if } \operatorname{diam}(x_1,\ldots,x_m) \leq \rho,
            \\= 0 \text{ if } \operatorname{diam}(x_1,\ldots,x_m) > \Theta \rho ,
            \\ \in \{0\} \cup [M_1,M_2] ,
        \end{cases}
    \end{aligned}
\end{equation}
where $\operatorname{diam}(x_1,\ldots,x_m) = \max_{i\neq j}\|x_i-x_j\|_2$.

\subsubsection{Relations between the assumptions}

With the next lemma, we explain how the various assumptions introduced above relate to each other. In particular we see that \eqref{eq:fBound} is more general than each of the other assumptions.
On the other hand, it is often convenient to work under more restrictive assumptions since they are easier to check, and in some cases allow to derive slightly better concentration bounds (compare the constants of Theorem \ref{thm:largeorder2} and Theorem \ref{tm:momentConcRes_NEW}.a).

\begin{lemma}[Relations between the assumptions]\ \label{lem:relAssumptions}
    \begin{enumerate}
        \item $\eqref{assumptions} \Rightarrow \eqref{eq:fBound}$: 
        If \eqrefb{assumptions} is satisfied for some constants $\alpha_1, \alpha_2 \in (0,\infty)$, then \eqrefb{eq:fBound} holds with $\beta_0=1$, $\beta_1 = \alpha_1$, $\beta_2 = \alpha_2$ and $q=0$.
        \item $\eqref{assumptionsBR} \Rightarrow \eqref{assumptionsGeneral}$: 
        If \eqrefb{assumptionsBR} is satisfied for some constants $0<\rho\leq \Theta\rho$ and $0<M_1\leq M_2$, then \eqrefb{assumptionsGeneral} holds with $\bX=\bR^d$, $\dist(x,y)=\|x-y\|_2$, $M=M_2$ and $g(r)={\bf 1}(r\leq \Theta\rho)$.
        \item $\eqref{assumptionsGeneral} \Rightarrow \eqref{eq:fBound}$: 
        If \eqrefb{assumptionsGeneral} is satisfied for some function $g\colon\bR_+ \to [0,1]$ and constants $M$ and $C(g,\Lambda)$, then for any $s\in[0,1]$, \eqrefb{eq:fBound} holds with $q=0$,  $\beta_0=1$,
        $\beta_1 = C(g,\Lambda) \max\left( 1 , \frac{\|f\|_{L^1(\Lambda^m)}}{M C(g,\Lambda)^m} \right)^{\frac{s}{m}}  $ and 
        $\beta_2 = M \max\left( 1 , \frac{\|f\|_{L^1(\Lambda^m)}}{M C(g,\Lambda)^m} \right)^{\frac{1-s}{2}} $.
    \end{enumerate}
\end{lemma}

\begin{proof}[Proof of Lemma \ref{lem:relAssumptions}.1]
    Assuming \eqref{assumptions} we have that $\|f\|_{L^1(\Lambda^m)}\leq \alpha_1^{m}\alpha_2<\infty$ and for any subpartition $\sigma\in \Pi_{\geq 2}^{**}(m;\ell,k)$ we get
        $\Big|\int (f^{\otimes \ell})_{\sigma} \dint \Lambda^{m\ell +|\sigma|-\|\sigma\|}\Big|\leq \int (|f|^{\otimes \ell})_{\sigma} \dint \Lambda^{m\ell +|\sigma|-\|\sigma\|}\leq \alpha_1^k\alpha_2^{\ell}<\infty.$
\end{proof}

\begin{proof}[Proof of Lemma \ref{lem:relAssumptions}.2]
    Assume \eqref{assumptionsBR}. We only need to show that $\bE F < \infty$ and $ \sup_{x \in \bR^d} \Lambda(B_{\Theta\rho}(x)) < \infty$.

    Without loss of generality we assume $\rho=\sqrt{d}$ and $\Theta\rho \in \bN$.
    We can do this reduction by considering $\widetilde{f}(\cdot) = f( \frac{\sqrt{d}}{\rho} \cdot )$, and replacing $\Theta$ by a larger value.
    In particular if $x_1,\ldots,x_m$ are all within the same unit cube $z+[0,1)^d$, for an arbitrary $z\in\bZ^d$, then $\diam(x_1,\ldots,x_m) \leq  \rho$ and thus $f(x_1,\ldots,x_m)\geq M_1>0$.
    
    We set some notation. For $z\in \bZ^d$ we consider
    \begin{align*}
        F_z &\coloneqq  \!\!\!\!\!\!\!\!  \sum_{(x_1,\ldots,x_m)\in\eta^m_{\neq}} \!\!\!\!\!\!\! {\bf 1}\{x_1-z \in [0,1)^d \} f(x_1,\ldots,x_m) ,
        & B_z &\coloneqq  {\bf 1} \{ F_z > 0 \} ,
        & p_z &\coloneqq  \bP(B_z=1) ,
        & \lambda_z &\coloneqq  \Lambda(z + [0,1)^d) .
    \end{align*}
    First, we observe that $F=\sum_{z\in\bZ^d} F_z <\infty$ implies that $\sum_{z\in\bZ^d} B_z<\infty$.
    This follows from the fact that $F$ is a sum of local contributions $F_z$ taking values in $\{0\} \cup [M_1, \infty)$. Second, we observe that for any $z'\in \bZ^d$ the random variables $\{ B_{z} : z\in z' + (2 \Theta \rho+1) \bZ^d \}$ form a collection of independent Bernoulli random variables, from which using Borel-Cantelli lemma we derive the following implications
    \begin{align*}
        \sum_{z\in\bZ^d} B_z<\infty \text{ a.s.} 
        & \Rightarrow  \text{for all } z' \in \bZ^d, \  \sum_{z\in z' + (2 \Theta \rho+1) \bZ^d} p_z <\infty 
        \\& \Rightarrow \sum_{z\in\bZ^d} p_z = \sum_{z'\in \llbracket 0, 2 \Theta \rho\rrbracket^d} \sum_{z\in z' + (2 \Theta \rho+1) \bZ^d} p_z  <\infty ,
    \end{align*} 
    where $\llbracket a , b \rrbracket$ is a short notation for the set $\{ a, a+1 ,\ldots , b \}$ for any $a,b \in \bZ$ with $a<b$.
    
    Third, we observe that if there are $m$ points in $z+[0,1)^d$ then $B_z=1$. 
    Therefore,
    \[ p_z
    \geq \bP( P_{\lambda_z} \geq m )
    \geq \bP( P_{\min(1,\lambda_z)} \geq m )
    \geq \bP( P_{\min(1,\lambda_z)} = m )
    \geq \frac{e^{-1} \min(1 , \lambda_z^m)}{m!}, \]
    where $P_\alpha$ denotes a Poisson random variable with mean $\alpha>0$.
    Hence, $\sum_{z\in\bZ^d} \min(1,\lambda_z^m) <\infty$, which implies that $\lambda_z>1$ for only finitely many $z$, that $\max_{z\in\bZ^d} \lambda_z <\infty$, and that $\sum_{z\in\bZ^d} \lambda_z^m <\infty$. In particular $\sup_{x\in\bR^d} \Lambda(B_{\Theta\rho}(x)) <\infty$ since for any $x\in\bR^d$ the ball $B_{\Theta\rho}(x)$ can be covered by $(2\Theta\rho+1)^d$ cubes of the form $z+[0,1)^d$.
    
    It remains to show that $\bE F <\infty$.
    We start by bounding $\bE F_z$ for any arbitrary $z \in \bZ^d$.
    Note that for any $x \in z+[0,1)^d$ the ball $B_{\Theta\rho}(x)$ is covered by the cubes $z'+[0,1)^d$ with $z'\in z + \llbracket -\Theta\rho -1 , \Theta \rho \rrbracket^d$.
    Therefore, using that $f(x_1,\ldots,x_m)\leq M_2 {\bf 1} (x_2,\ldots,x_m \in B_{\Theta\rho}(x_1))$, we get that 
    \begin{align*}
        \bE F_z
        &\leq M_2 \int_{z+[0,1)^d} \Lambda(B_{\rho\Theta}(x))^{m-1} \Lambda(\dint x)
         \leq M_2 \lambda_z \Bigl(\sum_{\rlap{\scriptsize$z'\in z + \llbracket -\Theta\rho -1 , \Theta \rho \rrbracket^d$}} \lambda_{z'} \Bigr)^{m-1} 
         \leq M_2 (2\Theta\rho+1)^{m-1} \max_{z'\in z + \llbracket -\Theta\rho -1 , \Theta \rho \rrbracket^d} ( \lambda_{z'} )^m ,
    \end{align*}
    where, for the second inequality, after expanding the sum, we simply bounded every summand by their maximum.
    Next, using that the maximum of non-negative numbers is less than their sum, we observe 
    \begin{align*}
        \bE F
        = \sum_{z\in \bZ^d} F_z
        \leq M_2 (2\Theta\rho+1)^{m-1} \sum_{z\in \bZ^d} \sum_{z'\in z + \llbracket -\Theta\rho -1 , \Theta \rho \rrbracket^d} \lambda_{z'}^m 
        = M_2 (2\Theta\rho+1)^{m+d-1} \sum_{z\in \bZ^d} \lambda_{z}^m 
        <\infty ,
    \end{align*}
    which completes the proof.
\end{proof}

\begin{proof}[Proof of Lemma \ref{lem:relAssumptions}.3]
    For a given $\sigma \in \Pi_{\geq 2}^{**}(m;\ell,k)$, let $\widetilde{\sigma} = h(\sigma) = \{ B_1, \ldots , B_k \} $ be defined by \eqref{eq0317}.
    We assume that the blocks are ordered such that $\min(B_1)< \cdots < \min(B_k)$.
    In particular $\min(B_i)=i$ for $i\in[m]$.

    We define the graph $G_\sigma$ whose vertex set is  $ [k] $ and whose edges are the couples $\{i, j\}$ for which the blocks $B_i$ and $B_j$ have one row in common, meaning that there exists $r\in[\ell]$ such that $|B_i \cap R_r| = |B_i \cap R_r| = 1 $.
    We denote the set of connected components of this graph by  $\operatorname{comp}(\sigma)$.
    
    Note that, for any connected component $\mathcal{C}\in \operatorname{comp}(\sigma)$, the union $\cup_{i\in \mathcal{C}} B_i$ consists of $2$ or more rows.
    Thus $|\cup_{i\in \mathcal{C}} B_i| \geq 2m$ and it follows that the number of connected components satisfies $\left|\operatorname{comp}(\sigma) \right| \leq m \ell/(2m) = \ell/2$.
    We also note that any connected component consists of at least $m$ vertices, therefore $m \left|\operatorname{comp}(\sigma) \right| \leq k$ which implies $ \left|\operatorname{comp}(\sigma) \right| \leq k/m $.
    Combining these two bounds, we have 
    $$ \left|\operatorname{comp}(\sigma) \right| 
    \leq k \frac{s}{m} + \ell \frac{1-s}{2} \text{  for any  } s\in[0,1] .$$
    
    We readily have that $f\in L^1(\Lambda^m)$ and $f\geq 0$.
    So it is enough to show that
    \begin{align} \label{eq6224}
        \int (f^{\otimes \ell})_{\sigma} \dint \Lambda^{k}
        \leq \beta_1'^{k} \beta_2'^{\ell} \beta_3^{\left|\operatorname{comp}(\sigma) \right|} 
        \text{ for all $\ell,k\in\bN$, $\ell\geq 2$, $\sigma\in \Pi_{\geq 2}^{**}(m;\ell,k)$, }
    \end{align}
    with
    \[ 
        \beta_1' = C(g,\Lambda) , \qquad 
        \beta_2' = M , \qquad
        \beta_3 = \frac{\|f\|_{L^1(\Lambda^m)}}{M C(g,\Lambda)^m}.
    \]
    We only need to prove \eqref{eq6224} for connected partitions, i.e.\ for $\sigma$ satisfying $\left|\operatorname{comp}(\sigma) \right| = 1$.
    Let $\tau=\tau_\sigma \colon [m\ell] \to [k]$ be the function defined by $\tau(i) = j$ for any $i\in B_j$.
    Because of the ordering of the blocks, we have $\tau(i)=i$ for $i\in[m]$.
    Thus, we can write 
    \begin{align*}
        (f^{\otimes \ell})_{\sigma}(x_1,\ldots,x_k)
        &= \prod_{r=1}^\ell f(x_{\tau((r-1)m+1)} , \ldots , x_{\tau(rm)} ) 
        \\&\leq f(x_1,\ldots,x_m) M^{\ell-1} \prod_{r=2}^\ell g^{m-1}(\diam(x_{\tau((r-1)m+1)} , \ldots , x_{\tau(rm)} )) .
    \end{align*}    
    Let $T$ be a spanning tree of $G_\sigma$, i.e.\ a connected subgraph of $G_{\sigma}$ which does not have cycles.
    We denote its edge set by $E(T)$, and recall that a tree has one edge less than vertices, that is $|E(T)|=k-1$.
    By construction of the initial graph $G_\sigma$, we have that the edge set $E$ can be partitioned in the subsets
    \[ E_r := \left\{ \{ i,j \} \in E(T) : i , j \in \{\tau((r-1)m+1), \ldots , \tau(rm) \} \right\},\quad r\in [\ell] .\]
    We can choose $T$ such that $E_1$ consist of exactly $m-1$ edges, for instance $\{m,i\}$ with $i\in[m-1]$.
    Since $T$ does not have any cycle, $E_r$ consists of at most $m-1$ edges. 
    This implies
    \begin{align*}
        g^{m-1}(\diam(x_{\tau((r-1)m+1)} , \ldots , x_{\tau(rm)} )) 
        \leq \prod_{\{i,j\}\in E_r} g(\dist(x_i,x_j)) ,
    \end{align*}   
    where we also used that $\diam(x_{\tau((i-1)m+1)} , \ldots , x_{\tau(im)} ) \geq \dist(x_i,x_j)$ and that $g$ is decreasing taking value in $[0,1]$.
    We get
    \begin{align*}
        (f^{\otimes \ell})_{\sigma}(x_1,\ldots,x_k)
        &\leq f(x_1,\ldots,x_m) M^{\ell-1} \prod_{\{i,j\}\in E(T)\setminus E_1} g(\dist(x_i,x_j)) ,
    \end{align*}    
    and, therefore,
    \begin{align*}
        \int (f^{\otimes \ell})_{\sigma} \dint \Lambda^{k}
        &\leq M^{\ell-1} \int_{(\bR^d)^k} f(x_1,\ldots,x_m) \prod_{\{i,j\}\in E(T)\setminus E_1} g(\dist(x_i,x_j)) \Lambda^{k}(\dint x_{1},\ldots, \dint x_k)  .
    \end{align*} 
    Now we integrate iteratively.
    We start by a leaf $j\in[k]\setminus[m]$ of the tree.
    Let $i\in [m\ell]$ be the vertex to which it is connected. 
    At this step we need to bound the integral $\int g(\dist(x_i,x_j)) \Lambda(\dint x_j)$ for arbitrary $x_i\in\bR^d$.
    This integral is upper bounded by $C(g,\Lambda)$.
    We remove the leaf $j$ of the tree and iterate this process until the only vertices left are $[m]$.
    This produces the bound
    \begin{align*}
        \int (f^{\otimes \ell})_{\sigma} \dint \Lambda^{k}
        &\leq \|f\|_{L_1(\Lambda^m)} M^{\ell-1} C(g,\lambda)^{(k-1)-(m-1)}
        = \beta_1'^k \beta_2'^\ell \beta_3 ,
    \end{align*} 
    which completes the proof.
\end{proof}


\subsection{First concentration bounds}

As was already mentioned in the introduction there are few concentration inequalities for Poisson $U$-statistics available in the literature. 
In particular,  Bachmann and Peccati obtained in \cite{bachmann2016concentration} a concentration bound for the lower tail for general Poisson $U$-statistic $F_m(f,\eta)$ with non-negative kernel $f$, such that $f_k\in L^1(\Lambda^k)\cap L^2(\Lambda^k)$, $1\leq k\leq m$.
Adapting the above result to our setting we get the following proposition.

\begin{proposition}[Application of Bachmann, Peccati 2016] \label{prop:LowerTail} 
   Assume that $m\ge 1$, $\Lambda$ has no atoms, $f$ is non-negative and $F_m=F_m(f,\eta)$ satisfies \eqref{eq:fBound} with $\gamma \beta_1\ge \Cr{c47} >0$.
   Then for all $t\ge 0$ we have
    \begin{equation} \label{eq1705a}
        \bP(F_m-\bE F_m\leq -t)
        \leq \exp\left(-\frac{t^2}{2m^2V}\right)
        \leq \exp\Big(-{t^2\over \Cr{c181} \beta_0 \beta_2^2(\gamma\beta_1)^{2m-1}}\Big) ,
    \end{equation}
    where $m^2V=\sum_{k=1}^m\gamma^{2m-k}kk!\|f_k\|^2_{L^2(\Lambda^k)}<\infty$ and $\Cl{c181} \coloneqq  2^{q+1}m^2(2^q\Cr{c47}^{-1}m+1)^{m-1}$.
\end{proposition}

\begin{proof}[Proof of Proposition \ref{prop:LowerTail}]
    Since by \eqref{eq:fkA2Bound1} we have $f_k\in L^2(\Lambda^k)$, $1\leq k\leq m$, then by \cite[Proposition 5.7]{bachmann2016concentration} $V<\infty$. Hence, the first bound follows by \cite[Corollary 5.5]{bachmann2016concentration}.

    Further we note that
    \begin{align*}
    m^2V
    &=\sum_{k=1}^m\gamma^{2m-k}kk!\|f_k\|^2_{L^2(\Lambda^k)}
    \leq 2^q\beta_0\beta_2^2(\gamma\beta_1)^{2m-1}\sum_{k=1}^m(\gamma \beta_1)^{-k+1}kk!\binom{m}{k}^22^{q(k-1)},
    \end{align*}
    where we used \eqref{eq:fkA2Bound1} to bound $\|f_k\|^2_{L^2(\Lambda^k)}$.
    Thus, using the fact that $\gamma\beta_1\ge \Cr{c47}$, ${m!\over (m-k)!}\leq m^k$ and $k{m\choose k}\leq m{m-1\choose k-1}$ we get
    \begin{align*}
    2m^2V
    &\leq 2^{q+1}\beta_0\beta_2^2(\gamma\beta_1)^{2m-1}m^2\sum_{k=1}^m\big(2^q\Cr{c47}^{-1}m\big)^{k-1}{m-1\choose k-1}\leq \Cr{c181} \beta_0 \beta_2^2 (\gamma\beta_1)^{2m-1} .
    \end{align*}
\end{proof}

Another easy case is the first order Poisson $U$-statistics $F_1(f,\eta)$. Using the result of Wu \cite{wu2000new} we directly obtain the following bounds.

\begin{proposition}[Application of Wu 2000] 
\label{prop:m1} Assume that $F_1$ is a Poisson $U$-statistic of order $1$, satisfying \eqref{assumptions}. Then for any $t>0$ we get
    \begin{align*}
        \bP(F_1-\bE F_1 \geq t) 
        & \leq
        \exp \left( -\frac{t}{2\alpha_2} \log \left( 1 + \frac{t}{\gamma\alpha_1\alpha_2} \right) \right).
        \end{align*}
        If, moreover, we have $f\ge 0$, then
        \begin{align*}
        \bP(F_1-\bE F_1 \leq -t)
        & \leq \exp \left( - \frac{t^2}{2\gamma\alpha_2^2\alpha_1} \right).
    \end{align*}
\end{proposition}

\begin{remark}
    Note that, for Poisson $U$-statistic $F_1$ satisfying \eqref{assumptions}, from \eqref{eq1705a} we get
    $$
    \bP(F_1-\bE F_1 \leq -t) \leq \exp \left( - \frac{t^2}{2\gamma\|f\|^2_{L^2(\Lambda)}} \right)\leq \exp \left( - \frac{t^2}{2\gamma\alpha_2^2\alpha_1} \right).
    $$
    Thus, Proposition \ref{prop:m1} gives a slightly worse bound for the lower tail than Proposition \ref{prop:LowerTail} in this case.
\end{remark}

\begin{proof}
Observe that if \eqref{assumptions} is satisfied, we have $D_x F_1= f(x) \leq \alpha_2 $ and 
$$ 
\int_{\bX} |D_x F_1|^2 \lambda (\dint y)\leq \int_{\bX} |f(x)|^2 \lambda (\dint y) \leq \gamma\alpha_1\alpha_2^2.
$$
If additionally $f\ge 0$, then $D_x (-F_1) = -f(x) \leq 0 $.
The proposition follows by \cite[Proposition 3.1]{wu2000new}. 
\end{proof}

Note that the results from \cite{wu2000new} are applicable under \eqref{eq:fBound} only with the additional assumption that $f$ is bounded. In contrast to a $U$-statistic of order $1$, which in case of a finite non-atomic measure $\Lambda$ is simply a sum of Poisson number of i.i.d.\ random variables, establishing concentration bounds for Poisson $U$-statistic of order $m\ge 2$ is more complicated task. We recall that our results for this case are summarized in Theorem~\ref{tm:Summary} and that in Section \ref{sec:discussion} we will discuss these results in details.

The next proposition covers the situation when the order of deviation $t$ is at most $\sqrt{\bV F_m}$ and is a direct consequence of Chebyshev-Cantelli's concentration inequality.

\begin{proposition}[Application of Chebyshev-Cantelli's inequality] 
\label{prop:ChebyshevCantelli} 
Assume that $m\ge 1$ and $F_m$ satisfies \eqref{eq:fBound} and $\bV F_m>0$. Then for any constants $\Cr{c47} , \Cl{c43}>0$, we get
    \begin{align*}
        \bP(F_m-\bE F_m \geq t) 
        & \leq\exp\Big(-(1+\Cr{c43}^2)^{-1}{t^2\over \bV F_m}\Big)
        \leq 
        \exp \left( -\frac{t^2}{\Cr{c44}\beta_2^2(\beta_1\gamma)^{2m-1}}\right),
        \ 0\leq t<\Cr{c43}\sqrt{\bV F_m}, \ \gamma \beta_1\ge \Cr{c47},
    \end{align*}
    where
    $$
    \Cl{c44}=2^q\beta_0m^2\big(
    2^q\Cr{c47}^{-1}m+1\big)^{m-1}(1+\Cr{c43}^2).
    $$
\end{proposition}

\begin{remark}
    Note that according to \eqref{eq_UstatVarBounds} we have $\bV F_m\ge \|f_1\|^2_{L^2(\Lambda)}\gamma^{2m-1}$ and if $\|f_1\|^2_{L^2(\Lambda)}>0$ the Proposition \ref{prop:ChebyshevCantelli} holds for $t<\Cl{c45}\gamma^{m-{1\over 2}}$ for some constant $\Cr{c45}= \Cr{c45}(f) >0$.
\end{remark}

\begin{proof}[Proof of Proposition \ref{prop:ChebyshevCantelli}]
By Chebyshev-Cantelli's inequality (see for example \cite[Exercise 2.3]{BLM13}) we get for any $t\ge 0$ that
\begin{align*}
    \bP(F_m-\bE F_m \geq t)\leq {\bV F_m\over \bV F_m+t^2}=\exp\Big(-\log\Big(1+{t^2\over \bV F_m}\Big)\Big).
\end{align*}
Using \eqref{eq_UstatVarBounds} and the inequality $\log(1+x)>{x\over x+1}>(1+\Cr{c43}^2)^{-1}x$, which holds for any $x\leq \Cr{c43}^2$ we get, that for any $t<\Cr{c43}\sqrt{\bV F_m}$ it holds
$$
\log\Big(1+{t^2\over \bV F_m}\Big)\ge (1+\Cr{c43}^2)^{-1}{t^2\over \bV F_m}\ge {t^2\over \Cr{c44}\beta_2^2(\beta_1\gamma)^{2m-1}},
$$
which finishes the proof.
\end{proof}


\section{Concentration inequalities via concentration of Poisson random variables} \label{sec:large}

In this section, we provide upper and lower bounds for the probability $\bP(F_m-\bE F_m\ge t)$ under certain conditions when $t$ is of order at least $\gamma^m$, which in case of non-negative kernel is the order of expectation \eqref{eq_UstatExp}. 
The upper bound requires assumption \eqref{assumptions}, which assumes only that the kernel and the intensity measure are bounded.
In order to get a lower bound of the same order, we will require the assumption that there exist some positive constants $\aaa,\bbb\in(0,\infty)$ and some pairwise disjoint subsets $S_1,\ldots,S_m$ of $\bX$, such that
\begin{equation} \label{assumptionflowbound} 
    f\ge 0,\qquad f\bigr|_{ S_1 \times \cdots \times S_m} \geq \bbb,
    \qquad  \Lambda(S_i) \geq \aaa . \tag{A4}
\end{equation}

As commented in Remark \ref{rem:mildAssumption}, the latter is a particularly mild assumption.

\begin{theorem}[Application of concentration bounds for Poisson random variables]
\label{thm:largeorder2}
    Let $F_m$, $m\ge 1$ be a Poisson $U$-statistic satisfying \eqref{assumptions}, then 
        \begin{align*}
            \bP(|F_m - \bE F_m| \geq t) 
            & \leq 2\exp \left( -\frac{1}{m}\Big({t\over2\alpha_2}\Big)^{\frac{1}{m}}\log \left( \frac{t}{2e^m\alpha_2(\alpha_1\gamma)^m}\right) \right), 
            & \text{for any  $\gamma>0$ and $t\ge 2\alpha_2(\alpha_1\gamma)^m$.}
        \end{align*}
\end{theorem}

\begin{theorem}[Application of anti-concentration bounds for Poisson random variables]
\label{thm:largeorder2_lowerbound}
     Let $F_m$, $m\ge 1$ be a Poisson $U$-statistic satisfying \eqref{assumptionflowbound}, then there exist constants $\Cr{c1905c}$, and $\Cr{c1905d}$, independent of $\gamma$ and $t$, such that 
        \begin{align*}
            \bP(F_m - \bE F_m \geq t)
            &\ge \exp \left( -\Cr{c1905c} t^{\frac{1}{m}}\log \left( \frac{t}{\gamma^m}\right) \right)
            & \text{for any  $\gamma>1$ and $t\ge\Cr{c1905d}\gamma^m$},
        \end{align*}
        Moreover, one can set $\Cl{c1905c} = 2\bbb^{-\frac{1}{m}}$ and we have that $\Cl{c1905d}$  is a (non-explicit) constant which depends only on $\aaa$, $\bbb$ and $\|f\|_{L^1(\Lambda^m)}$.
\end{theorem}

The proof of Theorem \ref{thm:largeorder2} and Theorem \ref{thm:largeorder2_lowerbound} will conclude the current section. 
We start by considering the probability $\bP(F_m\ge t)$, $t>0$. 
Denote $P=\eta(\bX)$, which is a Poisson random variable of mean $\gamma \alpha_1$.
Under \eqref{assumptions} we have
$
     F_m \leq \alpha_2 P^m,
$
which leads to the upper bound
\begin{align} \label{UpperBound0617}
    \bP(F_m \geq t) 
    & \leq \bP(\alpha_2 P^m \geq t) 
    = \bP(P \geq (t/\alpha_2)^{\frac{1}{m}}) 
    \leq \bP(P_{\gamma\alpha_1} \geq (t/\alpha_2)^{\frac{1}{m}}) ,
\end{align}
where $P_{\alpha}$ denotes a Poisson random variable with mean $\alpha>0$.
If one assumes \eqref{assumptionflowbound} and denote $P_i=\eta(S_i)$, $1\leq i\leq m$, which are independent Poisson random variables with means $\gamma\Lambda(S_i)$, we have 
$
\bbb \prod_{i=1}^m P_i \leq F_m,
$
which leads to the lower bound
\begin{align} \label{LowerBound0617}
    \bP(F_m \geq t) 
    & \geq \bP\left( \bbb \prod_{i=1}^m P_i \geq t \right)
   \geq \prod_{i=1}^m \bP\left( P_i \geq (t/\bbb)^{\frac{1}{m}}  \right) 
    \geq \bP\left( P_{\gamma\aaa} \geq (t/\bbb)^{\frac{1}{m}}  \right)^m .
\end{align}
This argument reduces the problem of bounding $\bP(F_m\ge t)$ to bounding the tail of a Poisson random variable.
The next lemma provides such bounds, which we include here for completeness of our arguments. 
Note that stronger bounds (but with a slightly more involved form) can be found in \cite[p. 1225]{Hou02}, for example.
\begin{lemma} \label{lem0204}
    Let $P_{\alpha}$ denote a Poisson random variable with mean $\alpha>0$.
    For any constants $\Cl[assumptions]{C1},\Cl[assumptions]{C2}>0$, there exists a constant $\Cl{c14}>0$, independent of $\alpha$ and $y$, such that 
    \begin{align*}
        \bP(P_{\alpha} \geq y) 
        &\leq \left(e \frac{\alpha}{y}\right)^y,
        && \text{for any  $\alpha>0$ and $y\ge\alpha+1$,}
        \\
        \bP(P_{\alpha} \geq y) 
        &\geq \left(\Cr{c14}\frac{\alpha}{y} \right)^y ,
        && \text{for any  $\alpha>\Cr{C1}$ and $y>\Cr{C2}\alpha$.}
    \end{align*}
    The constant $\Cr{c14}$ depends on $\Cr{C1}$ and $\Cr{C2}$ only.
\end{lemma}
\begin{proof}
    \textit{Upper bound:}
    First, we consider the case when $y\in\bN$. Then using Stirling's approximation for $y!$ (or, more precisely, Robbins estimate) we get
    \begin{align*}
        \bP(P_{\alpha}\geq y)
        &= \sum_{k\geq y} e^{-\alpha} \frac{\alpha^k}{k!}
        \leq e^{-\alpha} \frac{\alpha^y}{y!} \sum_{k\geq y} \frac{\alpha^{k-y}}{(k-y)!} 
        = \frac{\alpha^y}{y!}
        \leq \left(e \frac{\alpha}{y}\right)^y .
    \end{align*}
    When we remove the extra assumption that $y$ is an integer, one has
    \[ \bP(P_{\alpha}\geq y) 
    \leq \bP(P_{\alpha}\geq \lfloor y \rfloor) 
    \leq \left(e \frac{\alpha}{\lfloor y \rfloor}\right)^{\lfloor y \rfloor}
    \leq \left(e \frac{\alpha}{y}\right)^y ,\]
    where the last inequality follows from the fact that $t\mapsto \left(\frac{\alpha e}{t}\right)^t $ decreases for $t\geq \alpha$.
    
    \medskip
    \textit{Lower bound:}
    Observe, that 
    \begin{align*}
        \bP(P_{\alpha}\geq y)
        & \geq \bP(P_{\alpha}= \lceil y\rceil)
        = e^{-\alpha} \frac{\alpha^{\lceil y\rceil}}{\lceil y\rceil!} .
    \end{align*}
    Applying Stirling's approximation to $\lceil y\rceil!$ we get that
    $$
    \lceil y\rceil ! \leq \sqrt{2\pi}\lceil y\rceil^{\lceil y\rceil+{1/2}}e^{-\lceil y\rceil+1}\leq (\Cr{c19} y)^y
    $$ 
    for some large enough constant $\Cl{c19}$. Here, $\Cr{c19}$ depends only on $\Cr{C1}\Cr{C2}$, which is a lower bound for $y$. Getting an explicit representation of this constant seems to be difficult in general and we choose to omit it.
    Note also that $\alpha\leq y/\Cr{C2}$ by assumption.
    Thus, we get
    \begin{align*}
        \bP(P_{\alpha}\geq y)
        & \geq e^{-y/\Cr{C2}} \frac{\alpha^y}{(\Cr{c19} y)^y} \alpha^{\lceil y\rceil -y} 
        = \left( \frac{ \Cr{C1}^{\frac{\lceil y\rceil}{y} -1}}{e^{1/\Cr{C2}}\Cr{c19}} \frac{\alpha}{y} \right)^y ,
    \end{align*}
    and the lemma holds by setting $\Cr{c14} = \frac{1}{e^{1/\Cr{C2}}\Cr{c19}} \min(1,\Cr{C1})$.
\end{proof}

As a direct consequence of the bounds \eqref{UpperBound0617}, \eqref{LowerBound0617} and Lemma \ref{lem0204} we prove the following theorem.

\begin{theorem} \label{thm:largeorder1}
    Let $F_m$, $m\ge 1$ be a Poisson $U$-statistic satisfying \eqref{assumptions}.
    Then 
    \begin{align*}
        \bP(F_m \geq t) 
        & \leq \exp \left( - \frac{1}{m}\Big(\frac{t}{\alpha_2}\Big)^{\frac{1}{m}} \log \left( \frac{t}{e^m\alpha_2(\alpha_1\gamma)^m} \right) \right), 
        && \text{for any  $\gamma>0$ and $t\ge\alpha_2(\alpha_1\gamma)^m$.}
    \end{align*}
   Let $F_m$ be a Poisson $U$-statistic satisfying \eqref{assumptionflowbound}. Then for any constants $\Cl[assumptions]{C3},\Cl[assumptions]{C4}>0$, there exist positive constants $\Cl{c17}$ and $\Cl{c18}$, independent of $\gamma$ and $t$, such that  
    \begin{align*}
        \bP(F_m \geq t) 
        & \geq \exp\left( - \Cr{c17} t^{\frac{1}{m}} \log \left( \Cr{c18} \frac{t}{\gamma^m} \right) \right),
        && \text{for any  $\gamma>\Cr{C3}$ and $t\ge\Cr{C4} \gamma^m$.}
    \end{align*}
    Moreover, one can set $\Cr{c17}= \bbb^{-\frac{1}{m}}$
    and we have that $\Cr{c18}$ is a (non-explicit) constant which depends only on $\Cr{C3}$, $\Cr{C4}$, $\aaa$ and $\bbb$.
\end{theorem}

\begin{proof}
    \textit{Upper bound:} Let $\gamma>0$ and $t\ge\alpha_2(\alpha_1\gamma)^m$.
    Combining the upper bound \eqref{UpperBound0617} with the upper bound of Lemma \ref{lem0204} we get
    \begin{align*}
        \bP(F_m \geq t) 
        & \leq \bP(P_{\gamma\alpha_1} \geq (t/\alpha_2)^{\frac{1}{m}})
        \leq \left(e \frac{\alpha_1\gamma}{(t/\alpha_2)^{\frac{1}{m}}}\right)^{(t/\alpha_2)^{\frac{1}{m}}} =
        \exp \left( - \frac{1}{m}\Big(\frac{t}{\alpha_2}\Big)^{\frac{1}{m}} \log \left( \frac{1}{e^m} \frac{t}{\alpha_2(\alpha_1\gamma)^m} \right) \right). 
    \end{align*}
    
    \medskip
    \textit{Lower bound:} Let $\gamma>\Cr{C3}$ and $t>\Cr{C4} \gamma^m$.
    This time we use the lower bound \eqref{LowerBound0617} and apply the lower bound of Lemma \ref{lem0204} where $\alpha$ is replaced by $\gamma \aaa$, $y=(t/\bbb)^{\frac{1}{m}}$,  $\Cr{C1} = \Cr{C3} \aaa$ and $\Cr{C2} = \aaa^{-1}(\Cr{C4}/\bbb)^{\frac{1}{m}}$. This gives
    \begin{align*}
        \bP(F_m \geq t) 
        & \geq \bP\left( P_{\gamma\aaa} \geq (t/\bbb)^{\frac{1}{m}}  \right)^m 
        \geq \left(\Cr{c14}\frac{\gamma \aaa}{(t/\bbb)^{\frac{1}{m}}} \right)^{m (t/\bbb)^{\frac{1}{m}}}
        = \exp\left( - \Big(\frac{t}{\bbb}\Big)^{\frac{1}{m}} \log \left(  \frac{1}{\Cr{c14}^m} \frac{t}{\bbb(\aaa\gamma)^m} \right) \right),
    \end{align*}
    where $\Cr{c14}$ depends only on $\Cr{C1}$ and $\Cr{C2}$, and therefore only on  $\Cr{C3}$, $\Cr{C4}$, $\aaa$ and $\bbb$.
\end{proof}

Now, we can derive the bounds of Theorem \ref{thm:largeorder2} as a simple corollary of the previous theorem; i.e.\ we bound $\bP(F_m - \bE F_m \geq t)$ instead of $\bP(F_m \geq t)$.

\begin{proof}[Proof of Theorem \ref{thm:largeorder2}]
    For the upper bound under condition \eqref{assumptions} we have
    $$
    |\bE F_m|=\gamma^m\Big|\int_{\bX}f(x_1,\ldots,x_m)\Lambda(\dint x_1)\cdots\Lambda(\dint x_m)\Big|\leq \gamma^m\int_{\bX}|f(x_1,\ldots,x_m)|\Lambda(\dint x_1)\cdots\Lambda(\dint x_m)\leq \alpha_2(\alpha_1\gamma)^m.
    $$
    Then by Theorem \ref{thm:largeorder1} for any $\gamma>0$ and  $t\ge 2\alpha_2(\alpha_1\gamma)^m$ we get
    \begin{align}
        \bP(F_m - \bE F_m \geq t) 
        &\leq \bP(F_m \geq t-|\bE F_m|)\leq \bP(F_m \geq t/2)\leq \exp \left( - \frac{1}{m}\Big(\frac{t}{2\alpha_2}\Big)^{\frac{1}{m}} \log \left( \frac{t}{2e^m\alpha_2(\alpha_1\gamma)^m} \right) \right).\label{eq:ConcentrationBoundA1}
    \end{align} 
    Finally we note that
    \[
    \bP(F_m-\bE F_m\leq -t)=\bP((-F_m)-\bE[(-F_m)]\ge t),
    \]
    and $-F_m$ satisfies \eqref{assumptions} with the same $\alpha_1$ and $\alpha_2$. Thus, \eqref{eq:ConcentrationBoundA1} holds for $-F_m$ as well, which finishes the proof.
    \end{proof}

    \begin{proof}[Proof of Theorem \ref{thm:largeorder2_lowerbound}]
    Recall, that we assume that $f\ge 0$ and according to \eqref{eq_UstatExp} we have $\bE F_m = \gamma^m \norm{f}_{L^1(\Lambda^m)}>0 $. Applying Theorem \ref{thm:largeorder1} with $\Cr{C3} = 1$, $\Cr{C4}=\|f\|_{L^1(\Lambda^m)}$ we have
    \begin{align*}
        \bP(F_m - \bE F_m \geq t) 
        & = \bP(F_m \geq t + \bE F_m) 
        \\&\geq \exp\left( - \Cr{c17} (t + \bE F_m)^{\frac{1}{m}} \log \left( \Cr{c18} \frac{t + \bE F_m}{\gamma^m} \right) \right),
    \end{align*}
    for any  $\gamma> 1 $ and $ t + \bE F_m\ge \|f\|_{L^1(\Lambda^m)}\gamma^m$, where we recall that
    \( \Cr{c17}
       = {\bbb^{-\frac{1}{m}}}\)
    and $\Cr{c18}$ is a (non-explicit) constant depending on $\aaa$, $\bbb$ and $\norm{f}_{L^1(\Lambda^m)}$. Further, it follows that there exists a (non-explicit) constant $\Cr{c1905d}$, depending on $\aaa$, $\bbb$ and $\norm{f}_{L^1(\Lambda^m)}$, which is sufficiently large so that
    \begin{align*}
        (t + \bE F_m)^{\frac{1}{m}} \log \left( \Cr{c18} \frac{ t + \bE F_m}{\gamma^m} \right)
        \leq 2 t^{\frac{1}{m}} \log \left( \frac{t}{\gamma^m} \right) ,
        && \text{for any $ t \ge \Cr{c1905d} \gamma^m$.}
    \end{align*}
    We also assume that $ \Cr{c1905d} $ is bigger than $\|f\|_{L^1(\Lambda^m)}$ so that the condition $ t \ge \Cr{c1905d} \gamma^m$ is more restrictive than $ t + \bE F_m\ge \|f\|_{L^1(\Lambda^m)}\gamma^m$.
    Combining the two last displayed equations, we get
    \begin{align*}
        \bP(F_m - \bE F_m \geq t) 
        &\geq \exp\left( - 2 \Cr{c17} t^{\frac{1}{m}} \log \left( \frac{ t}{\gamma^m} \right) \right),
        && \text{for any  $\gamma> 1 $ and $ t \ge \Cr{c1905d} \gamma^m$.}
    \end{align*}
\end{proof}


\section{Concentration inequalities via bounds on the centred moments}\label{sec:ConcentrationViaMoments}

In this section, we establish bounds for the $\ell$-th centred moments of a Poisson $U$-statistic $F_m$, and use them together with Markov's inequality to get concentration bounds.
The key ingredient of this approach is an existing formula for the $\ell$-th centred moments of a Poisson $U$-statistic.

\subsection{Bound on the centred moments}

Recall that by \eqref{eq:MomentUStat}, given $\ell\ge 2$, if $\int (|f|^{\otimes \ell})_{\sigma} \dint \Lambda^{|\sigma|}<\infty$
holds for all $\sigma\in \Pi(m;\ell)$, then 
\[
\bE[(F_m-\bE F_m)^\ell]
= \sum \gamma^{k(\sigma)}\int (f^{\otimes \ell})_{\sigma} \dint \Lambda^{k(\sigma)},
\]
where the sum runs over subpartitions $\sigma\in \Pi_{\geq 2}^{**}(m;\ell)$, and where $k(\sigma) := m\ell +|\sigma|-\|\sigma\|$.
In particular if $f\ge 0$, then the $\ell$-th centred moment of $F_m$ is non-negative and increasing with respect to its kernel.

Note that for any $\sigma\in \Pi(m;\ell)$ we have $h^{-1}(\sigma)\in \Pi^*_{\ge 2}(m;\ell)$, where we recall that $h$ is defined by \eqref{eq0317}. Let $J:=\{i\in [\ell]\colon \text{ there exists }B\in h^{-1}(\sigma)\text{ such that }|B\cap R_i|=1\}$ and $j=|J|$. Thus, by applying corresponding renumbering of elements and rows we may consider $h^{-1}(\sigma)$ as a partition $\widetilde \sigma\in \Pi_{\ge 2}^{**}(m;j)$. 
 Then under assumption \eqref{eq:fBound} we get
\begin{align*}
    \int (|f|^{\otimes \ell})_{\sigma} \dint \Lambda^{|\sigma|}=(\|f\|_{L^1(\Lambda^m)})^{\ell-j}\int (|f|^{\otimes \ell})_{\widetilde\sigma} \dint \Lambda^{k(\widetilde\sigma)}<\infty.
\end{align*}
Moreover, under \eqref{eq:fBound} we have 
$$ 
\gamma^{k(\sigma)}\Big|\int (f^{\otimes \ell})_{\sigma} \dint \Lambda^{k(\sigma)}\Big|\leq \beta_0 \beta_2^{\ell} (\gamma\beta_1)^{k(\sigma)}\sum_{\sigma\in \Pi_{\geq 2}^{**}(m;\ell,k)}\Big(\prod_{J\in h(\sigma)}|J|!\Big)^q,
$$
and, thus,
\begin{align*}
    \bE[(F_m-\bE F_m)^\ell] 
    & \leq \sum_{\sigma\in \Pi_{\geq 2}^{**}(m;\ell)}\gamma^{k(\sigma)}\Big| \int (f^{\otimes \ell})_{\sigma} \dint \Lambda^{k(\sigma)}\Big|\leq \beta_0 \beta_2^{\ell} \sum_{k}(\gamma\beta_1)^k\sum_{\sigma\in \Pi_{\geq 2}^{**}(m;\ell,k)}\Big(\prod_{J\in h(\sigma)}|J|!\Big)^q.
\end{align*}

Further we introduce the notation
\begin{align*}
    S_{\geq 2}^m(\ell,k)
    & \coloneqq  |\Pi_{\geq 2}^{**}(m;\ell, k)|.
\end{align*}
Note that $\{h(\sigma)\colon \sigma\in\Pi_{\geq 2}^{**}(m;\ell,k)\}$ is the set of partitions of $[m\ell]$ consisting of $k$ sets and which satisfy additional constraints. In particular it is a subset of the set $\Pi_{m\ell}(k)$ of all partitions of $[m\ell]$ into $k$ blocks. The cardinality of $\Pi_{m\ell}(k)$ is the Stirling number of the second kind $S(m \ell,k)$.
Also note, that if $k>m\ell - \ell/2$ or $k<m$ then $\Pi_{\geq 2}^{**}(m;\ell,k)$ is an empty set.
Hence, we get 
\begin{equation}\label{eq:New2}
    \bE[(F_m-\bE F_m)^\ell]
    \leq \beta_0 \beta_2^{\ell} \sum_{k=m}^{\lfloor m\ell - \ell/2\rfloor} (\gamma\beta_1)^k\sum_{\sigma'\in \Pi_{m\ell}(k)}\Big(\prod_{J\in \sigma'}|J|!\Big)^q.
\end{equation}

Before we continue let us formulate and prove the following lemma, providing an information about the properties of the term $\prod_{J\in \sigma'}|J|!$. The proof is inspired by the proof of Lemma 3.5 in \cite{eichelsbacher2015Moderate}.

\begin{lemma}\label{lem:FaadiBrino}
\begin{itemize}
\item [(a)] For any $1\leq k\leq m\ell$ we have
\[
\sum_{\sigma\in \Pi_{m\ell}(k)}\prod_{J\in\sigma}|J|!={(m\ell-1)!\over (k-1)!}{m\ell\choose k}.
\] 
\item[(b)] For any $\sigma\in \Pi_{\geq 2}^{**}(m;\ell,k)$ we have
\[
\Big(\prod_{J\in h(\sigma)}|J|!\Big)^q\leq \ell^{qm\ell}.
\]
\end{itemize}
\end{lemma}

\begin{proof}
\textit{Proof of (a):} We use the multivariate Fa\'a di Bruno formula. Namely let $f\colon\bR\to \bR$ and $g\colon\bR^n\to \bR$ be some functions. Then the following equality
\begin{equation}\label{eq:FaadiBruno}
{\partial^n\over \partial x_1\ldots\partial x_n}f(g(x_1,\ldots,x_n))=\sum_{\sigma\in \Pi_n}f^{(|\sigma|)}(g(x_1,\ldots,x_n))\prod_{J\in\sigma}{\partial^{|J|}\over \prod_{j\in J}\partial x_j}g(x_1,\ldots,x_n),
\end{equation}
holds regardless of whether $x_1,\ldots,x_n$ are distinct or not. We choose $n=m\ell$, $f(y)=(y-1)^k$ and $g(x)=g(x,\ldots,x)=(1-x)^{-1}$. Then note that, for any $k'\in\bN$,
\[
f^{(k')}(y)|_{y=1}=
\begin{cases}
k!,\qquad &k'=k,\\
0,\qquad &k'\neq k.
\end{cases}
\]
Moreover $g^{(j)}(x)|_{x=0}=j!$ for any $j\in\bN$. Thus, for $x_1=\cdots=x_n=x=0$, the right hand side of \eqref{eq:FaadiBruno} becomes
\[
\sum_{\sigma\in \Pi_n}f^{(|\sigma|)}(y)|_{y=1}\prod_{J\in\sigma}g^{|J|}(x)|_{x=0}=k!\sum_{\sigma\in \Pi_n(k)}\prod_{J\in\sigma}|J|!.
\]
On the other hand, $f(g(x)) = ((1-x)^{-1}-1)^k = x^k (1-x)^{-k} $, and thus the left hand side of \eqref{eq:FaadiBruno} is equal to
\begin{align*}
(x^k(1-x)^{-k})^{(m\ell)}|_{x=0}={m\ell\choose k}k![(1-x)^{-k}]^{(m\ell-k)}|_{x=0}=k!{(m\ell-1)!\over (k-1)!}{m\ell\choose k}.
\end{align*}
This finishes the proof.

\textit{Proof of (b):} First we note that since $\sigma\in \Pi_{\geq 2}^{**}(m;\ell,k)$ each block of $\sigma$ contains not more then $\ell$ elements. At the same time blocks of $h(\sigma)$ are the blocks of $\sigma$ with additional singletons. Hence, each block $J\in h(\sigma)$ satisfies $|J|\leq \ell$. Let $h(\sigma)=\{J_1,\ldots, J_k\}$. Next we note that if there are blocks $J_{i_1}, J_{i_2}\in h(\sigma)$, $1\leq i_1,i_2\leq k$, $i_1\neq i_2$ with $1<|J_{i_2}|\leq |J_{i_1}|<\ell$, then 
\begin{align*}
    \prod_{i=1}^{k}(|J_i|)!=\frac{|J_{i_2}|-1}{|J_{i_1}|+1}(|J_{i_1}|+1)!(|J_{i_2}|-1)!\prod_{i\neq i_1,i_2}(|J_i|)!\leq (|J_{i_1}|+1)!(|J_{i_2}|-1)!\prod_{i\neq i_1,i_2}(|J_i|)!.
\end{align*}
This means that the maximal value of $\prod_{J\in h(\sigma)}|J|!$ is achieved for the permutation $h(\sigma)$ consisting of $k_1$ blocks of size $\ell$, $k_2$ blocks of size $1$ and remaining $1$ block of size $m\ell-\ell k_1-k_2$ if $m\ell-\ell k_1-k_2\neq 0$, where
\[
k_1+k_2+1=k,\qquad (m-1)\ell+1\leq k_1\ell+k_2\leq  m\ell-1.
\]
In case $m\ell-\ell k_1-k_2=0$ we have instead $k_1+k_2=k$.
In both cases we have $m\ell-\ell k_1-k_2\leq \ell$ and $k_2\ge 0$ and, thus, we get
\[
\Big(\prod_{J\in h(\sigma)}|J|!\Big)^q\leq (\ell!)^{qk_1}((m\ell-\ell k_1-k_2)!)^q\leq \ell^{qm\ell-qk_2}\leq \ell^{qm\ell},
\]
which finishes the proof.
\end{proof}

Assuming \eqref{eq:fBound}, using Lemma \ref{lem:FaadiBrino}.b and \eqref{eq:New2} we may write
\begin{equation}\label{eq:New3}
    \bE[(F_m-\bE F_m)^\ell]\leq \beta_0 \beta_2^{\ell}\ell^{qm\ell} \sum_{k=m}^{\lfloor m\ell - \ell/2\rfloor} S_{\ge 2}^m(\ell,k)(\gamma\beta_1)^k
    \leq \beta_0 \beta_2^{\ell}\ell^{qm\ell}  \sum_{k=1}^{m\ell} S(m\ell,k)(\gamma\beta_1)^k.
\end{equation}
At the same time by \cite[Equation 1.3-14]{Haight}, we have that $ \bE[P_\alpha^n] = \sum_{k=1}^n S(n,k) \alpha^k $, for any $\alpha>0$ and $n\in\bN$, where $P_\alpha$ denotes a Poisson random variable with mean $\alpha$.
Hence we have that the right hand side of the last displayed equation is bounded by $\beta_0 (\beta_2\ell^{qm})^{\ell} \bE (P_{\gamma\beta_1})^{m\ell}$.
In the following lemma, we group these observations together with an application to the more restrictive assumption \eqref{assumptions}.
\begin{lemma} \label{lem:mommon}
    Let $\ell\in\bN$ and recall that $P_{\alpha}$ denotes a Poisson random variable with mean $\alpha>0$. Then
    \begin{enumerate}
        \item The $\ell$-th centred moment of a Poisson $U$-statistic with non-negative kernel is increasing with respect to its kernel. 
        \item Assuming \eqref{assumptions}, we have
        \begin{equation} \label{eq_momentBound}
            \bE[(F_m-\bE F_m)^\ell]
            \leq \alpha_2^{\ell}\sum_{k} S_{\geq 2}^m(\ell,k) (\gamma\alpha_1)^k
            = \alpha_2^{\ell}\bE[((P_{\gamma\alpha_1})_m-\bE ((P_{\gamma\alpha_1})_m))^\ell] ,
        \end{equation}
        where 
        $(x)_m = x  (x-1) \cdots (x-m+1)$ denotes the falling factorial, for any $x\in \bR$ and $m\in \bN$.
        \item Assuming \eqref{eq:fBound}, we have
        \begin{align} \label{eq_momentBoundGeneral}
            \bE[(F_m-\bE F_m)^\ell]
            &\leq \beta_0 (\beta_2\ell^{qm})^{\ell} \sum_{k=m}^{\lfloor m\ell - \ell/2\rfloor} S(m\ell,k)(\gamma\beta_1)^k
            \\&\leq \beta_0 (\beta_2\ell^{qm})^{\ell} \bE (P_{\gamma\beta_1})^{m\ell}.\label{eq_momentBoundGeneral2}
        \end{align}
    \end{enumerate}    
\end{lemma}

\begin{proof} 
    Points 1 and 3 are established right above the lemma.   
    The inequality in \eqref{eq_momentBound} follow from \eqref{eq:New3} and Lemma \ref{lem:relAssumptions}, and 
    the equality in \eqref{eq_momentBound} follows from the observation that, when the kernel is constant equal to $\alpha_2$, $F_m$ is equal (in distribution) to $\alpha_2(P_{\gamma\alpha_1})_m$, and $ \bE[(F_m-\bE F_m)^\ell]
    = \sum \gamma^{k(\sigma)}\int (f^{\otimes \ell})_{\sigma} \dint \Lambda^{k(\sigma)}$, with $f \equiv \alpha_2$.
\end{proof}
Note that the bounds \eqref{eq_momentBoundGeneral} and \eqref{eq_momentBoundGeneral2} are expressed in terms of Stirling numbers and moments of Poisson random variables.
By utilizing known bounds on these quantities, we will derive more explicit centred moment bounds in Theorem \ref{thm:mombound} below.
If one would establish a good bound the coefficient $S_{\geq 2}^m(\ell,k)$, then applying it to equation \eqref{eq:New3} could lead to even better bounds than the ones derived in Theorem \ref{thm:mombound} below.
In the next remark, we gather some elementary observations about these numbers.
\begin{remark}
    To the best of the authors' knowledge, the numbers $S_{\geq 2}^m(\ell,k)$ have not been studied when $m\geq 2$.
    However, when $m=1$ these numbers are known in the literature as $2$-associated Stirling numbers of the second kind and denoted $S_2(\ell,k) \coloneqq  S_{\geq 2}^1 (\ell,k)$.
    More generally the $r$-associated Stirling numbers of the second kind $S_r(\ell,k)$ is the number of possible partitions of an $\ell$ elements set into $k$ subsets of at least $r$ elements each.
    For $r\leq 3$, these numbers are listed by the sequences \href{https://oeis.org/A008277}{A008277} ($r=1$, corresponding to the classical Stirling numbers of the second kind),  \href{https://oeis.org/A008299}{A008299} ($r=2$) and \href{https://oeis.org/A059022}{A059022} ($r=3$) on the On-Line Encyclopedia of Integer Sequences (OEIS) \cite{oeis}, where the curious reader can find a few references about these numbers.
\end{remark}

\begin{theorem}[Centred moment bounds] \label{thm:mombound}
    Assume that $F_m$ is a Poisson $U$-statistic satisfying \eqref{eq:fBound}. Then, 
    \begin{align*}
        \bE[(F_m-\bE F_m)^\ell]
        &\leq \beta_0\left({2^{qm}(m\ell)^m\beta_2\max\big(1,{\gamma\beta_1\over m\ell}\big)^{(m-{1\over 2})q} \over \log^{m(1-q)}\big(1+{m\ell\over\gamma\beta_1}\big)}\right)^{\ell}
        \quad \text{ for all $\ell
        \ge 2$, $\gamma\beta_1>0$}, \\
        \bE[(F_m-\bE F_m)^\ell]
        &\leq \beta_0 \Big(2^{2m+1}m\ell\beta_2^2(\gamma\beta_1)^{2m-1}\Big)^{\ell/2} 
        \quad \text{ for all $\ell
        \ge 2$ when $\gamma\beta_1 \geq 2 m \ell$.}
    \end{align*}
\end{theorem}

The proof of Theorem \ref{thm:mombound} can be found after the next corollary, in which we simplify slightly the bounds of the last theorem by getting rid of the prefactor $\beta_0$.
We do this at the cost of increasing slightly the constant exponentiated to the $\ell$ and assuming $\ell$ to be large enough.
\begin{corollary}
\label{cor:mombound}
    Assume that $F_m$ is a Poisson $U$-statistic satisfying \eqref{eq:fBound}. Then, 
    \begin{align*}
        \bE[(F_m-\bE F_m)^\ell]
        &\leq \left({2^{qm}e(m\ell)^m\beta_2\max\big(1,{\gamma\beta_1\over m\ell}\big)^{(m-{1\over 2})q} \over \log^{m(1-q)}\big(1+{m\ell\over\gamma\beta_1}\big)}\right)^{\ell}
        \quad \text{ for all $\gamma\beta_1>0$ and $\ell \geq \max(\log (\beta_0),2)$},\\
        \bE[(F_m-\bE F_m)^\ell]
        &\leq \left(2^{2m+1}e^2m\ell\beta_2^2(\gamma\beta_1)^{2m-1}\right)^{\ell/2} 
        \quad \text{ when $\gamma\beta_1 \geq 2 m \ell$ and $\ell\geq \max(\log(\beta_0),2)$.}
    \end{align*}    
\end{corollary}
\begin{proof}[Proof of Corollary \ref{cor:mombound}]      
    This follows from the observation that $\beta_0 \leq (e^{1/k})^{\ell}$ for $\ell\geq k \log \beta_0$, $k\ge 0$.
\end{proof}

\begin{proof}[Proof of Theorem \ref{thm:mombound}]

    We start by proving the first inequality. By \eqref{eq:New2} we have 
    \[
    \bE[(F_m-\bE F_m)^\ell]
    \leq \beta_0 \beta_2^{\ell} \sum_{k=m}^{\lfloor m\ell - \ell/2\rfloor} (\gamma\beta_1)^k\sum_{\sigma'\in \Pi_{m\ell}(k)}\Big(\prod_{J\in \sigma'}|J|!\Big)^q.
    \]
    Further we note that $x^q$, $x\ge 0$, $q\in [0,1]$ is concave function. Hence, by applying Jensen's inequality and recalling that $|\Pi_{m\ell}(k)|=S(m\ell,k)$ we get
    \[
    \sum_{\sigma'\in \Pi_{m\ell}(k)}\Big(\prod_{J\in \sigma'}|J|!\Big)^q\leq S(m\ell,k)\Big({1\over S(m\ell,k)}\sum_{\sigma'\in \Pi_{m\ell}(k)}\prod_{J\in \sigma'}|J|!\Big)^q=S(m\ell,k)^{1-q}\Big({m\ell\choose k}{(m\ell-1)!\over (k-1)!}\Big)^q,
    \]
    where the second equality follows from Lemma \ref{lem:FaadiBrino}.a. Combining this together with \eqref{eq:New2} and using H\"olders inequality with $p_1={1\over 1-q}$ and $p_2={1\over q}$, $q\in (0,1)$ we obtain
    \begin{align*}
    \bE[(F_m-\bE F_m)^\ell]
    &\leq \beta_0 \beta_2^{\ell} \sum_{k=m}^{\lfloor m\ell - \ell/2\rfloor} (\gamma\beta_1)^{k(1-q+q)}S(m\ell,k)^{1-q}\Big({m\ell\choose k}{(m\ell-1)!\over (k-1)!}\Big)^q\\
    &\leq \beta_0 \beta_2^{\ell}
    \left(\sum_{k=1}^{m\ell}(\gamma\beta_1)^kS(m\ell,k)\right)^{1-q}\left(\sum_{k=m}^{\lfloor m\ell - \ell/2\rfloor}(\gamma\beta_1)^k{m\ell\choose k}{(m\ell-1)!\over (k-1)!}\right)^q.
    \end{align*}
    Also note that for $q=0$ and $q=1$ this inequality trivially holds.
    Finally according to  \cite[Theorem 1]{Ahle22}, we have that $\bE P_{\alpha}^n \leq ( n / \log(1+n/\alpha) )^n$ for $\alpha>0$ and $n\in\bN$ and, hence,
    \[
    \left(\sum_{k=1}^{m\ell}(\gamma\beta_1)^kS(m\ell,k)\right)^{1-q}=\Big(\bE\big[ (P_{\gamma\beta_1})^{m\ell}\big]\Big)^{1-q}\leq \left({(m\ell)^{m(1-q)}\over\log^{m(1-q)}\big(1+{m\ell\over \gamma\beta_1}\big)}\right)^\ell.
    \]
    Moreover since ${(m\ell-1)!\over (k-1)!}\leq (m\ell)^{m\ell-k}$ 
    we get
    \[
    \left(\sum_{k=m}^{\lfloor m\ell - \ell/2\rfloor}(\gamma\beta_1)^k{m\ell\choose k}{(m\ell-1)!\over (k-1)!}\right)^q\leq (m\ell)^{qm\ell}\left(\sum_{k=m}^{\lfloor m\ell - \ell/2\rfloor}\left({\gamma\beta_1\over m\ell}\right)^k{m\ell\choose k}\right)^q \leq (2m\ell)^{qm\ell}\max\Big(1,{\gamma\beta_1\over m\ell}\Big)^{q(m\ell-\ell/2)}.
    \]
    Combining these estimates together finishes the proof.
    
    In order to prove the second inequality, we combine \eqref{eq:New2} with Lemma \ref{lem:FaadiBrino} and obtain 
\[
    \bE[(F_m-\bE F_m)^\ell]
    \leq \beta_0 \beta_2^{\ell} \sum_{k=m}^{\lfloor m\ell - \ell/2\rfloor} (\gamma\beta_1)^k{(m\ell-1)!\over (k-1)!}{m\ell\choose k},
    \]
     since $\prod_{J\in \sigma'}|J|!\ge 1$.  Further note that ${m\ell\choose k}\leq 2^{m\ell}$ and ${(m\ell-1)!\over (k-1)!}\leq (m\ell)^{m\ell-k}$. Thus for any $\gamma\beta_1\ge 2m\ell$, we get
    \[
    \bE[(F_m-\bE F_m)^\ell]
    \leq \beta_0 (2^{2m}m\ell\beta_2^2(\gamma\beta_1)^{2m-1})^{\ell/2} \sum_{k=m}^{\lfloor m\ell - \ell/2\rfloor} \Big({m\ell\over \gamma\beta_1}\Big)^{m\ell-k-\ell/2}\leq 2\beta_0 (2^{2m}m\ell\beta_2^2(\gamma\beta_1)^{2m-1})^{\ell/2},
    \]
which finishes the proof since $2\leq 2^{\ell/2}$.
  
\end{proof}

\subsection{Concentration bounds}

In this section we prove the following concentration bounds.
In the results below we did not aim for optimal constants in order to get simpler proofs.

\begin{theorem}[Application of the bounds on centred moments] \label{tm:momentConcRes_NEW}
    Assume that $F_m$, $m\ge 1$ is a Poisson $U$-statistic satisfying \eqref{eq:fBound}.
    \begin{itemize}
    \item[(a)] Let $\Cl{c151} \coloneqq  \left( \lceil \max(2e^{2} , \lceil\log(\beta_0)\rceil )\times 2^qe^{1+{1\over m}} m\rceil  \right)^m$.
    If $\gamma\beta_1\geq 1$, we have 
    \begin{equation} \label{eq:MomentBoundExponentialTail}
        \bP(|F_m-\bE F_m|\geq t)
         \leq \exp\Big(- \frac{1}{2^{1+q}em} \Big({t\over e\beta_2}\Big)^{1\over m} \log^{1-q}\Big({1\over 2^{2mq}e^{m+1}}{t\over \beta_2(\gamma\beta_1)^m}\Big) \Big) ,
    \end{equation} 
    for any $t\ge \Cr{c151} \beta_2 (\gamma\beta_1)^{m}$.
    \item[(b)] Let $\Cl{c162} \coloneqq  \max(1, \lceil \log(\beta_0) \rceil ) $.
    If $\gamma\beta_1\geq 8m\Cr{c162}$,
    we have
    \begin{align} \label{eq:MomentBoundSubExponentialTail}
        \bP(|F_m-\bE F_m|\geq t)
        \leq \exp\Big(-\frac{t^2}{2^{2m+8}m \beta_2^2 (\gamma\beta_1)^{2m-1}}\Big),
    \end{align}
    for $2^{m+3}\sqrt{2\Cr{c162}m}  \beta_2 (\gamma\beta_1)^{m-{1\over 2}}
        \leq t
        \leq 2^{m+2}\beta_2 (\gamma\beta_1)^{m}$.
    \end{itemize}
\end{theorem}

\begin{remark}
    In (a) we recover the behavior of Theorem \ref{thm:largeorder2} with slightly worse constants.
\end{remark}

\begin{proof}[Proof of Theorem \ref{tm:momentConcRes_NEW}.a]
    Applying Markov's inequality and Corollary \ref{cor:mombound} with $\gamma\beta_1\ge 1$, for any even $\ell\geq \max(\log(\beta_0),\gamma\beta_1)$, we have 
    \begin{align}
        \bP(|F_m-\bE F_m|\geq t)
        &\leq \frac{\bE[(F_m-\bE F_m)^\ell]}{t^\ell}\notag
        \\
        &\leq \left({e (2^qm\ell)^m \beta_2\max\big(1,{\gamma\beta_1\over m\ell}\big)^{(m-{1\over 2})q}\over t\log^{m(1-q)}\big(1+{m\ell\over\gamma\beta_1}\big)}\right)^{\ell} = \left({\ell\over e\varphi(t)\log^{1-q}\big(1+{m\ell\over\gamma\beta_1}\big)}\right)^{m\ell} , \label{eq:151a}
    \end{align}
    where $\varphi(t) \coloneqq  t^{1\over m} / [2^{q}e^{1+{1\over m}} m \beta_2^{1\over m}] $. 
    
    Note that the assumption $t\geq \Cr{c151} \beta_2 (\gamma\beta_1)^{m}$ gives that $\varphi(t) \geq (\Cr{c151}^{1\over m}/2^{q}e^{1+{1\over m}} m) \gamma\beta_1$.
    Further recalling, that $\Cr{c151}= \left( \lceil \max(2e^{2} , \lceil\log(\beta_0)\rceil )\times 2^{q}e^{1+{1\over m}} m\rceil  \right)^m $ and $\gamma\beta_1\ge 1$ we obtain 
    \begin{equation}\label{eq:152}
    \varphi(t)\ge \max(2e^{2}, \lceil\log(\beta_0)\rceil )\gamma\beta_1\ge 2e^2\gamma\beta_1\ge 2e^2>2.
    \end{equation}
    Let $\ell_0$ be the largest even integer number such that $\ell_0/ \log^{1-q}(1+m\ell_0/(\gamma\beta_1)
    ) \leq \varphi(t) $.  
   This choice is possible since for $\ell=2^{1-q}\lfloor\varphi(t)\rfloor\ge 2^{1-q}\lfloor 2e^{2}\gamma\beta_1\rfloor\ge 2\lfloor e^2\rfloor\lfloor \gamma\beta_1\rfloor$ by \eqref{eq:152} we have
    $$
    {\ell\over \log^{1-q}\big(1+{m\ell\over \gamma\beta_1}\big)}\leq  {2^{1-q}\lfloor\varphi(t)\rfloor\over \log^{1-q}\big(1+{2m\lfloor e^2\rfloor\lfloor \gamma\beta_1\rfloor\over \gamma\beta_1}\big)}\leq {2^{1-q}\lfloor\varphi(t)\rfloor\over \log^{1-q}(1+\lfloor e^{2}\rfloor)}\leq \lfloor\varphi(t)\rfloor,
    $$
    and since $\ell/\log^{1-q}\big(1+{m\ell\over\gamma\beta_1}\big)$ is strictly increasing in $\ell$, we conclude $\ell_0 \geq 2^{1-q}\lfloor\varphi(t)\rfloor>2$. Further, from the definition of $\ell_0$ and since $\ell_0\ge 2$, we obtain
    \begin{equation*}
        \frac{\ell_0}{\log^{1-q}\big(1+{m\ell_0\over\gamma\beta_1}\big)}
        \leq \varphi(t) 
        < \frac{\ell_0+2}{\log^{1-q}\big(1+{m(\ell_0+2)\over\gamma\beta_1}\big)}
        \leq \frac{2 \ell_0}{\log^{1-q}\big(1+{m\ell_0\over\gamma\beta_1}\big)} .
    \end{equation*}
    Our next aim is to show that $\ell_0\ge \log(\beta_0)$ and $\ell_0\ge \gamma\beta_1$ so that \eqref{eq:151a} applies. Indeed since $\gamma\beta_1\ge 1$ and due to \eqref{eq:152} we have
    $$
    \ell_0 \geq 2^{1-q}\lfloor\max(2e^{2}, \lceil\log(\beta_0)\rceil )\gamma\beta_1\rfloor \ge \lceil\log(\beta_0)\rceil\ge \log(\beta_0),\qquad\qquad \ell_0 \geq 2\lfloor e^2\rfloor\lfloor \gamma\beta_1\rfloor\ge \lfloor \gamma\beta_1\rfloor+1\ge \gamma\beta_1.
    $$
    In particular Equation \eqref{eq:151a} holds for $\ell=\ell_0$, which gives
    \begin{align*}
        \bP(|F_m-\bE F_m|\geq t)
        &\leq \left({\ell_0\over e\varphi(t)\log^{1-q}\big(1+{m\ell_0\over\gamma\beta_1}\big)}\right)^{m\ell_0}
        \leq e^{- m \ell_0}
        \leq \exp\Big(- {m\over 2} \log^{1-q}\Big(1+{m\ell_0\over \gamma\beta_1}\Big) \varphi(t)\Big) .
    \end{align*}    
    Note that we have $\ell_0\geq 2^{1-q}\lfloor \varphi(t)\rfloor\ge 2\lfloor e^2\rfloor\lfloor\gamma\beta_1\rfloor$ and it holds that
    \begin{align*}
    \log\left(1+{2^{1-q}m\ell_0\over\gamma\beta_1\log(1+{m\ell_0\over \gamma\beta_1})}\right)&\leq \log\left(1+{2^{1-q}m\ell_0\over\gamma\beta_1\log^{1-q}(1+{2m\lfloor e^2\rfloor\lfloor\gamma\beta_1\rfloor\over \gamma\beta_1})}\right)\\
     &\leq \log\left(1+{2^{1-q}m\ell_0\over\gamma\beta_1\log^{1-q}(1+{\lfloor e^{2}\rfloor})}\right)\leq \log\left(1+{m\ell_0\over\gamma\beta_1}\right).
    \end{align*}
    Thus, we also get that,
    \begin{align*} 
        \log\Big(1+{m\ell_0\over\gamma\beta_1}\Big)&\ge \log\Big(1+{2^{1-q}m\ell_0\over\gamma\beta_1\log^{1-q}(1+{m\ell_0\over \gamma\beta_1})}\Big)\ge \log\Big({m\over 2^q}{\varphi(t)\over \gamma\beta_1}\Big)={1\over m}\log\Big({1\over 2^{2mq}e^{m+1}}{t\over \beta_2(\gamma\beta_1)^m}\Big).
    \end{align*}
    Therefore ,
    \begin{align*}
        \bP(|F_m-\bE F_m|\geq t)
        & \leq \exp\Big(- \frac{1}{2^{1+q}e^{1+{1\over m}}m}\Big({t\over \beta_2}\Big)^{1\over m} \log^{1-q}\Big({1\over 2^{2mq}e^{m+1}}{t\over \beta_2(\gamma\beta_1)^m}\Big) \Big) .
    \end{align*}
\end{proof}

\begin{proof}[Proof of Theorem \ref{tm:momentConcRes_NEW}.b]
    Again, we use Markov's inequality in combination with Corollary \ref{cor:mombound}. For any even number $\ell\geq \max(2, \log(\beta_0)) $ and $\gamma\beta_1\ge 2m\ell$, we get
    \begin{align}
        \bP(|F_m-\bE F_m|\geq t)
        \leq \frac{\bE[(F_m-\bE F_m)^\ell]}{t^\ell}
        &\leq \left(\frac{\beta_2^2 (\gamma\beta_1)^{2m-1}2^{2m+1}e^2 m\ell}{t^2}  \right)^{\ell/2} 
        \leq \left( \frac{\ell}{e \varphi(t)} \right)^{\ell/2} .\label{eq:154}
    \end{align}
    with $ \varphi(t) \coloneqq  t^2 / [ 2^{2m+6}m \beta_2^2 (\gamma\beta_1)^{2m-1} ] $, where we used that $e^3 \leq 2^5$.
    Note that the assumptions on $t$ are equivalent to 
    \begin{equation}\label{eq:153}
    2\max(1, \lceil \log(\beta_0) \rceil ) \leq \varphi(t) \leq (4m)^{-1}\gamma\beta_1.
    \end{equation}
    Let $\ell_0$ be the largest even number such that $\ell_0\leq \varphi(t)$.
    Since the left hand side of \eqref{eq:153} is an even number, there exists such a $\ell_0$ and it satisfies $\ell_0 \geq 2\max(1, \lceil \log(\beta_0) \rceil ) \geq 2$. 
    It follows that
    \begin{equation*}
       \ell_0\leq \varphi(t)\leq \ell_0+2\leq 2\ell_0,
    \end{equation*}
    and in combination with \eqref{eq:153} this implies, that $\gamma\beta_1\ge 4m\ell_0$. Hence, $\ell_0$ satisfies \eqref{eq:154} and we get
    \begin{align*}
        \bP(|F_m-\bE F_m|\geq t) 
        & \leq \left( \frac{\ell_0}{e \varphi(t)} \right)^{\ell_0/2} 
        \leq \left( \frac{1}{e} \right)^{\varphi(t)/4}
        \leq \exp\Big(-\frac{t^2}{2^{2m+8}m \beta_2^2 (\gamma\beta_1)^{2m-1} }\Big).
    \end{align*}
\end{proof}


\section{Summary and discussion of the results}\label{sec:discussion}

Our main theorem was presented in the introduction in a simple form, without explicit constants. 
We give now its extended version, as well as its proof.
It summarizes the results of the previous sections. 

\newtheorem*{theorem*}{Theorem \ref{tm:Summary} Extended}

\begin{theorem*}[Main result]\label{tm:SummaryComplete} \,
    Assume $m\geq 1$and $F_m$ satisfies \eqref{eq:fBound}. Then
    \begin{subnumcases}{\bP(|F_m-\bE F_m| \geq t) \leq \label{boundA2}}
        2\exp\left( - \Cr{c_26} \frac{t^2}{\gamma^{2m-1}} \right), 
        & \text{if $0\leq  t\leq \Cr{c_24} \gamma^{m-\frac{1}{2}}$, $\gamma\ge \Cr{c_11}$;} \label{newbound4}
        \\ \exp\left( - \Cr{c_23} \frac{t^2}{\gamma^{2m-1}} \right), 
        & \text{if $ \Cr{c_24} \gamma^{m-\frac{1}{2}} \leq  t\leq \Cr{c_17}\gamma^{m}$, $\gamma\ge \Cr{c_11}$;} \label{newbound5}
        \\\exp \left( -\Cr{c_15} t^{\frac{1}{m}}\log^{1-q} \left( \Cr{c_16} \frac{t}{\gamma^m}\right) \right), 
        & \text{if $t\ge \Cr{c_17}\gamma^m$, $\gamma \geq \Cr{c_9}$} \label{newbound6}
    \end{subnumcases}
    where
    \begin{align*}
        \Cl{c_26} &= \frac{\|f_1\|_{L^2(\Lambda)}^2 }{ 2^{17m+4}(\beta_0 m)^{2m+3} \beta_2^4 \beta_1^{4m-2}} &
        \Cl{c_23} &= \frac{1}{2^{16m+4} (m \beta_0)^{2m+1} \beta_2^2 \beta_1^{2m-1}} &
        \Cl{c_24} &= 2^{8m+{3
        \over 2}}(m \beta_0)^{m+{1\over 2}} \beta_2 \beta_1^{m-\frac{1}{2}} \\
        \Cl{c_11} &={8 m \beta_0}{\beta_1^{-1}}, &
        \Cl{c_9} &=\beta_1^{-1}, &
        \Cl{c_15}&=(2^{1+q}em)^{-1}(e\beta_2)^{-{1\over m}},\\
        \Cl{c_16}&=2^{-mq}e^{-m-1}\beta_2^{-1}\beta_1^{-m}, &
        \Cl{c_17}&= 2^{8m}\beta_2 (m \beta_0 \beta_1)^m. &
    \end{align*}
    In particular, if $\beta_1 \gamma \geq 8 m \beta_0 $, then for any $t\geq 0$, we have 
    \begin{align} \label{onebound2} \tag{\ref*{boundA2}d}
        \bP(|F_m-\bE F_m| \geq t)
        &\leq 2\exp \left( - \Cr{c_27} \Big(\frac{t}{\beta_2} \Big)^{\frac{1}{m}} \min \left( \left[ \frac{t}{\beta_2 (\beta_1 \gamma)^m} \right]^{2-\frac{1}{m}} , \left[1 + \log_+ \left( \frac{t}{\beta_2 (\beta_1 \gamma)^m} \right)\right]^{1-q} \right)  \right), 
    \end{align}
    where $\Cl{c_27} := 2^{-17m-4}(\beta_0 m)^{-2m-3} \min\left( 1 , \|f_1\|_{L^2(\Lambda)}^2 \beta_2^{-2} \beta_1^{1-2m} \right) $. If we consider the bound for upper or lower tail only, the factor 2 in front of the exponent can be removed.
\end{theorem*}

\begin{remark}
	Let us point out that in case when $F_m$ satisfies \eqref{assumptions} we may alternatively apply Theorem \ref{thm:largeorder2} instead of Theorem \ref{tm:momentConcRes_NEW}.b. The bounds obtained in this case will be of the same form as bounds \eqref{newbound4}, \eqref{newbound5}, \eqref{newbound6} and \eqref{onebound2}, with $\beta_0=1$, $\beta_1=\alpha_1$, $\beta_2=\alpha_2$ and $q=0$, but the corresponding constants will be sharper in terms of $m$. Moreover we may take $\Cr{c_9}=0$ in this case.
\end{remark}

\begin{proof}[Proof of \eqref{newbound6}]
    It follows from Theorem \ref{tm:momentConcRes_NEW}.a and the observations that $\Cr{c151}\beta_2\beta_1^m\leq \Cr{c_17}$, because
    \begin{align*}
        \lceil \max(2e^{2} , \lceil\log(\beta_0)\rceil )\times 2^qe^{1+{1\over m}} m\rceil
        \leq 2^8 m\lceil\log(\beta_0)\rceil
        \leq  2^8 m \beta_0 ,
    \end{align*}
    which follows from the trivial bounds
    $\max(2e^2, \lceil\log(\beta_0)\rceil ) e^{1+\frac{1}{m}} \leq 2e^{4}\lceil\log(\beta_0)\rceil$,
    $e^4\leq 2^6$,
    and $\lceil\log(\beta_0)\rceil \leq \beta_0$.     
\end{proof}

\begin{proof}[Proof of \eqref{newbound5}]
    Theorem \ref{tm:momentConcRes_NEW}.b gives that for $\gamma\beta_1\geq 8m\Cr{c162}$,
    we have
    \begin{align*}
        \bP(|F_m-\bE F_m|\geq s)
        &\leq \exp\Big(-\frac{s^2}{2^{2m+8}m \beta_2^2 (\gamma\beta_1)^{2m-1}}\Big),
    \end{align*}
    for $2^{m+3}\sqrt{2\Cr{c162}m} \beta_2 (\gamma\beta_1)^{m-{1\over 2}}
    \leq s
    \leq 2^{m+2}\beta_2 (\gamma\beta_1)^{m}$,
    where $\Cr{c162} = \max(1,\lceil \log \beta_0 \rceil) \leq \beta_0$.
    Taking $s= \frac{2^{m+2}}{2^{8m}(m \beta_0)^m} t < t $ and noting that $\bP(|F_m-\bE F_m|\geq t) \leq \bP(|F_m-\bE F_m|\geq s)$ gives 
    \begin{align*}
        \bP(|F_m-\bE F_m|\geq t)
        &\leq \exp\Big(- \frac{t^2}{2^{16m+4} (m \beta_0)^{2m+1}\beta_2^2 (\gamma\beta_1)^{2m-1}}\Big),
        \\& \text{for } 2^{8m+1}\sqrt{2\Cr{c162}m}(m \beta_0)^m \beta_2 (\gamma\beta_1)^{m-{1\over 2}}
            \leq  t
            \leq 2^{8m}(m \beta_0)^m\beta_2 (\gamma\beta_1)^{m}.
    \end{align*}
    Thus \eqref{newbound5} holds.    
\end{proof}

\begin{proof}[Proof of \eqref{newbound4}]
    If $\|f_1\|_{L^2(\Lambda)}=0$ then $\Cr{c_26}=0$ and the statement is trivial.
    For the rest of the proof, we assume that $\|f_1\|_{L^2(\Lambda)}>0$.
    Due to \eqref{eq_UstatVarBounds} this implies that $\bV F_m>0$.
    Thus, applying Proposition \ref{prop:ChebyshevCantelli} gives that, for any constants $\Cr{c47}, \Cr{c43}>0$,
    \begin{align*}
        \bP(F_m-\bE F_m \geq t) 
        & \leq \exp \left( -\frac{t^2}{\Cr{c44}\beta_2^2(\beta_1\gamma)^{2m-1}}\right),
        \ 0\leq t<\Cr{c43}\sqrt{\bV F_m}, \ \gamma \beta_1\ge \Cr{c47},
    \end{align*}
    where
    $$
    \Cr{c44}=2^q\beta_0 m^2\big(
    2^q\Cr{c47}^{-1} m + 1 \big)^{m-1}(1+\Cr{c43}^2).
    $$
    We set $ \Cr{c47}=8m\beta_0 \ge 2^q m $ and 
    $\Cr{c43} 
    = \Cr{c_24} \gamma^{m-\frac{1}{2}} / \sqrt{\bV F_m} $.
    By \eqref{eq_UstatVarBounds} we have $\bV F_m \geq \|f_1\|_{L^2(\Lambda)}^2 \gamma^{2m-1}$ and, thus, 
    $$ \Cr{c43}^2 
    \leq \Cr{c_24}^2 / \|f_1\|_{L^2(\Lambda)}^2 
    =2^{16m+3}(m\beta_0)^{2m+1} \beta_2^2 \beta_1^{2m-1} / \|f_1\|_{L^2(\Lambda)}^2.$$
    Moreover, since by \eqref{eq:fkA2Bound1} we have $\|f_1\|_{L^2(\Lambda)}^2 
    \leq 2^qm^2 \beta_0 \beta_2^2 \beta_1^{2m-1}$, it follows that
    \begin{align*}
        1+\Cr{c43}^2 
        & \leq ( 1 + 2^{16m+3-q}m^{2m-1} \beta_0^{2m} ) \frac{2^qm^2 \beta_0 \beta_2^2 \beta_1^{2m-1}}{\|f_1\|_{L^2(\Lambda)}^2}
        \leq 2^{16m+4}(m\beta_0)^{2m+1} \frac{\beta_2^2 \beta_1^{2m-1}}{\|f_1\|_{L^2(\Lambda)}^2} .
    \end{align*}
    Thus,
    \begin{align*}
        \Cr{c44} 
        &\leq 2^q\beta_0 m^2 2^{m-1}(1+ \Cr{c43}^2 ) 
        \leq \beta_0 m^2 2^{17m+4} (m\beta_0)^{2m+1} \frac{\beta_2^2 \beta_1^{2m-1}}{\|f_1\|_{L^2(\Lambda)}^2}
        \leq 2^{17m+4}(\beta_0 m)^{2m+3} \frac{\beta_2^2 \beta_1^{2m-1}}{\|f_1\|_{L^2(\Lambda)}^2}.
    \end{align*}
    Further we note that $\bP(F_m-\bE [F_m]\leq -t)=\bP((-F_m)-\bE[(-F_m)]\ge t)$ and $\bV [F_m]=\bV[(-F_m)]$. Thus, by applying Proposition \ref{prop:ChebyshevCantelli} for $-F_m$ and using the same arguments as above we get
    $$
    \bP(F_m-\bE F_m\leq -t)\leq \exp\left( - \Cr{c_26} \frac{t^2}{\gamma^{2m-1}} \right), 
    $$
    for any  $0\leq  t\leq \Cr{c_24} \gamma^{m-\frac{1}{2}}$, $\gamma\ge \Cr{c_11}$, which finishes the proof.
\end{proof}

\begin{proof}[Proof of \eqref{onebound2}]
    First, assume that $t\leq \beta_2 (\beta_1 \gamma)^m$.
    Then $t$ is either in the range given by \eqref{newbound4} or by \eqref{newbound5}.
    The right hand side of \eqref{onebound2} evaluates to $2\exp( - \min(\Cr{c_26},\Cr{c_23}) t^2 \gamma^{-m} )$.
    Thus \eqref{newbound4} and \eqref{newbound5} implies \eqref{onebound2}.

    Now, assume that $t\geq \beta_2 (\beta_1 \gamma)^m$.
    The right hand side of \eqref{onebound2} equals $2\exp \left( - \Cr{c_27} \left(\frac{t}{\beta_2} \right)^{\frac{1}{m}} \log \left( \frac{e t}{\beta_2 (\beta_1 \gamma)^m} \right) \right) ,$ which is larger that the right hands sides of \eqref{newbound4}, \eqref{newbound5} and \eqref{newbound6}.
    Thus \eqref{onebound2} holds as well.
\end{proof}

\begin{remark}
    Under additional condition \eqref{assumptionflowbound} we have $\|f_1\|^2_{L^2(\Lambda)}\ge m^2\theta_2^2\theta_1^{2m-1}>0$. Also if $F_m$ satisfies \eqref{assumptionsBR} we get
    \[
    \|f_1\|^2_{L^2(\Lambda)}
    \ge 
    M_1^2m^2\int_{\bR^d}\Lambda(B_{\rho}(x))^{2m-2}\Lambda(\dint x)>0.
    \]
    In particular, in these cases all constants in the above theorem are strictly positive, and can be lower bounded by functions of the parameters introduced in the corresponding assumptions.
\end{remark}

\subsection{Optimality of the bounds}

We start by discussing the quality of the obtained bounds in different regimes. In particular we will consider the situations, when  $\gamma$ is fixed and $t\to \infty$, and when $\gamma\to \infty$ and $t=\gamma^{a}$ for some $a>0$. It is also assumed that the measure $\Lambda$ and the kernel $f$ are fixed. Recall that the bounds from Theorem \ref{tm:Summary} can be represented in the form
\begin{align*}
\bP(F_m-\bE F_m\leq -t)&\leq \exp(-I_-(\gamma,t)) ,\\
\bP(F_m-\bE F_m\geq t)&\leq \exp(-I_+(\gamma,t)),
\end{align*}
where $I_+ , I_- \colon [0,\infty) \times [0,\infty) \to [0,\infty)$ are some given rate functions.

\paragraph*{Fixed intensity and $t\to\infty$:}
First assume that $\gamma>1$ is fixed and that $I_+$ is such that 
$$
\bP(F_m-\bE F_m\ge t)\leq \exp(-I_+(\gamma,t))
$$ 
holds for sufficiently big $t$. Then, under assumption \eqref{assumptionflowbound}, Theorem \ref{thm:largeorder2_lowerbound} implies that
$$
\lim\sup_{t\to\infty} {I_+(\gamma,t)\over t^\frac{1}{m}\log t}\leq \Cr{c1905c},
$$
where $\Cr{c1905c}$ is an explicit constant. In particular this means, that the bound \eqref{newbound6} is of optimal order in $t$ if $q=0$ and \eqref{assumptionflowbound} holds.
The following remark points out that this holds in many classical applications.
\begin{remark} \label{rem:mildAssumption}
It should be observed that \eqref{assumptionflowbound} is a mild assumption which.
Suppose that $\Lambda$ is a Borel measure on a topological space $\bX$ which assigns positive measure to all open sets. Assume further that the kernel $f$ is positive and continuous on some open set $U \subset \bX^m$. Then $f$ can be bounded from below on some product $S_1\times\cdots\times S_m$ of pairwise disjoint open sets, and hence condition \eqref{assumptionflowbound} is satisfied.
In particular, this holds for the applications considered in Sections \ref{sec:geometricGraphs} and \ref{sec:Applications}. 
\end{remark}

\paragraph*{Increasing intensity and $t$ of order $\gamma^\alpha\to\infty$:}\label{sec:AnalysingCLT} 

Let us now assume that $t$ and $\gamma$ tend to infinity simultaneously and $t=\gamma^{a}$, $a>0$. We define a value $b$ as 
$$
b(a)=\lim_{\gamma\to\infty}{\log I_+(\gamma, \gamma^{a})\over \log \gamma},
$$
if the above limit exists. The values of $b(a)$ corresponding to the bounds from Theorem \ref{tm:Summary} are presented on Figure \ref{fig:graphics}.

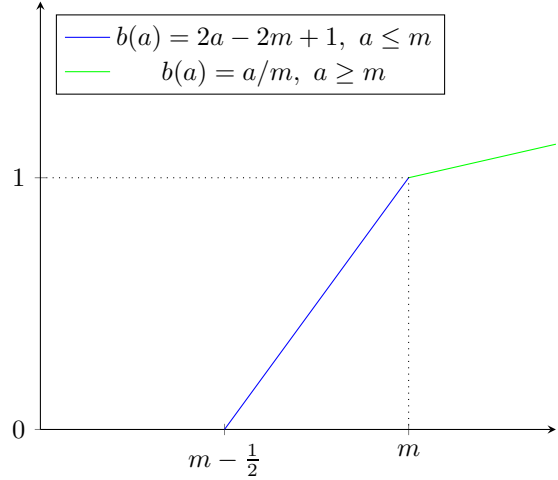
\begin{figure}
\begin{center}
        \begin{tikzpicture}
        \tikzmath{\m = 3;} 
        \begin{axis}[
            axis lines = left,
            ytick={0,1},
            xtick={{\m-1/2},\m},
            xticklabels={$m-\frac12$, $m$},
            yticklabels={$0$,$1$},
            legend pos=north west,
            xmin = {\m-1},
            ymax = 1.7,
        ]
        \addplot [
            domain=(\m-1/2):(\m), 
            samples=100, 
            color=blue,
            ]
            {2*x-2*\m+1};
        \addlegendentry{\(b(a) = 2 a -2m+1 ,\ a \leq m\)}
        \addplot [
            domain=\m:(\m+0.4), 
            samples=100, 
            color=green,
            ]
            {x/\m};
        \addlegendentry{\(b(a) = a/m ,\ a \geq m \)}
        \addplot[dotted] coordinates {(0,1) (\m,1) (\m,0)};
        \end{axis}
        \end{tikzpicture}
\end{center}
    \caption{Plots of $b(a) = \lim_{\gamma\to\infty} \frac{\log I_+(\gamma,\gamma^a)}{\log(\gamma)}$, $m\ge 1$. The blue and green curves correspond to bounds \eqref{newbound5} and \eqref{newbound6}, respectively.}\label{fig:graphics}
\end{figure}

\paragraph*{Increasing intensity and $t= s \sqrt{\bV F_m}$ (CLT regime):} In the above settings when the kernel $f$ and the measure $\Lambda$ are independent of $\gamma$ (which is named \textit{geometric $U$-statistic} in \cite[Definition 5.1]{RS13});  $F_m$ and Poisson $U$-statistic with kernel $|f|$ are square integrable; $\|f_1\|_{L^2(\Lambda)}>0$ and $\gamma\to \infty$ there is a quantitative Central Limit Theorem  proven in \cite[Corollary 4.3]{Schu16} of the following form: there exists a constant $\Cl{c43b}$, independent of $\gamma$, such that
\begin{equation}\label{eq:CLT}
\sup_{s\ge 0}\left|\bP\big(F_m-\bE F_m \ge s\sqrt{\V F_m }\big)-(1-\Phi(s))\right|=\sup_{s\ge 0}\left|\bP\big(F_m-\bE F_m \leq -s\sqrt{\V F_m }\big)-\Phi(-s)\right|\leq \Cr{c43b}\gamma^{-1/2},
\end{equation}
where $\Phi(s)$ is cumulative distribution function of a standard Gaussian random variable. Using the estimates from \cite[Theorem 2 with $\beta=2$]{ErfctL} for lower bound and \cite[Theorem 1]{ErfctL} for the upper bound we get 
$$
\frac{1}{2}\sqrt\frac{e}{2\pi}e^{-s^2}\leq 1-\Phi(s)=\Phi(-s)=\frac{1}{2}\text{ercf}\Big({s\over \sqrt{2}}\Big)={1\over \sqrt{2\pi}}\int_{s}^{\infty}e^{-x^2/2}\dint x\leq \frac{1}{2}e^{-s^2/2},
$$
which are valid for all $s\ge 0$. Hence, if
$$
s\leq \Big(\frac{1}{2}\log(\gamma)-\log \Big(4\Cr{c43b}\sqrt{\frac{2\pi}{e}}\Big)\Big)^{1/2},
$$
the inequality \eqref{eq:CLT} implies 
\begin{align}
\bP\big(|F_m-\bE F_m|\ge s\sqrt{\V F_m }\big)&\leq e^{-s^2/2}+2\Cr{c43b}\gamma^{-1/2}\leq e^{-s^2/2}+\frac{1}{2}\sqrt\frac{e}{2\pi}e^{-s^2}\leq 2 e^{-s^2/2}, \label{eq:241a}\\
\bP\big(|F_m-\bE F_m |\ge s\sqrt{\V F_m }\big)&\ge \sqrt\frac{e}{2\pi}e^{-s^2}-2\Cr{c43b}\gamma^{-1/2}\ge2\Cr{c43b}e^{-s^2}, \notag
\end{align}
which in particular holds for any fixed $s\ge 0$ and sufficiently big $\gamma$. 

Applying our results in the CLT regime, namely for $t=s\sqrt{\bV F_m}$, we obtain the following concentration inequalities, which agree with \eqref{eq:241a} up to constants.

\begin{corollary}
\label{cor:CLT}
    Assume that $F_m$ is a Poisson $U$-statistic satisfying \eqref{eq:fBound} and $\|f_1\|_{L^2(\Lambda)}>0$.
    Let $\Cr{c162} = \max(1, \lceil \log(\beta_0) \rceil ) $ as in Theorem \ref{tm:momentConcRes_NEW}.b.
    If $\gamma\beta_1\geq 8m\Cr{c162}$, then there exist positive constants $\Cl{c10}$ and $\Cl{c11}$, independent of $\gamma$ and $t$, such that
    \begin{align} \label{eq:ourconcbound}
        \bP(|F_m-\bE F_m| \geq s \sqrt{\V F_m}) 
        &\leq 2\exp \left( - \Cr{c10} s^2 \right) ,
        & 0 \leq s \leq \Cr{c11} \sqrt{\gamma \beta_1}.
    \end{align}
    Moreover, one can set
    \begin{align*}
        \Cr{c10} 
        &= {\|f_1\|^4_{L^2(\Lambda)}\over 2^{3m+7+q}(\beta_0m)^2m\beta_2^{4}\beta_1^{4m-2}},&
         \Cr{c11} 
        &= 2^{{m+5-q\over 2}}\beta_0^{-{1\over 2}}m^{-1} .
    \end{align*}
\end{corollary}

This corollary in particular means that the rate function of the form $I_+(\gamma,t)=\text{const}\cdot t^2\gamma^{-2m+1}$ gives an optimal bound (up to a constant) for $t=\text{const}\cdot \sqrt{\V F_m}\backsimeq \gamma^{m-1/2}$ as follows from \eqref{eq_UstatVar}. Thus, the CLT regime correspond to the point $(m-1/2,0)$ on Figure \ref{fig:graphics}. In comparison with the Central Limit Theorem our concentration bound \eqref{eq:ourconcbound} holds for large range of $s$ allowing to take $s$ to be of order up to $\gamma^{\frac{1}{2}}$. 

\begin{proof}
    Setting $t=s \sqrt{\V F_m}$, Theorem \ref{tm:momentConcRes_NEW}.b provides
    \begin{align}\label{eq:09.04.2024_1}
        \bP(|F_m-\bE F_m| \geq s \sqrt{\V F_m}) 
        & \leq \exp \Big( -  \frac{s^2 \V F_m}{2^{2m+8}m \beta_2^2 (\gamma\beta_1)^{2m-1}} \Big),
    \end{align}
    for $2^{m+3}\sqrt{2\Cr{c162}m}\beta_2 (\gamma\beta_1)^{m-{1\over 2}} (\V F_m)^{-{1\over 2}}  \leq  s \leq 2^{m+2} \beta_2 (\gamma\beta_1)^{m} (\V F_m)^{-{1\over 2}} $.
    Thus, we only need to bound $\V F_m$ appropriately to get the desired result. Recall that by  \eqref{eq_UstatVarBounds} with $\Cr{c47}=8m\Cr{c162}$ we have
    \begin{align*}
        0<\gamma^{2m-1}\|f_1\|^2_{L^2(\Lambda)}
        \leq \V F_m
        &\leq 2^q\beta_0 m^22^{m-1}\beta_2^2(\gamma\beta_1)^{2m-1}.
    \end{align*}
   It then follows that 
    \begin{align}
        2^{-2m-8}m^{-1}\beta_2^{-2}(\gamma\beta_1)^{-2m+1}\V F_m
         &\ge {\|f_1\|^2_{L^2(\Lambda)}\over 2^{2m+8}m\beta_2^2\beta_1^{2m-1}},\label{eq:09.04.2024_2}\\
         2^{m+2} \beta_2 (\gamma\beta_1)^{m} (\V F_m)^{-{1\over 2}}
         & \geq 2^{{m+5\over 2}}m^{-1}\sqrt{\gamma\beta_1}
         = \Cr{c11} \sqrt{\gamma\beta_1},\notag
         \\
        2^{m+3}\sqrt{2\Cr{c162}m}\beta_2 (\gamma\beta_1)^{m-{1\over 2}} (\V F_m)^{-{1\over 2}}
         & \leq  {2^{m+3}\sqrt{2\beta_0m}\beta_2\beta_1^{m-{1\over 2}}\over \norm{f_1}_{L^2(\Lambda)}}.\notag
    \end{align}
   Further by applying Proposition \ref{prop:ChebyshevCantelli} with 
   $$
   \Cr{c43}= {2^{m+3}\sqrt{2\beta_0m}\beta_2\beta_1^{m-{1\over 2}}\over \norm{f_1}_{L^2(\Lambda)}}\ge {2^{m+3}\over \sqrt{m}}>1,
   $$ 
   as follows from \eqref{eq:fkA2Bound1} and $m\ge 1$, with $t=s\sqrt{\bV F_m}$ and $\Cr{c47}=8m\Cr{c162}$ and since $\bV[F_m]=\bV[-F_m]$ we get
   \[
   \bP(|F_m-\bE F_m| \geq s \sqrt{\V F_m})\leq 2\exp \Big( -  \frac{s^2 \V F_m}{\Cr{c44}\beta_2^2 (\gamma\beta_1)^{2m-1}} \Big),
   \]
   for $0\leq s\leq \Cr{c43}$, where
   \[
   \Cr{c44}^{-1}\beta_2^{-2} (\gamma\beta_1)^{-2m+1}\V F_m\ge {\|f_1\|^4_{L^2(\Lambda)}\over 2^{3m+7+q}(\beta_0m)^2m\beta_2^{4}\beta_1^{4m-2}}.
   \]
   Combining this with \eqref{eq:09.04.2024_1} and \eqref{eq:09.04.2024_2} finishes the proof.
\end{proof}

\subsection{Comparison with known results for Poisson \texorpdfstring{$U$}{U}-statistics}
Finally we compare our results with known bounds for Poisson $U$-statistics. 

The first set of bounds we consider was obtained by Bachmann and Reitzner in \cite[Theorem 3.1]{BR18} and it applies to Poisson $U$-statistics over $\bR^d$ satisfying \eqref{assumptionsBR}. The results have been motivated by the applications to subgraph counts of random geometric graphs.
Together with bounds for the lower tail, which are of the similar form as the one presented in Proposition \ref{prop:LowerTail}, and bounds around the median, they show that
\begin{align*}
    \bP(F_m-\bE F_m\geq t)
    & \leq \exp \Big( -  \Cl{constBR} \cdot\big((\bE F_m+t)^\frac{1}{2m}-(\bE F_m)^\frac{1}{2m}\big)^2 \Big),
    && \text{for any $t>0$ ,}
\end{align*}
where the constant $\Cr{constBR} = \Cr{constBR}(d, m, \Theta, M_1 , M_2) $ has an explicit representation in terms of the parameters involved in the assumption \eqref{assumptionsBR}.
Up to the specific value of the constant, 
this result is equivalent to the existence of constants $\Cr{constBR}'$, $\Cr{constBR}''$ and $\Cl{constBR2}$, depending on $m$, $d$, $\rho$, $\Theta$, $M_1$, $M_2$ and $\bE F_m$
for which we have
\begin{align}
    \label{BRsubMean}
    \bP(F_m-\bE F_m\geq t)
    & \leq 
    \exp \Big( -  \Cr{constBR}'  \frac{t^2}{\gamma^{2m-1}} \Big),
    && \text{if $0<t\leq \Cr{constBR2}\gamma^m$ ,} 
    \\
    \label{BRsupMean}
    \bP(F_m-\bE F_m\geq t)
    & \leq 
    \exp \Big( -  \Cr{constBR}'' t^\frac{1 }{m} \Big),
    && \text{if $t\geq 2\Cr{constBR2}\gamma^m$.}
\end{align} 
On the other hand, recall that according to Lemma \ref{lem:relAssumptions} assumption \eqref{assumptionsBR} implies \eqref{eq:fBound} with $q=0$ and $\|f_1\|_{L^2(\Lambda)}>0$, hence, from Theorem \ref{tm:Summary} we get 
\begin{align*}
    \bP(F_m-\bE F_m \geq t)
    &\leq \exp\left( - \min(\Cr{c_26}, \Cr{c_23})\frac{t^2}{\gamma^{2m-1}} \right), 
    && \text{if $0\leq  t\leq \Cr{c_17}\gamma^{m}$, $\gamma\ge \Cr{c_11}$,}
    \\
    \bP(F_m-\bE F_m \geq t)
    &\leq \exp \left( -\Cr{c_15} t^{\frac{1}{m}}\log \left( \Cr{c_16}{t\over \gamma^m}\right) \right), 
    && \text{if $t\ge \Cr{c_17}\gamma^m$, $\gamma\ge \Cr{c_9}$,}
\end{align*}
where the constants $\Cr{c_9}, \Cr{c_11}$ depend only on $\Lambda$ and $\Theta\rho$, while constants $\Cr{c_26}$, $\Cr{c_23}$ and $\Cr{c_15},\ldots,\Cr{c_17}$ depend additionally on $m$, $M_2$ and only constant $\Cr{c_26}$ depends on $
\|f_1\|_{L^2(\Lambda)}$. The exact values of these constants can be derived by combining Theorem \ref{tm:Summary} and Lemma \ref{lem:relAssumptions}. The main advantage of our approach is that in contrast to \eqref{BRsupMean} we gain a logarithmic factor in the case $t\ge \Cr{c_17}\gamma^m$ and provide a bound of optimal order in $t$ when  $t\to\infty$. Moreover, we can drop the conditions that $f\ge M_1>0$ whenever $f>0$ and $f(x_1,\ldots,x_m)>0$ if $\operatorname{diam}(x_1,\ldots,x_m) \leq \rho$, if we additionally require $\bE F_m<\infty$ (and $\|f_1\|_{L^2(\Lambda)}>0$). Thus, in some situations our bounds hold under slightly relaxed conditions.

In the recent paper \cite{ST23}, Schulte and Th\"ale have used the cumulants method for Poisson $U$-statistics in order to establish moderate deviation principle and derive Bernstein and Cram\'er type concentration inequalities. We will not provide a detailed analysis of their results here since it would require introducing additional notation, but we note, that the concentration bounds from \cite{ST23} hold under the assumption that there is a specific bound for the cumulants of Poisson $U$-statistic. This condition is similar to \eqref{eq:fBound}, but is slightly more restrictive and in particular implies \eqref{eq:fBound}. Moreover application of the cumulant method allows only to obtain a bound of the form $\exp(-c t^{1\over m})$ in case when $t\to\infty$ and other parameters are kept fixed, while in our results the missing logarithmic factor is recovered, namely we obtain a bound of the form $\exp(-c 
 t^{1\over m}\log(c't))$.

Another set of concentration bounds obtained in \cite[Proposition 5.6]{BR18} is for almost surely finite Poisson $U$-statistics over $\bR^d$ with non-negative kernel satisfying
$$
\sum\limits_{x\in \eta} (D_x F_m(\eta-\delta_x))^2 
    \leq c F_m^{\delta},
$$
almost surely for some $\delta\in[0,2)$ and $c>0$. The bounds are of the form
\begin{align}
\bP(F_m-\bM F_m\ge t)&\leq 2\exp\Big(-{\rm const}\cdot{t^2\over (t+\bM F_m)^{\delta}}\Big),\label{eq:ConcentrationUStatMedian1}\\
\bP(F_m-\bM F_m\leq  -t)&\leq 2\exp\Big(-{\rm const}\cdot{t^2\over (\bM F_m)^{\delta}}\Big),\notag
\end{align}
where $\bM F_m$ is the median of $F_m$. In \cite{NRP21} this concentration bounds have been extended to more general settings, namely when $\bX$ is a complete separable metric space, $\Lambda$ is some $\sigma$-finite measure on $\bX$ and $F_m$ is non-negative convex functional. In particular, if $F_m$ satisfies \eqref{assumptionsBR}, it was shown that one can take $\delta=2-1/m$, see \cite[Theorem 3.2]{BR18}. One more set of concentration inequalities for Poisson $U$-statistic around the median was obtained in \cite{RST17} using the Talagrand's inequality for convex distances \cite{R13}. It takes the following form (see also \cite{PeccatiReitzner}). Let $\bX$ be a Polish space equipped with the norm $\|\cdot\|$, $F_m$ be a local $U$-statistic of radius $r$, $B_r(x)$ be the ball of radius $r>0$ around $x\in\bX$, $E\coloneqq \sup_{x\in\bX}\Lambda(B_r(x))$, then for any $t^2/(t+\bM F_m)\ge \gamma^m E^me^{2m}\|f\|_{\infty}$ we have
\begin{equation}\label{eq:ConcentrationUStatMedian2}
\bP(|F_m-\bM F_m|\ge t)\leq 4\gamma\Lambda(\bX)\exp\Big(-{\rm const}\cdot \|f\|_{\infty}^{-\frac{1}{m}}\Big(\frac{t^2}{t+\bM F_m}\Big)^\frac{1}{m}\Big).
\end{equation}
For this estimate to be non trivial, it is required that $\Lambda$ is a finite measure and $f$ is bounded, which is our assumption~\eqref{assumptions}. 

In order to compare our bounds around the mean and the bounds around the median, we note that due to Chebyshev's inequality we have $|\bM F_m- \bE F_m|\leq \sqrt{2\bV F_m}$, see \cite[proof of Theorem 4.3(ii)]{BR18}, and, thus, the concentration inequalities around $\bM F_m$ and $\bE F_m$ can be compared only if $t$ is at least of the order of $\sqrt{\bV F_m}=c\cdot \gamma^{m-{1\over 2}}$.

For fixed $\gamma$ and $t\to\infty$ the right hand side of \eqref{eq:ConcentrationUStatMedian1} and \eqref{eq:ConcentrationUStatMedian2} is greater than $\exp( - c t^{\frac{1}{m}} )$ for some constant $c$, which is worse (by a logarithmic factor) than our upper bound \eqref{newbound6}. Similarly, for $t=\gamma^{a}$, $a\ge m$, the estimate \eqref{eq:ConcentrationUStatMedian2} and \eqref{eq:ConcentrationUStatMedian1} has the form  $\exp( - c \gamma^{\frac{a}{m}})$, while our result include additional logarithmic factor.
Finally, due to the restrictions $t^2/(t+\bM F_m)\ge \gamma^m E^me^{2m}\|f\|_{\infty}$ the case $m-1/2\leq a< m$ is not covered by \eqref{eq:ConcentrationUStatMedian2}, while \eqref{eq:ConcentrationUStatMedian1} gives $\exp( - c t^2/\gamma^{2m-1})$, which is of the same quality as \eqref{newbound5}.


\section{Application to random geometric graphs} \label{sec:geometricGraphs}

\subsection{Poisson \texorpdfstring{$U$}{U}-statistics for the random geometric graph in a metric space}
Let $\bX$ be a metric space with distance denoted by $\dist(\cdot,\cdot)$, and equip $\bX$ with the induced topology and sigma algebra. Let $B_r(x)$ denote the closed ball in $\bX$ of radius $r>0$ centred in $x$.
Let $\Lambda$ be a non-atomic and $\sigma$-finite measure, $\gamma>0$ and set $\eta$ to be a Poisson point process on $\bX$ with intensity measure $\gamma\Lambda$.
Given a parameter $\rho>0$ and a countable subset $V\subset \bX$, define the undirected graph $G=G(V,\rho)=(V,E)$, where the set of vertices is $V$ and set of edges $E$ consists of all pairs $\{x,y\}$, $x,y\in V$ such that $0<\dist(x,y)\leq \rho$. 
When $V=\eta$ the random graph $G(\eta,\rho)$ is called \textbf{random geometric graph} or Gilbert graph (see \cite{Penrose} for more details on the model). The random geometric graphs have been studied in \cite{bachmann2016concentration, BR18, L-RP13b,L-RP13,  RST17}.

Given a graph $H$ we denote by $d(H)$ its combinatorial diameter, with the convention that $d(H)=\infty$ if $H$ is disconnected. 
Further we write $H'\cong H$ if graphs $H$ and $H'$ are isomorphic, with respect to the combinatorial distance.

Let $m,n\in\bN$ with $1\leq n \leq m-1$, $\rho,\, Q>0$ and $h\colon\bX^m\to\bR_+$ a measurable function satisfying $h(x_1,\ldots,x_m) \in [0,Q]$ for any $x_1,\ldots,x_m\in \bX$ with $d(G(\{x_1,\ldots,x_m\},\rho))\leq n$.
Then, we set
\begin{align} \label{eq:UstatisticGraph1}
    F_m
    &:= F_m(\eta, \rho,n,h)
    := {1\over m!}\sum_{(x_1,\ldots,x_m)\in \eta^m_{\neq}}h(x_1,\ldots,x_m){\bf 1}\{d(G(\{x_1,\ldots,x_m\},\rho)) \leq n \} . 
\end{align}

We will see below that under reasonable assumptions this model satisfies \eqref{assumptionsGeneral} and, therefore, our concentration bounds can be applied, see Lemmas \ref{lm:1condEnough} and \ref{lm:A3forGraphs} below.
In the Euclidean setting $\bX=\bR^d$, this covers several special cases which have attracted a lot of attention in the literature.
We will describe them in the two next paragraphs.
We will also use this to derive concentration bound for the total weighted edge length in hyperbolic and spherical setting, see Theorem \ref{thm:RGGnonEuclidean}.
To the best of our knowledge, it is the first time that concentration bounds are established in this context.

\paragraph{Edge count:} The simplest non trivial example for $F_m$ is arguably obtained by setting $m=2$, $n=1$, $h\equiv 1$ and $Q=1$.
In this case $F_m$ is simply the number of edge of the graph $G(\eta,\rho)$.
Concentration bounds for this specific functional were established by Bachmann and Peccati in \cite[Section 6]{bachmann2016concentration} for the Euclidean setting and with $\Lambda$ satisfying \eqref{eq:IntensityMeasureCondition1}.
It has also been studied in \cite{BR18,RST17} via two distinct generalizations, which we present in the next paragraphs. Both situations are special cases of our general setting.

\paragraph{Induced and Included Subgraph Counts:}
Let $H$ be a connected graph on the vertex set $V=[m]$.
For $n = d(H)$, $Q=1$, and $h(x_1,\ldots,x_m) = {\bf 1}(G(\{x_1,\ldots,x_m\},\rho) \cong H)$, we denote the $U$-statistic $F_m$ by
$$
    F_{=}^{H}
    := F_{=}^{H} (\eta,\rho)
    = \frac{1}{m!} \sum_{(x_1,\ldots,x_m) \in \eta^m_{\neq}}  {\bf 1}(G(\{x_1,\ldots,x_m\},\rho) \cong H),
$$
which is known as the \textit{induced} subgraph count of $H$ in $G(\eta,\rho)$.
In \cite{BR18}, where the assumptions \eqref{assumptionsBR} have been introduced, Bachmann and Reitzner noted that, unless $H$ is a complete graph, such kernel do not satisfy their assumption (see \cite[Remark 4.1]{BR18}).
Instead, they establish concentration bounds on a slightly different subgraph count, which one could call \textit{included} subgraph count.
For this one sets $n = d(H)$, $Q=m!$ and $h(x_1,\ldots,x_m) = | \{ \text{subgraphs $H'$ of $G(\{x_1,\ldots,x_m\},\rho)$} : H' \cong H \} |$, and in this case we denote $F_m$ by
\begin{equation} \label{eq:includedgraphs}
    F_{\subset}^{H}
    := F_{\subset}^{H} (\eta,\rho)
    = \frac{1}{m!} \sum_{(x_1,\ldots,x_m) \in \eta^m_{\neq}}  | \{ \text{subgraphs $H'$ of $G(\{x_1,\ldots,x_m\},\rho)$} : H' \cong H \} |.
\end{equation}

\paragraph{Power-weighted edge length:} 
Set a parameter $\tau\geq 0$ and let $m=2$, $n=1$, $Q=\rho^\tau$, and $h(x,y) = \dist(x,y)^\tau$.
Then we consider a Poisson $U$-statistic
\begin{align} \label{eq:powerlength}
    F_2^{(\tau)}
    := F_2^{(\tau)}(\eta,\rho)
    = \frac{1}{2} \sum_{(x,y) \in \eta^2_{\neq}}  \dist(x,y)^\tau {\bf 1}\{\dist(x,y) \leq \rho \},
\end{align}
which is known in the literature as the power-weighted edge length 
or length-power functionals.
For $\tau$ varying from $0$ to $1$, it interpolates the total number of edges ($\tau=0$) and the total edge length ($\tau=1$) of the graph $G(\eta,\rho)$.
Concentration bounds for this functional were established by Reitzner, Schulte and Thäle in \cite[Proposition 6.1]{RST17} for the Euclidean setting $\bX=\bR^d$ and with $\Lambda$ being the Lebesgue measure restricted to a compact and convex set $W$.
They obtain this result as an application of \cite[Theorem 1.3]{eichelsbacher2015Moderate} combined with a variance estimate of $F_2^{(\tau)}$.
This result applies to the \textit{thermodynamic regime} where the expected degree of a typical point is fixed, meaning that $\rho$ is of the form $\delta \gamma^{-\frac{1}{d}}$ for some fixed $\delta>0$.
The bound is written as
\begin{equation}\label{eq:ConcentrationRSTGilbertGraph}
    \bP\Bigl(\bigl| F_2^{(\tau)} \bigl( \eta_\gamma, \frac{\delta}{\gamma^{1/d}} \bigr) -\bE F_2^{(\tau)} \bigl( \eta_\gamma , \frac{\delta}{ \gamma^{1/d}} \bigr) \bigr| \ge t \Bigr)
    \leq \exp\Big(-c\min\big\{\gamma^{{2\tau\over d}-1}t^2, \gamma^{\tau\over 3d}t^{1\over 3}, \gamma^{3\tau-d\over 4d}t^{3\over 4}\big\}\Big) ,
    \quad \gamma\geq 1,\ t\geq 0,
\end{equation}
where $c$ is a (non explicit) constant depending on $\tau \geq 0$, $\delta > 0$ and $W$.

\bigskip
The next lemma demonstrates that, under mild assumptions on the measure $\Lambda$, the Poisson $U$-statistic \eqref{eq:UstatisticGraph1} satisfies \eqref{assumptionsGeneral} and, thus, \eqref{eq:fBound} (see Lemma \ref{lem:relAssumptions}) with which one can derive concentration bounds (see Theorem \ref{tm:Summary}).
This covers the settings of the various special cases described above.
We stress that contrary to the results aforementioned this applies not only to the Euclidean setting.
The mild assumptions, given by \eqref{eq:IntensityMeasureCondition1} and \eqref{eq:IntensityMeasureCondition2} in the lemma, are verified whenever $\Lambda(\bX)<\infty$ but can also hold when the space has infinite measure.

\begin{lemma}[Sufficient criterion for \eqref{assumptionsGeneral} for Poisson $U$-statistics of RGG]  \label{lm:A3forGraphs}
    Let $m,n\in\bN$ with $1\leq n \leq m-1$, $\rho,Q>0$ and $h\colon\bX^m\to\bR_+$ a measurable function satisfying $h(x_1,\ldots,x_m) \in [0,Q]$ for any $x_1,\ldots,x_m\in \bX$ with $d(G(\{x_1,\ldots,x_m\},\rho))\leq n$.
    Let $F_m$ be a Poisson $U$-statistic of the form \eqref{eq:UstatisticGraph1}.
    If 
    \begin{equation} \label{eq:IntensityMeasureCondition1}
    \int_{\bX}\Lambda(B_{n\rho}(x))^{m-1}
    \Lambda(\dint x)<\infty,
    \end{equation}
    and 
    \begin{equation} \label{eq:IntensityMeasureCondition2}
    \sup_{x\in\bX}\Lambda(B_{n\rho}(x))<\infty.
    \end{equation}
    are satisfied, then \eqref{assumptionsGeneral} holds with $g(r)={\bf 1} \{r\leq n \rho\}$ and $M=Q/m!$.
    Moreover, if additionally
    \begin{enumerate} 
         \item (induced subgraph count) $h(x_1,\ldots,x_m) = {\bf 1}(G(\{x_1,\ldots,x_m\},\rho) \cong H)$ for some connected graph $H$ with $m$ vertices and combinatorial diameter $d(H)=n$, then \eqref{assumptionsGeneral} holds with $g(r) = {\bf 1} \{r\leq d(H) \rho\} $ and $M= 1/m!$
         \item (included subgraph count) $h(x_1,\ldots,x_m) = | \{ \text{subgraphs $H'$ of $G(\{x_1,\ldots,x_m\},\rho)$} : H' \cong H \} |$ for some connected graph $H$ with $m$ vertices and combinatorial diameter $d(H)=n$, then \eqref{assumptionsGeneral} holds with $g(r) = {\bf 1} \{r\leq d(H) \rho\}$ and $M=1$.
         \item (power-weighted edge length) $m=2$ and $h(x,y) = \dist(x,y)^\tau$ for some $\tau\geq 0$, then \eqref{assumptionsGeneral} holds with $g(r)={\bf 1}\{r\leq \rho\}$ and $M=\rho^{\tau}/2$. 
    \end{enumerate}
\end{lemma}

\begin{proof}
Set $f(x_1,\ldots,x_m) := {1\over m!} h(x_1,\ldots,x_m){\bf 1} \{d(G(\{x_1,\ldots,x_m\},\rho)) \leq n \}$ so that $F_m =\sum_{x\in\eta^m_{\neq}} f(x)$.
Note, that when $d(G(\{x_1,\ldots,x_m\},\rho)) \leq n$ one has $\diam(x_1,\ldots,x_m)  \leq n \rho$.
Thus,
\begin{align*}
    f(x_1,\ldots,x_m) 
    &\leq {Q\over m!} {\bf 1} \{\diam(x_1,\ldots,x_m) \leq n \rho \} 
    \leq {Q\over m!} \prod_{j=2}^{m} {\bf 1}\{\dist(x_1,x_j)\leq n \rho\} .
\end{align*}
From the first inequality above, we see that the second condition from assumption \eqref{assumptionsGeneral} holds with $g(r)={\bf 1}\{r\leq n\rho\}$ and $M=Q/m!$.
From the second inequality we get
\begin{align*}
    \int_{\bX^m}f(x_1,\ldots,x_m)\Lambda^{\otimes m}(\dint (x_1,\ldots,x_m))
    &\leq  {Q\over m!}\int_{
    \bX}\prod_{j=2}^{m}\int_{\bX}{\bf 1}\{\dist(x_1,x_j)\leq n \rho\}\Lambda(\dint x_j)\Lambda(\dint x_1)\\
    &={Q\over m!}\int_{\bX}\big(\Lambda(B_{n \rho}(x))\big)^{m-1}\Lambda(\dint x)<\infty,
\end{align*}
by \eqref{eq:IntensityMeasureCondition1} and, thus, $f\in L^1(\Lambda^m)$.
The remaining condition of \eqref{assumptionsGeneral} holds, since
$$
C(g,\Lambda)=\sup_{x\in\bX}\int_{\bX}{\bf 1}\{\dist(x,y)\leq n\rho\}\Lambda(\dint y)=\sup_{x\in\bX}\Lambda(B_{n\rho}(x))<\infty,
$$
which is a direct consequence of \eqref{eq:IntensityMeasureCondition2}.
This proves the main part of the lemma. The three items follow immediately as special cases.
\end{proof}

Under additional conditions on metric space $\bX$ the second assumption \eqref{eq:IntensityMeasureCondition2} will follow from \eqref{eq:IntensityMeasureCondition1}.
In particular this holds when $\bX$ is a complete simply connected Riemannian manifold of constant sectional curvature $\kappa\in\bR$.

\begin{lemma}
    \label{lm:1condEnough}
    Let $\bX$, $n$ and $\rho$ be as above.
    If there exists a constant $C_{n\rho}$ such that, for any $x\in\bX$, the ball $B_{n\rho}(x)$ can be covered by $C_{n\rho}$ sets of diameter less than $\rho$, then \eqref{eq:IntensityMeasureCondition1} implies \eqref{eq:IntensityMeasureCondition2}.
\end{lemma}
\begin{proof}
    We slightly adapt the steps of the proof of a similar statement at the beginning of Section 6 in \cite{bachmann2016concentration}.
    Note that $C_{n\rho}\geq 1$.
    By the pigeonhole principle, for any $y\in\bX$ we can choose a set $S_y \subset B_{n\rho}(y)$ satisfying: (i) $S_y \subset B_{n\rho}(x)$ for any $x\in S_y$, and (ii) $\Lambda(S_y) \geq C_{n\rho}^{-1} \Lambda(B_{n\rho}(y))$.
    Now,
    \begin{align*}
        \int_{\bX} \Lambda(B_{n\rho}(x))^{m-1} \Lambda(\dint x)
        &\geq \sup_{y\in\bX} \int_{S_y} \Lambda(B_{n\rho}(x))^{m-1} \Lambda(\dint x)
        \geq C_{n\rho}^{-m+1} \sup_{y\in\bX} \int_{S_y} \Lambda(B_{n\rho}(y))^{m-1} \Lambda(\dint x)
        \\&= C_{n\rho}^{-m+1} \sup_{y\in\bX} \Lambda(S_y) \Lambda(B_{n\rho}(y))^{m-1} 
        \geq C_{n\rho}^{-m} \sup_{y\in\bX} \Lambda(B_{n\rho}(y))^{m} . \qedhere
    \end{align*}
\end{proof}

Lemma \ref{lm:A3forGraphs} in particular implies that $F_m$ satisfies \eqref{eq:fBound} due to Lemma \ref{lem:relAssumptions} and, thus, one can apply Theorem \ref{tm:Summary} as well, as Theorem \ref{tm:momentConcRes_NEW}, Proposition \ref{prop:ChebyshevCantelli} and Proposition \ref{prop:LowerTail}, leading to concentration bounds with explicit constants, depending on $Q$, $\Lambda$, $H$, $m$ and $\rho$.
For simplicity of the representation we will state the results only for a few special cases.

\subsection{Random geometric graph in constant curvature spaces}

Let $d\geq 2$.
In this section, we assume that $\bX=\bM_{\kappa}=\bM_{\kappa,d}$ is a $d$-dimensional complete simply connected Riemannian manifold of constant sectional curvature $\kappa\in\bR$ with fixed origin $o$ and corresponding Riemannian metric inducing the distance function $\dist = \dist_{\kappa}$. 
Note that $\bM_{0}=\bR^d$, $\bM_{-1}=\bH^d$ (hyperbolic space) and $\bM_{1}=\bS^d$ (unit sphere). 
Further let $\cH_{\kappa}^d$ be the measure induced by the Riemannian metric on $\bM_{\kappa}$. 
In particular $\cH_{\kappa}^d$ is isometry invariant, and when $\kappa > 0$ the space $\bM_{\kappa}$ is compact with $\cH_{\kappa}^d(\bM_{\kappa})=\kappa^{-d/2}\omega_d$ and $\diam(\bM_{\kappa})=\kappa^{-1/2}\pi$, where $\omega_d={2\pi^{d\over 2}/\Gamma({d\over 2})}$.

We set $o\in\bM_\kappa^d$ to be an arbitrary point, which we refer to as zero or the origin.
For $r\geq 0$, we denote by $B_r^d=B_r:=B_r(o)$ the $d$-dimensional ball of radius $r$ around zero, where in case of $\kappa>0$ we additionally require $r<{\pi\over 2\sqrt{\kappa}}$. 
It holds that \cite[Equation (17.47)]{SantaloBook}
$$
\cH_{\kappa}^d(B_r)=\omega_d\int_{0}^r({\rm sn}_{\kappa}(s))^{d-1}\dint s,
\quad \text{ where } \quad 
{\rm sn}_{\kappa}(s)=\begin{cases}
{\sin(\sqrt{\kappa}s)\over \sqrt{\kappa}},\, &\kappa>0;\\
s,\, &\kappa = 0;\\
{\sinh(\sqrt{-\kappa} s)\over \sqrt{-\kappa}},\, &\kappa<0.
\end{cases}
$$
Next, we observe that, for any given curvature, the volume of balls of small radius $r>0$ is of order $r^d$.
To see this, first we note that, for $0\leq x \leq \pi/2$, one has $\frac{s}{2} \leq \frac{2}{\pi} s \leq \sin x \leq \sinh x \leq \frac{\sinh (\pi/2)}{\pi/2} x \leq  2 x$.
Thus $\frac{1}{2} s \leq {\rm sn}_{\kappa}(s) \leq 2 s $ for $0\leq s \leq \frac{\pi}{2 \sqrt{|\kappa|}}$, with the convention $\frac{1}{0}=\infty$. 
Therefore, for any $\kappa\in\bR$, and for $0 \leq r \leq r_0(\kappa) := |\kappa|^{-1/2} $, we have $\cH_{\kappa}^d(B_r)= \Theta(r^d)$. 

Further, we assume that the measure $\Lambda$ is non-zero and absolutely continuous with respect to $\cH_{\kappa}^d$, with Radon-Nikodym derivative $\lambda$.
That means that $\Lambda$ is a $\sigma$-finite measure on $\bM_{\kappa}$ given by
\begin{equation}\label{eq:IntensityMeasureDensity}
\Lambda(\cdot)=\int_{\bM_{\kappa}}{\bf 1}\{x\in \cdot\}\lambda(x)\cH_{\kappa}^d(\dint x),
\end{equation}
where $\lambda\colon\bM_{\kappa}\to [0,\infty)$ is some measurable function.

\subsubsection{Sugbraph count}
For a given connected graph $H$ with $m \geq 2$ vertices, we recall that the included subgraph count of the random geometric graph $G(\eta,\rho)$, given by \eqref{eq:includedgraphs}, is 
$$
F_{\subset}^{H}
= F_{\subset}^{H}(\eta,\rho)
= {1\over m!}\sum_{(x_1,\ldots,x_m)\in\eta^m_{\neq}} |\{ \text{subgraphs $H'$ of $G(\{x_1,\ldots,x_m\},\rho)$} : H' \cong H \}|.
$$

\begin{theorem}[Subgraph count]
\label{tm:SubgraphCounts}
    Let $m\ge 2$, $\rho> 0$ and $\kappa\in\bR$.
    If $\kappa>0$ we additionally assume that $\rho<\pi/2\sqrt{\kappa}$.
    Let $H$ be a connected graph with $m$ vertices and diameter $d(H)=n$.
    Let $\Lambda$ be a measure on $\bM_{\kappa}$, absolutely continuous with respect to $\cH_{\kappa}^d$ with Radon-Nikodym derivative $\lambda$ satisfying
    $$
    0<\ell_m:=\int_{\bM_{\kappa}}\lambda(x)^m\cH_{\kappa}^d(\dint x)<\infty.
    $$
    The following holds:
    \begin{enumerate}
        \item The constants 
        \begin{align*}
            C_1(\rho) 
            & :=\sup_{x\in\bM_{\kappa}}\Lambda(B_{n\rho}(x))^m,
            &
            C_2(\rho) 
            & :=\int_{\bM_{\kappa}}\Lambda(B_{n\rho}(x))^{m-1}\Lambda(\dint x) ,
            &
            C_3(\rho)
            &:= \int_{\bM_{\kappa}}\Lambda(B_{\rho/2}(x))^{2m-2}\Lambda(\dint x),
        \end{align*}
        are in $(0,\infty)$, with $C_2(\rho) \leq 2\big[\cH_{\kappa}^d(B_{n\rho})\big]^{m-1}\ell_m $.
    
        \item Set $s\in[0,1]$ and let $\beta_1 = C_1(\rho)^{\frac{1}{m}} \max\Big(1,{C_2(\rho)\over C_1(\rho)}\Big)^{\frac{s}{m}}$,
        $\beta_2 = \max\Big(1,{C_2(\rho)\over C_1(\rho)}\Big)^{\frac{1-s}{2}}$,
        and 
        \[
        \Cr{c_27}' 
        = 2^{-17m-4}m^{-2m-3} \min\left( 1 , \frac{1}{[(m-1)!]^2} C_3(\rho) \beta_2^{-2}\beta_1^{1-2m} \right).
        \]
        Then, for any $t\geq 0$ and $\gamma \geq 8m/\beta_1$,
        \begin{align*}
            \bP(|F_{\subset}^{H}-\bE F_{\subset}^{H}| \geq t)
            &\leq 2\exp \left( - \Cr{c_27}' \left(\frac{t}{\beta_2}\right)^{\frac{1}{m}} \min \left( \left( \frac{t}{\beta_2 (\beta_1 \gamma)^m} \right)^{2-\frac{1}{m}} , 1 + \log_+ \left( \frac{t}{ \beta_2 (\beta_1 \gamma)^m} \right) \right)  \right).
        \end{align*}
        \end{enumerate}
\end{theorem}

\begin{remark}
In case of induced subgraphs count, it can happen that $F_{=}^{H}$ is trivially zero, for example, for $H=S_n$ with $n$ large enough, where $S_n$ denotes the star graph which is a tree with one root and $n$ leaves.  
However the concentration bound of Theorem \ref{tm:SubgraphCounts} holds also for $F_{=}^{H}$ if one changes $C_3(\rho)$ by 
\begin{align*}
C_4(\rho)
&:= {[(m-1)!]^2} \|f_1^{\mathrm{ind}}\|_{L_2(\Lambda)}^2
\\&=\int_{\bM_{\kappa}}\Big(\int_{(\bM_{\kappa})^{m-1}} {\bf 1}\{G (\{x_1,\ldots,x_{m}\},\rho)\cong H \} \Lambda(\dint x_1)\cdots\Lambda(\dint x_{m-1})\Big)^2\Lambda(\dint x_m) ,
\end{align*}
assuming that this quantity is strictly positive, which is the case as soon as $F_{=}^{H}$ is not trivially zero. 
The proof applies without any changes, apart from the very last estimate.
\end{remark}

\begin{remark}
    Let us point out that the concentration inequalities for the functional $F_{\subset}^{H}$ have been previously obtained in \cite[Theorem 1.1]{BR18} for the Euclidean setting $\kappa=0$.
    In comparison with our results the estimates in \cite{BR18} are of the form $\exp(-{\rm const}\cdot t^{1\over m})$ for big $t$.
    We improve this by adding a logarithmic factor and extending the bounds to the hyperbolic and spherical settings as well.
    Apart from this, our bounds are of the same order in $t$ and $\gamma$, but it appeared to be difficult to compare the corresponding constants and, respectively, dependence of $\rho$.
\end{remark}

\begin{proof}[Proof of Theorem \ref{tm:SubgraphCounts}]       
    We start by showing that $C_2(\rho) \leq 2\big[\cH_{\kappa}^d(B_{n\rho})\big]^{m-1}\ell_m <\infty $.
    We note that
    \begin{align*}
        C_2(\rho)
        &=\int_{\bM_{\kappa}}\Bigl[\int_{\bM_{\kappa}}{\bf 1}\{y\in B_{n\rho}(x)\}\lambda(y)\cH_{\kappa}^d(\dint y) \Bigr]^{m-1}\lambda(x)\cH_{\kappa}^d(\dint x)\\
        &\leq \int_{\bM_{\kappa}}\Big[\Big(\int_{\bM_{\kappa}}{\bf 1}\{y\in B_{n\rho}(x)\}\cH_{\kappa}^d(\dint y)\Big)^{m-2\over m-1}\\
        &\qquad\qquad\times\Big(\int_{\bM_{\kappa}}{\bf 1}\{y\in B_{n\rho}(x)\}\lambda(y)^{m-1}\cH_{\kappa}^d(\dint y)\Big)^{1\over m-1}\Big]^{m-1}\lambda(x)\cH_{\kappa}^d(\dint x)\\
        &=\big[\cH_{\kappa}^d(B_{n\rho})\big]^{m-2}\int_{\bM_{\kappa}}\int_{\bM_{\kappa}}{\bf 1}\{y\in B_{n\rho}(x)\}\lambda(x)\lambda(y)^{m-1}\cH_{\kappa}^d(\dint x)\cH_{\kappa}^d(\dint y),
    \end{align*}
    where in the second step we applied Hölder's inequality to the inner integral with $p=m-1$ and $q=(m-1)/(m-2)$ and in the third equation we used isometry invariance of the measure $\cH_{\kappa}^d$. Now applying the inequality $ab^{m-1}\leq a^m+b^m$, $a,b\ge 0$, and noting that ${\bf 1}\{y\in B_{n\rho}(x)\}$ is equivalent to ${\bf 1}\{x\in B_{n\rho}(y)\}$ we arrive at
    \begin{align*}
        C_2(\rho)
        &\leq \big[\cH_{\kappa}^d(B_{n\rho})\big]^{m-2}\Big(\int_{\bM_{\kappa}}\int_{\bM_{\kappa}}{\bf 1}\{y\in B_{n\rho}(x)\}\lambda(x)^{m}\cH_{\kappa}^d(\dint x)\cH_{\kappa}^d(\dint y)\\
        &\qquad\qquad+\int_{\bM_{\kappa}}\int_{\bM_{\kappa}}{\bf 1}\{x\in B_{n\rho}(y)\}\lambda(y)^{m}\cH_{\kappa}^d(\dint x)\cH_{\kappa}^d(\dint y)\Big)\\
        &=2\big[\cH_{\kappa}^d(B_{n\rho})\big]^{m-1}\int_{\bM_{\kappa}}\lambda(x)^m\cH_{\kappa}^d(\dint x)<\infty.
    \end{align*}
    We have shown $C_2(\rho)<\infty$ which is precisely \eqref{eq:IntensityMeasureCondition1}.
    By Lemma \ref{lm:1condEnough}, it implies  \eqref{eq:IntensityMeasureCondition2} which says that $C_1(\rho)<\infty$. Since \eqref{eq:IntensityMeasureCondition1} and \eqref{eq:IntensityMeasureCondition2} hold, Lemma \ref{lm:A3forGraphs} gives that \eqref{assumptionsGeneral} holds with $g(r) = {\bf 1} \{r\leq d(H) \rho\}$, $M=1$ and $C(g,\Lambda) = C_1(\rho)^{1\over m}$.
    Therefore, according to Lemma \ref{lem:relAssumptions}, $F_{\subset}^{H}$ satisfies \eqref{eq:fBound} with  
    $$
    q=0, \quad 
    \beta_0 = 1, \quad
    \beta_1 = C_1(\rho)^{1\over m} \max\Big(1,{C_2(\rho)\over C_1(\rho)}\Big)^{\frac{s}{m}}, \quad
    \beta_2 = \max\Big(1,{C_2(\rho)\over C_1(\rho)}\Big)^{\frac{1-s}{2}} ,
    $$ 
    where we additionally used that
    $$
    \|f\|_{L_1(\Lambda)}
    \leq \int_{\bM_{\kappa}}\Lambda(B_{n\rho}(x))^{m-1}\Lambda(\dint x)
    = C_2(\rho).
    $$
    
    Thus, we can apply Theorem \ref{tm:Summary} as stated at the beginning of Section \ref{sec:discussion}, with the parameters $q$, $\beta_0$, $\beta_1$ and $\beta_2$ as above and with 
    $$ f_1(x_m)
    = {1\over (m-1)!} \int_{(\bM_{\kappa})^{m-1}} |\{ \text{subgraphs $H'$ of $G(\{x_1,\ldots,x_m\},\rho)$} : H' \cong H \}| \, \Lambda(\dint x_1)\cdots\Lambda(\dint x_{m-1}) .$$
    To show the second part of the theorem, it remains only to observe that $\Cr{c_27}' < \Cr{c_27} $, which follows from
    \begin{align*}
        \|f_1\|_{L_2(\Lambda)}^2
        &\ge {1\over [(m-1)!]^2}\int_{\bM_{\kappa}}\Big(\int_{(\bM_{\kappa})^{m-1}}{\bf 1}\{x_1,\ldots,x_{m-1}\in B_{\rho/2}(x_m)\}\Lambda(\dint x_1)\cdots\Lambda(\dint x_{m-1})\Big)^2\Lambda(\dint x_m)\\
        &= \frac{1}{[(m-1)!]^2}C_3(\rho) .
    \end{align*}
    By \eqref{eq:fkA2Bound1} this also shows that 
    $
    C_3(\rho)\leq (m!)^2\beta_1^{2m-1}\beta_2^2<\infty,
    $
    concluding the proof of the first part of the theorem as well.
\end{proof}

The above theorem applies for rather general family of measures $\Lambda$. For this reason the  expressions for the constant $C_1(\rho)$, $C_2(\rho)$ and $C_3(\rho)$ are quite complex in general.
In particular their dependence on the parameter $\rho$ is obscure.
On the other hand for some specific choices of measure $\Lambda$, these constants can be made more explicit.
In what follows we consider one example of such measure. 

For a compact subset $W\subset \bM_{\kappa}$, we denote its inradius by
$$
r(W):=\max\{r>0\colon B_{r}(x)\subset W\text{ for some }x\in \bM_{\kappa}\} ,
$$
and the restriction of $\cH_\kappa^d$ to $W$ by
\begin{equation}\label{eq:LambdaW}
\Lambda(\cdot ) = \Lambda_W(\cdot )=\cH^d_{\kappa}(W\cap \cdot) .
\end{equation}

\begin{corollary}
\label{cor:RGGsubgraphConstantCurve}
    Let $m\ge 2$, $\rho> 0$, $\kappa\in\bR$ and $W\subset \bM_{\kappa}$ be a compact set with inradius $r(W)>0$.
    If $\kappa > 0$, we additionally assume that $W$ does not contain antipodal points.
    Let $H$ be a connected graph with $m$ vertices and diameter $d(H)=n$.
    Let $\Lambda = \Lambda_W$ be defined by \eqref{eq:LambdaW}.
    Then for any $\rho\leq n^{-1}r(W)$, $s\in[0,1]$,  $t\geq 0$, and $\gamma \geq 8m/\beta_1''$, we have
    \begin{align*}
        \bP(|F_{\subset}^{H}-\bE F_{\subset}^{H}| \geq t)
        &\leq 2\exp \left( - \Cr{c_27}'' \left(\frac{t}{\beta_2''}\right)^{\frac{1}{m}} \min \left( \left( \frac{t}{\beta_2''(\beta_1'' \gamma)^m} \right)^{2-\frac{1}{m}} , 1 + \log_+ \left( \frac{t}{\beta_2''(\beta_1'' \gamma)^m} \right) \right)  \right).
    \end{align*}
    where 
    $\beta_1'' 
    = \cH^{d}_{\kappa} (B_{n\rho})^{1-\frac{s}{m}} \cH_{\kappa}^d(W)^{\frac{s}{m}}$,
    $\beta_2''
    = \left(\frac{\cH_{\kappa}^d(W)}{\cH^{d}_{\kappa}(B_{n\rho})} \right)^{\frac{1-s}{2}} $
    and
    \[
    \Cr{c_27}''
    = 2^{-17m-4}m^{-2m-1}(m!)^{-2}\Big({\cH_{\kappa}^d(B_{\rho/2})\over \cH_{\kappa}^d(B_{n\rho})}\Big)^{2m-2}\Big({\cH_{\kappa}^d(B_{n\rho})\over \cH_{\kappa}^d(W)}\Big)^{s(1-{1\over m})}\Big({\cH_{\kappa}^d(B_{r(W)/2})\over \cH_{\kappa}^d(W)}\Big).
    \]
\end{corollary}

\begin{proof}
    Without loss of generality we assume that $B_{r(W)}\subset W$. 
    First we note that the measure $\Lambda_W$ can be represented in the form \eqref{eq:IntensityMeasureDensity} with $\lambda(x)={\bf 1}\{x\in W\}$,
    and thus $\ell_m=\cH_{\kappa}^d(W)<\infty$, where $\ell_m$ is as defined in Theorem \ref{tm:SubgraphCounts}. 
    We also note that for $\kappa>0$, we have $\rho \leq n^{-1} r(W) \leq r(W) < r(\bM_\kappa) / 2 = \pi/2 \sqrt{\kappa}$ because $W$ does not contain antipodal points.
    Hence Theorem \ref{tm:SubgraphCounts} applies and it remains only find an upper bound for the terms $\beta_1$ and $\beta_2$, which means to bound $C_1(\rho)$, $C_2(\rho)$ from above and $C_3(\rho)$ from below.

    Note that, since $n\rho\leq r(W)$ and $B_{r(W)}\subset W$, we get
    $$ C_1(\rho)
    = \sup_{x\in\bM_{\kappa}} \cH_{\kappa}^d(B_{n\rho}(x)\cap W)^m
    = \cH^{d}_{\kappa} (B_{n\rho})^m .$$
    We also note that $C_2(\rho)
    \leq \big[\cH_{\kappa}^d(B_{n\rho})\big]^{m-1}\cH_{\kappa}^d(W)$ and, thus, 
    $$
    \max\Big(1,{C_2(\rho)\over C_1(\rho)}\Big)
    \leq \max\Big(1,{\cH_{\kappa}^d(W)\over \cH^{d}_{\kappa}(B_{n\rho})}\Big)
    = \frac{\cH_{\kappa}^d(W)}{\cH^{d}_{\kappa}(B_{n\rho})}.
    $$
    Therefore,
    \begin{align*}
        \beta_1 
        & = C_1(\rho)^{\frac{1}{m}} \max\Big(1,{C_2(\rho)\over C_1(\rho)}\Big)^{\frac{s}{m}} 
        \leq
        \cH^{d}_{\kappa} (B_{n\rho})^{1-\frac{s}{m}} \cH_{\kappa}^d(W)^{\frac{s}{m}} ,
        &\beta_2 
        &= \max\Big(1,{C_2(\rho)\over C_1(\rho)}\Big)^{\frac{1-s}{2}}
        \leq \left(\frac{\cH_{\kappa}^d(W)}{\cH^{d}_{\kappa}(B_{n\rho})} \right)^{\frac{1-s}{2}} .
    \end{align*}
    Finally for $C_3(\rho)$ we conclude
    \begin{align*}
        C_3(\rho)
        & = \int_{\bM_{\kappa}} \cH_{\kappa}^d(B_{\rho/2}(x)\cap W)^{2m-2}{\bf 1}\{x\in W\} \cH_{\kappa}^d(\dint x)\\
        &\ge \int_{\bM_{\kappa}} \cH_{\kappa}^d(B_{\rho/2}(x))^{2m-2}{\bf 1}\{x\in B_{r(W)-\rho/2}\} \cH_{\kappa}^d(\dint x)\\
        &\ge \cH_{\kappa}^d(B_{\rho/2})^{2m-2} \cH_{\kappa}^d (B_{r(W)/2}),
    \end{align*}
    since $r(W)-\rho/2\ge r(W)/2$.
    
\end{proof}

We will compare our result with Corollary 4.2 of \cite{ST23}, which applies to the Euclidean setting.
For simplicity, we will make a number of assumptions which will simplify the constants in both results.
First, we take $\kappa=0$ and work in a spherical window $B_r$.
Second, we will consider only the setting when $\gamma \rho^d$ is large enough. 

\begin{corollary}
\label{cor:RGGsubgraphZeroCurve}
    Let $m\ge 2$, $\rho> 0$, $r>0$ and $W = B_r \subset \bR^d$.
    Let $H$ be a connected graph with $m$ vertices and diameter $d(H)=n$.
    Let $\Lambda$ be the restriction of the Lebesgue measure to $B_r$.
    Set
    \begin{align*}
        \Cl{c_33} &:= \frac{8m}{\kappa_d n^d} ,
        & \Cl{c_32} &:= m^{-4m-2}2^{-15m-2md+d-6}\kappa_d^{1-2m}n^{-d(4m-4)},
        & \Cl{c_34} &:=  \kappa_d^mn^{dm-{d\over 2}},\\ 
         \Cl{c_35} &:= e\kappa_d^{-m}n^{-dm+{d\over 2}}, & \Cl{c_36} &:= m^{-4m-2}2^{-15m-2md+d-6}n^{-d(2m-2-{1\over 2m})},
    \end{align*}
    where $\kappa_d=\pi^{d\over 2}/\Gamma(d/2+1)$ denotes the Lebesgue measure of a $d$-dimensional Euclidean ball.
    \begin{enumerate}
        \item Then for any $\rho\leq n^{-1}r$, $t\geq 0$ and $\gamma \geq \Cr{c_33} \rho^{-d}$, 
        \begin{align*}
            \bP(|F_{\subset}^{H}-\bE F_{\subset}^{H}| \geq t)
            &\leq  \begin{cases}
            2\exp \left( - \Cr{c_32} \left({\rho\over r}\right)^{d}\left(  \rho^d \gamma \right)^{-2m+1}  t^2 \right) ,
            & t \leq \Cr{c_34}\left(\frac{r}{\rho} \right)^{\frac{d}{2}}(\rho^{d} \gamma)^m ,
            \\
            2\exp \left( - \Cr{c_36}\left(\frac{\rho}{r}\right)^{{d\over 2m}} t^{1\over m} \log \left( \Cr{c_35} \left(\frac{\rho}{r} \right)^{\frac{d}{2}} \frac{t}{ (\rho^{d}  \gamma)^m} \right) \right) ,
            & t \geq \Cr{c_34} \left(\frac{r}{\rho} \right)^{\frac{d}{2}} (\rho^{d} \gamma)^m .
        \end{cases}
        \end{align*}
        \item $\bV F_{\subset}^{H} \geq \left({\kappa_d\over m2^{d-1}}\right)^{2m-1} (\rho^d\gamma)^{2m-1} \left(\frac{r}{\rho}\right)^d$.
        \item
        For any $\rho\leq n^{-1}r$, $z\geq 0$ and $\gamma \geq \Cr{c_33} \rho^{-d}$ we have 
        \begin{align*}
            \bP(|F_{\subset}^{H}-\bE F_{\subset}^{H}| \geq z \sqrt{\bV F_{\subset}^{H}})
            &\leq \begin{cases}
                2 \exp \left( - \Cr{c_37} z^2 \right) ,
                & z  \leq \Cr{c_41} (\rho^d \gamma)^{1\over 2} ,
                \\
                2 \exp \left( - \Cr{c_38} (\rho^{d} \gamma)^{1-{1\over 2m}} z^{1\over m} \log \left( \Cr{c_40} (\rho^d \gamma)^{-{1\over 2}} z \right) \right) ,
                & z  \geq \Cr{c_41} (\rho^d \gamma)^{1\over 2} .
            \end{cases}
        \end{align*}
        where
        \begin{align*}
            \Cl{c_37} &:=  m^{-6m-1}2^{-13m-4md+2d-7}n^{-d(4m-4)},
            & \Cl{c_40} &= e\kappa_d^{-{1\over 2}}({2^{d-1}m n^d})^{-m+{1\over 2}} ,\\
             \Cl{c_38} &=  m^{-4m-3+{1\over 2m}}2^{-15m-2md-5}\kappa_d^{1-{1\over 2m}}n^{-d(2m-2-{1\over 2m})} ,
            & \Cl{c_41} &=  \kappa_d^{1\over 2}({2^{d-1}m n^d})^{m-{1\over 2}}.
        \end{align*}
    \end{enumerate}
\end{corollary}

\begin{proof}

We start by noting that by Stirling's formula we have $m!\leq \sqrt{2\pi}m^{m+{1\over 2}}e^{-m}\leq m^{m+{1\over 2}}2^{-m+1}$.

    For the first point, we apply Corollary \ref{cor:RGGsubgraphConstantCurve} with $s=0$ and observe that
    \begin{align*}
        \beta_1'' 
        &= \cH^{d}_{0} (B_{n\rho})
        = \kappa_d (n\rho)^{d},
        & \beta_2''
        &= \left(\frac{\cH_{0}^d(B_r)}{\cH^{d}_{0}(B_{n\rho})} \right)^{\frac{1}{2}}
        = \left(\frac{r}{n\rho} \right)^{\frac{d}{2}} , &\beta_2''(\beta_1'' \gamma)^m
        &= \kappa_d^mn^{dm-{d\over 2}} \left(\frac{r}{\rho} \right)^{\frac{d}{2}} (\rho^{d}  \gamma)^m,
    \end{align*}
    and
    \begin{align*}
        \Cr{c_27}''
        &=  (m!)^{-2}m^{-2m-1}2^{-17m-2md+d-4}n^{-d(2m-2)}\ge m^{-4m-2}2^{-15m-2md+d-6}n^{-d(2m-2)}.
    \end{align*}
    Therefore, for $t \geq 0$ and 
    $\gamma\geq 8 m / \beta_1'' 
    = \Cr{c_33} \rho^{-d}$ we have
    \begin{align*}
        \bP(F_{\subset}^{H}-\bE F_{\subset}^{H} \geq t)
        &\leq \exp \left( - \Cr{c_27}'' \left(\frac{t}{\beta_2''}\right)^{\frac{1}{m}} \min \left( \left( \frac{t}{\beta_2''(\beta_1'' \gamma)^m} \right)^{2-\frac{1}{m}} , 1 + \log_+ \left( \frac{t}{\beta_2''(\beta_1'' \gamma)^m} \right) \right)  \right)       
        \\&= \begin{cases}
            \exp \left( - \Cr{c_32} \left({r\over \rho}\right)^{-d}\left(  \rho^d \gamma \right)^{-2m+1}  t^2 \right) ,
            & t \leq \Cr{c_34}\left(\frac{r}{\rho} \right)^{\frac{d}{2}}(\rho^{d} \gamma)^m ,
            \\
            \exp \left( - \Cr{c_36}\left(\frac{r}{\rho}\right)^{-{d\over 2m}} t^{1\over m} \log \left( \Cr{c_35} \left(\frac{r}{\rho} \right)^{-\frac{d}{2}} \frac{t}{ (\rho^{d}  \gamma)^m} \right) \right) ,
            & t \geq \Cr{c_34} \left(\frac{r}{\rho} \right)^{\frac{d}{2}} (\rho^{d} \gamma)^m .
        \end{cases}
    \end{align*}

    \medskip
    For the variance estimate of the second point, we recall that $\bV(F) \geq \gamma^{2m-1} \|f_1\|^2_{L^2(\Lambda)}$.
    But we have seen, at the end of the proof of Theorem \ref{tm:SubgraphCounts} that $\|f_1\|^2_{L^2(\Lambda)} \geq \frac{m^2}{(m!)^2}C_3(\rho)\ge m^{1-2m}2^{2m-1}C_3(\rho)$, where 
    $$C_3(\rho) 
    \geq \cH_{0}^d(B_{\rho/2})^{2m-2} \cH_{0}^d (B_{r/2}) 
    = \left({\kappa_d\over 2^d}\right)^{2m-1} \rho^{d(2m-1)}\left({r\over \rho}\right)^{d} , $$
    as seen in the proof of Corollary \ref{cor:RGGsubgraphConstantCurve}.
    The estimate follows.

    The bound of the third point follows from the two first points.
\end{proof}

\begin{remark}
    In Corollary 4.2 of \cite{ST23}, Schulte and Thäle obtain a similar concentration bound for functionals which are linear combinations of (induced or included) subgraph counts in the Euclidean random geometric graph.
    In order to simplify the comparison with our result, we assume that $\gamma \kappa_d \rho^d \geq 1$.
    Under this assumption, the condition $(4.2)$ from \cite{ST23} reads as $\bV F_{\subset}^H \geq v \gamma^{2m-1} (\kappa_d \rho^d)^{2m-2}$.
    With the variance estimate $\bV F_{\subset}^{H} \geq \left(\frac{\kappa_d }{m2^{d-1} }\right)^{2m-1} \gamma^{2m-1} \rho^{d(2m-2)} r^d$ of Corollary \ref{cor:RGGsubgraphZeroCurve}, this is satisfied for 
    $ v 
    = \frac{ \kappa_d r^d}{ (2^{d-1} m)^{2m}}
    $.    
    With the help of Table \ref{tab:dic}, applying the result from \cite{ST23} gives that $F_{\subset}^H$ satisfies a concentration inequality ${\bf CI}(m-1 , \tau )$ with 
    \begin{align*}
        \tau 
        &= \frac{\sqrt{v \gamma \min(1,\gamma \kappa_d \rho^d)^{m-1}}}{ m^{3m} \max(1, \cH_0^d(B_r)/v) } 
        =  \frac{\sqrt{\frac{ \kappa_d r^d}{ (2^{d-1} m)^{2m-1}} \gamma}}{ m^{3m} (2^{d-1} m)^{2m-1} } 
        =  \frac{\sqrt{\kappa_d r^d \gamma}}{ 2^{3m(d-1)-{2\over 2}(d-1)} m^{6m-{3\over 2}} } ,
    \end{align*} 
    meaning that for any $\gamma\geq (\kappa_d \rho^d)^{-1}$ and any $z\geq 0$, one has
    \begin{align*}
        \bP \left( |F_{\subset}^{H}-\bE F_{\subset}^{H}| \geq z \sqrt{\bV F_{\subset}^{H}} \right)
        &\leq 
        \begin{cases}
             2 \exp \left( - 2^{-m-2} z^2 \right) & z^{2m-1} \leq  2^{m^2} \tau 
             \\ 
             2 \exp \left( - 2^{-2} (z \tau)^{1/m} \right) & z^{2m-1} \geq  2^{m^2} \tau 
        \end{cases} 
        \\&\leq
        \begin{cases}
            2 \exp \left( - 2^{-m-2} z^2 \right) 
            & z \leq \Cr{c_42} (r^{d} \gamma)^{1/(4m-2)} ,
            \\
            2 \exp \left( - 2^{-3d-2} m^{-6} (\kappa_dr^{d} \gamma)^{\frac{1}{2m}} z^{\frac{1}{m}} \right) 
            & z \geq \Cr{c_42} (r^{d} \gamma)^{1/(4m-2)} .
        \end{cases}
    \end{align*}    
    where $\Cl{c_42} = (2^{d-1}m)^{-{3\over 2}}( 2^{m^2} \kappa_d^{1\over 2} m^{-3m} )^{1/(2m-1)}$.
    
    Let us now point out that in any regime where $r$ (observation window) stays fixed, $\gamma\to\infty$, and $\rho^d\leq c \gamma^{-1}$, in the above bound we observe the Gaussian tail for the large range of $z$, which eventually extends to $[0,\infty)$ as $\gamma\to\infty$. This means that in this case the above bound is better then the estimate from Corollary \ref{cor:RGGsubgraphZeroCurve}, which provides Gaussian tail only for $z$ from some fixed interval.
    
    On the other hand if both $r$ and $\rho$ stays fixed and $\gamma\to\infty$ the bounds from Corollary \ref{cor:RGGsubgraphZeroCurve} are of the better quality, since they give Gaussian tail for $0<z<{\rm const}\cdot \gamma^{1\over 2}$ in contrast to the above bounds, providing Gaussian tail only for $0<z<{\rm const}\cdot \gamma^{1\over 4m-2}$.
    
    There are certain regimes where both bounds agree, apart from the fact that bound from Corollary \ref{cor:RGGsubgraphZeroCurve} has an additional logarithmic factor for $z\to\infty$, which makes it better in this case. For example, this is the case when $(\rho^d\gamma)^{2m-1}=r^d\gamma$.

    \begin{table}
        \centering
        \begin{tabular}{ccccccccccccc}
            Notation from \cite{ST23}: & $W$ & $k$ & $a_1$ & $\diamondsuit_1$ & $G_1$ & $q_1$ & $S_n$ & $p$ & $q$ & $a$ & $t_n$ & $r_n$ \\
            Our notation: & $B_r$ & $1$ & $1$ & $\subset$ & $H$ & $m$ & $F_{\subset}^H$ & $m$ & $m$ & $1$ & $\gamma$ & $\rho$ 
        \end{tabular}
        \caption{Dictionary to apply Corollary 4.2 of \cite{ST23} to our setting.}
        \label{tab:dic}
    \end{table}
\end{remark}

\begin{remark}
    Similarly to the observation made in \cite[Remark 4.4]{ST23}, we point out that random geometric graphs generalize to the random connection model, and that our methods and results are expected to extend naturally. In random connection model, for every pair of Poisson points an edge is drawn with a probability which is a function of these two points, and where all the edges are drawn independently. 
    This additional randomness excludes functionals of the random connections model from the Poisson $U$-statistics framework.
    
    However, in Proposition 5.1 of the recent preprint \cite{liu2023normal}, Liu and Privault establish (non-centred) moment and cumulant formulas for the subgraph count in this model. They have a similar form as the one obtained for the random geometric graph, with an additional factor corresponding to the additional randomness of this model.
    We expect that centred moments formulas extend in a similar way to the random connection model and that, once this is done, one can apply our concentration bounds.
\end{remark}

\subsubsection{Power-weighted edge length}
Finally we consider the power weighted edge length $F_2^{(\tau)}$ defined by \eqref{eq:powerlength}.

\begin{theorem}
\label{thm:RGGnonEuclidean}
    Let $\kappa\in \bR$, $d\ge 2$, and $W\subset \bM_{\kappa}$ be a compact set with inradius $r(W)>0$. If $\kappa > 0$, we additionally assume that $W$ does not contain antipodal points.
    Let $\eta$ be a Poisson point process with intensity measure $\gamma\Lambda_W$, defined by \eqref{eq:LambdaW}. 
    Set
    \begin{align*}
        s & \in[0,1] ,
        &\beta_1 &= \cH_{\kappa}^d(W)^{\frac{s}{2}} \cH_{\kappa}^d(B_{\rho})^{1-\frac{s}{2}} ,
        & \beta_2 &= \frac{\rho^\tau}{2} \left( \frac{\cH_{\kappa}^d(W)}{\cH_{\kappa}^d(B_{\rho})} \right)^{\frac{1-s}{2}} .
    \end{align*}
    Then for any  $\tau\ge 0$, $0 < \rho\leq r(W)$, $\gamma \geq 16 \beta_1^{-1} $ and any $t\geq 0$,  
    \begin{align*}
        \bP(|F_2^{(\tau)}-\bE F_2^{(\tau)}| \geq t)
        &\leq 2\exp \left( - \Cr{c_27}''' \left(\frac{t}{\beta_2} \right)^{\frac{1}{2}} \min \left( \left( \frac{t}{\beta_2 (\beta_1 \gamma)^2} \right)^{\frac{3}{2}} , 1 + \log_+ \left( \frac{t}{\beta_2 (\beta_1 \gamma)^2} \right) \right)  \right), 
    \end{align*}
    where 
    \[
    \Cr{c_27}''' := 2^{-47-2\tau} \Big({\cH_{\kappa}^d(B_{\rho/2})\over \cH_{\kappa}^d(B_{\rho})}\Big)^{2}\Big({\cH_{\kappa}^d(B_{\rho})\over \cH_{\kappa}^d(W)}\Big)^{s\over 2}\Big({\cH_{\kappa}^d(B_{r(W)})\over \cH_{\kappa}^d(W)}\Big).
    \]
\end{theorem}

\begin{proof}
    As argued in the proof of Theorem \ref{tm:SubgraphCounts} (see also the proof of Corollary \ref{cor:RGGsubgraphConstantCurve}) $\Lambda_W$ satisfies \eqref{eq:IntensityMeasureCondition1} and \eqref{eq:IntensityMeasureCondition2}.
    Thus, Lemma \ref{lm:A3forGraphs} applied with $m=2$, $n=1$, $Q=\rho^\tau$, and $h(x,y) = \dist(x,y)^\tau$ gives that $F_2^{(\tau)}({\gamma},\rho)$ satisfies \eqref{assumptionsGeneral} with $g(r)= {\bf 1} (r\leq \rho)$, $M=\rho^\tau/2$ and  $f(x,y) = \frac{1}{2}{\bf 1} (\dist_{\kappa}(x,y)\leq \rho) \dist(x,y)^\tau$.
    Thus, by item 3 of Lemma \ref{lem:relAssumptions}, we get that \eqref{eq:fBound} holds with
    \begin{align*}
        q&=0,
        &\beta_0&=1, 
        &\beta_1'
        &= C(g,\Lambda) \max\left( 1 , \frac{2\|f\|_{L^1(\Lambda^2)}}{\rho^{\tau} C(g,\Lambda)^2} \right)^{\frac{s}{2}} ,
        &\beta_2' &= \frac{\rho^\tau}{2} \max\left( 1 , \frac{2\|f\|_{L^1(\Lambda^2)}}{\rho^{\tau} C(g,\Lambda)^2} \right)^{\frac{1-s}{2}} ,
    \end{align*} 
    where $s\in[0,1]$ can be chosen arbitrarily.
    We note that since $\rho\leq r(W)$ we get
    \[ C(g,\Lambda)
    = \sup_{x\in\bM_{\kappa}} \int_{\bM_\kappa} g(\dist_{\kappa}(x,y)) \Lambda(\dint y)
    = \sup_{x\in\bM_{\kappa}} \cH_{\kappa}^d (B_{\rho}(x) \cap W)
    = \cH_{\kappa}^d(B_{\rho})  , \]
    and observe that
    \begin{align*}
        \|f\|_{L^1(\Lambda^2)}
        &= {1\over 2}\int_{\bM_{\kappa}}\int_{\bM_{\kappa}} \dist_{\kappa}(x,y)^{\tau}{\bf 1}\{\dist_{\kappa}(x,y) \leq \rho\}{\bf 1}\{y\in W\}{\bf 1}\{x\in W\}\cH^d_{\kappa}(\dint y)\cH^d_{\kappa}(\dint x)\\
        &\leq {1\over 2}\int_{\bM_{\kappa}}\int_{\bM_{\kappa}} \rho^{\tau}{\bf 1} \{y\in B_{\rho}(x)\} {\bf 1}\{x\in W\}\cH^d_{\kappa}(\dint y)\cH^d_{\kappa}(\dint x) \\
        &= {\rho^{\tau} \over 2}\cH^d_{\kappa}(W)\cH^d_{\kappa}(B_{\rho}).
    \end{align*}
    Hence, we get $\max\left( 1 , \frac{2\|f\|_{L^1(\Lambda^2)}}{\rho^{\tau} C(g,\Lambda)^2} \right) \leq \cH_{\kappa}^d(W) \cH_{\kappa}^d(B_{\rho})^{-1} $, and therefore $\beta_1' \leq \beta_1$ and $\beta_2' \leq \beta_2$.

    Let $\Cr{c_27} := 2^{-45} \min\left( 1 , \|f_1\|_{L^2(\Lambda)}^2 \beta_2^{-2} \beta_1^{-3} \right) $ where $ f_1(x)
    = \frac{1}{2} \int_{\bM_{\kappa}} \dist(x,y)^{\tau}{\bf 1}\{\dist_{\kappa}(x,y) \leq \rho\} \, \Lambda_W(\dint y) $.
    It remains only to show that $\Cr{c_27} \geq \Cr{c_27}'''$, since then the results follows immediately from Theorem \ref{tm:Summary}.
    For this we note that
    \begin{align*}
        \|f_1\|_{L^2(\Lambda)}^2
        &=\int_{\bM_{\kappa}}{\bf 1}\{y\in W\}\Big( \frac{1}{2} \int_{\bM_{\kappa}}\dist_{\kappa}(x,y)^{\tau}{\bf 1}\{x\in B_{\rho}(y)\}{\bf 1}\{x\in W\}\cH^d_{\kappa}(\dint x)\Big)^2\cH^d_{\kappa}(\dint y)\\
        &\ge \frac{1}{2^2} \int_{\bM_{\kappa}}{\bf 1}\{y\in B_{r(W)}\}\Big(\int_{\bM_{\kappa}}\dist_{\kappa}(x,y)^{\tau}{\bf 1}\{x\in B_{\rho}(y)\}{\bf 1}\{x\in B_{r(W)}\}\cH^d_{\kappa}(\dint x)\Big)^2\cH^d_{\kappa}(\dint y)\\
        &\ge \frac{(\delta\rho)^{2\tau}}{2^2} \int_{\bM_{\kappa}}{\bf 1}\{y\in B_{r(W)}\}\Big(\int_{\bM_{\kappa}}{\bf 1}\{x\in B_{\rho}(y)\setminus B_{\delta\rho}(y)\}{\bf 1}\{x\in B_{r(W)}\}\cH^d_{\kappa}(\dint x)\Big)^2\cH^d_{\kappa}(\dint y)\\
        &= \frac{(\delta\rho)^{2\tau}}{2^2} \int_{\bM_{\kappa}}{\bf 1}\{y\in B_{r(W)}\}\Big(\cH^d_{\kappa}(B_{\rho}(y)\cap B_{r(W)})-\cH^d_{\kappa}(B_{\delta\rho}(y)\cap B_{r(W)})\Big)^2\cH^d_{\kappa}(\dint y),
    \end{align*}
    which holds for any $\delta\in [0,1]$, $\tau\ge 0$, using the convention $0^0=1$. Since function 
    $$
    y\mapsto \cH^d_{\kappa}(B_{\rho}(y)\cap B_{r(W)})-\cH^d_{\kappa}(B_{\delta\rho}(y)\cap B_{r(W)})=\cH^d_{\kappa}((B_{\rho}(y)\setminus B_{\delta\rho}(y))\cap B_{r(W)})
    $$ 
    depends only on $\dist_{\kappa}(o,y)$ and is decreasing in $\dist_{\kappa}(o,y)$. Hence, it achieves its minimum on $B_{r(W)}$ for $y$ with $\dist_{\kappa}(o,y)=r(W)$. Further we note, that for such $y$ we have $\cH^d_{\kappa}(B_{\delta\rho}(y)\cap B_{r(W)})\leq {1\over 2}\cH^d_{\kappa}(B_{\delta\rho})$ and 
    $$
    \cH^d_{\kappa}(B_{\rho}(y)\cap B_{r(W)})\ge \cH^d_{\kappa}(B_{\rho}(z)\cap B_{\rho})\ge \cH^d_{\kappa}(B_{\rho/2}),
    $$
    where $z$ is any point with $\dist_{\kappa}(o,z)=\rho$. Hence, choosing $\delta= 1/2$ we have
    $$
    \|f_1\|_{L^2(\Lambda)}^2
    \ge {1\over 2^4}\big({\rho\over 2}\big)^{2\tau}\cH^d_{\kappa}(B_{\rho/2})^2 \cH^d_{\kappa}(B_{r(W)}) ,
    $$
    and get
    \begin{align*}
        2^{-45} \min\left( 1 , \|f_1\|_{L^2(\Lambda)}^2 \beta_2'^{-2} \beta_1^{-3} \right)
        &\geq 
        2^{-45} \min\left( 1 , 2^{-2-2\tau}\Big({\cH_{\kappa}^d(B_{\rho/2})\over \cH_{\kappa}^d(B_{\rho})}\Big)^{2}\Big({\cH_{\kappa}^d(B_{\rho})\over \cH_{\kappa}^d(W)}\Big)^{s\over 2}\Big({\cH_{\kappa}^d(B_{r(W)})\over \cH_{\kappa}^d(W)}\Big)\right)
        = \Cr{c_27}''' ,
    \end{align*}
    where the last equality follows from
    $\cH^d_{\kappa}(B_{\rho/2}) 
    \leq \cH_{\kappa}^d(B_{\rho})
    \leq \cH^d_{\kappa}(B_{r(W)}) 
    \leq \cH_{\kappa}^d(W)$.
\end{proof}

In the next corollary we consider the specific case of $\kappa=0$ (Euclidean space) and $W=B_r$.
This gives us a result with very explicit dependencies on the various parameters of the model.

\begin{corollary}
\label{cor:RGGEuclidean}
    Let $d\ge 2$ and $W = B_r \subset \bR^d$ be a $d$-dimensional Euclidean ball of radius $r>0$.
    Let $\eta$ be a Poisson point process with intensity measure $\gamma\Lambda_W$, defined by \eqref{eq:LambdaW}. 
    Set
    \begin{align*}
        s & \in[0,1] ,
        &\beta_1 &= \kappa_d r^{\frac{d s}{2}} \rho^{d (1-\frac{s}{2})} ,
        & \beta_2 &= 2^{-1} r^{{d(1-s)\over 2}} \rho^{\tau -{d(1-s)\over 2}},
    \end{align*}
     where $\kappa_d=\pi^{d\over 2}/\Gamma(d/2+1)$ denotes the Lebesgue measure of a $d$-dimensional Euclidean ball. Then for any  $\tau\ge 0$, $0< \rho \leq r$, $\gamma \geq 16\beta_1^{-1} $ and any $t\geq 0$,  
    \begin{align*}
        \bP(|F_2^{(\tau)}-\bE F_2^{(\tau)}| \geq t)
        &\leq 2\exp \left( - \Cr{c_27}''' \left(\frac{t}{\beta_2} \right)^{\frac{1}{2}} \min \left( \left( \frac{t}{\beta_2 (\beta_1 \gamma)^2} \right)^{\frac{3}{2}} , 1 + \log_+ \left( \frac{t}{\beta_2 (\beta_1 \gamma)^2} \right) \right)  \right), 
    \end{align*}
    where $\Cr{c_27}''' := 2^{-47-2\tau-2d} (\rho/r)^{ds\over 2} $.
\end{corollary}

\begin{proof}
    These bounds follow directly from Theorem \ref{thm:RGGnonEuclidean} and $\cH^d_0(B_{\rho})=\kappa_d\rho^d$.
\end{proof}

Finally note, that in case when $W$ is fixed and $\rho=\rho(\gamma)\to 0$ as $\gamma\to \infty$ we have that $\cH^d_{\kappa}(B_{\rho})=\Theta(\rho^d)$.
In this case the curvature parameter $\kappa$ does not influence the limiting behaviour of Poisson $U$-statistic $F_2^{(\tau)}(\gamma,\rho(\gamma))$ and for any $\kappa$ we will obtain bounds of the same form as in Corollary \ref{cor:RGGEuclidean}. This situation happens because as $\rho\to 0$ the edges of a random graph $G(\eta,\rho)$ are determined locally and this local geometry is almost Euclidean for any constant curvature space. This situation happens, for example if we fix the average degree $\delta = \gamma \cH_\kappa^d(B_{\rho}) $ of a typical point of the random graph, and let $\gamma\to \infty$.
We derive the following corollary which improves the bound \eqref{eq:ConcentrationRSTGilbertGraph} (when $t$ is large enough) and generalizes it to constant curvature spaces.
\begin{corollary}
\label{thm:RGGfixeddeg}
    Let $\kappa\in \bR$, $d\ge 2$, and $W\subset \bM_{\kappa}$ be a compact set with inradius $r(W)>0$. If $\kappa > 0$, we additionally assume that $W$ does not contain antipodal points.
    Let $\eta$ be a Poisson point process with intensity measure $\gamma\Lambda_W$. 
    Let $\delta \geq 16 $.
    For any $\gamma>0$, we set $\rho_\gamma := \rho(\gamma , d , \kappa , \delta)$ such that $ \gamma \cH_\kappa^d(B_{\rho}) = \delta $.

    Then, there exists $c>0$, independent of $\gamma$ (but dependent on $\kappa$, $d$, $r(W)$, $\delta$ and $\tau$), such that for any $ \gamma > \delta /  \cH_\kappa^d(B_{\min(|\kappa|^{-1/2},r(W))}) $ (with convention ${1\over \infty}:=0$) we have for any $s\in [0,1]$ that
    \begin{align*}
        \bP\left(|F_2^{(\tau)}(\eta_\gamma , \rho_\gamma)-\bE F_2^{(\tau)}(\eta_\gamma , \rho_\gamma)| \geq t \right)
        &\leq 
        \begin{cases}
            2\exp \left( -  c  t^2 \gamma^{\frac{2\tau}{d} -1}\right) ,
            & t \leq \gamma^{-\frac{\tau}{d}+\frac{1}{2} } ,
            \\
            2\exp \left( -  c  t^2 \gamma^{\frac{2\tau}{d} -1-\frac{s}{2} }\right) 
            =\exp \left( -  c  (\gamma^{\tau\over d}t)^{\frac{s}{1+s} }\right) ,
            & t = \gamma^{ -\frac{\tau}{d} + \frac{1+s}{2} } ,
            \\
            2\exp \left( -  c  t^{\frac{1}{2}} \gamma^{\frac{\tau}{2d}} \log \left( e t \gamma^{\frac{\tau}{d} - 1} \right)\right) ,
            & t \geq \gamma^{ -\frac{\tau}{d} + 1} .
        \end{cases}  
    \end{align*}
\end{corollary}

\begin{remark}
    We compare our result with the bound of Reitzner, Schulte and Thäle presented in Equation~\eqref{eq:ConcentrationRSTGilbertGraph}.
    We observe that the term $\gamma^{3\tau-d\over 4d}t^{3\over 4}$ in \eqref{eq:ConcentrationRSTGilbertGraph} is never smaller than the two other terms in the minimum therein, and thus the bound can also be written as 
    \begin{equation*}
        \bP\Bigl(\bigl| F_2^{(\tau)} \bigl( \eta_\gamma, \frac{\widetilde \rho}{\gamma^{1/d}} \bigr) -\bE F_2^{(\tau)} \bigl( \eta_\gamma , \frac{\widetilde \rho}{ \gamma^{1/d}} \bigr) \bigr| \ge t \Bigr)
        \leq 
        \begin{cases}
            \exp\Big(-c \gamma^{{2\tau\over d}-1}t^2 \Big) ,
            & \gamma\geq 1,\ 0 \leq t \leq \gamma^{ -\frac{\tau}{d} +\frac{3}{5}}, \\
            \exp\Big(-c (\gamma^{\tau\over d}t)^{1\over 3} \Big) ,
            & \gamma\geq 1,\ t \geq \gamma^{ -\frac{\tau}{d} +\frac{3}{5}} .
        \end{cases}
    \end{equation*}
    We note that 
    \begin{itemize}
        \item for $t\leq \gamma^{-\frac{\tau}{d}+\frac{1}{2}}$ this bound is of the same order as ours,
        \item for $ \gamma^{-\frac{\tau}{d}+\frac{1}{2}} \leq t\leq \gamma^{-\frac{\tau}{d}+ \frac{3}{4}}$ this bound is of smaller order (i.e.\ better) than ours,
        \item for $ t\geq \gamma^{-\frac{\tau}{d}+ \frac{3}{4}}$ this bound is higher order (i.e.\ worse) than ours.
    \end{itemize}
\end{remark}

\begin{proof}
    In the following lines, we write $a\asymp b$ if there exist constants $0<c<C<\infty$  such that $c \leq a/b \leq C$, with $c$ and $C$ independent of $\gamma$, but dependent on $\kappa$, $d$, $r(W)$, $\delta$ and $\tau$.
    We have $\rho \asymp \gamma^{-1/d} $ and $\cH_{\kappa}^d(B_{\rho}) \asymp \gamma^{-1}$, since $\gamma\cH^d(B_{\rho})=\delta$ and $\gamma> \delta/\cH_\kappa^d(B_{|\kappa|^{-{1/2}}})$, meaning that $\rho<|\kappa|^{-1/2}$ and, hence, $\cH^d(B_{\rho})\asymp \rho^d$.
    Set
    \begin{align*}
        s & \in[0,1] ,
        &\beta_1 &= \cH_{\kappa}^d(W)^{\frac{s}{2}} \cH_{\kappa}^d(B_{\rho})^{1-\frac{s}{2}} 
        \asymp \gamma^{\frac{s}{2}-1},
        & \beta_2 &= \frac{\rho^\tau}{2} \left( \frac{\cH_{\kappa}^d(W)}{\cH_{\kappa}^d(B_{\rho})} \right)^{\frac{1-s}{2}} 
        \asymp \gamma^{\frac{1-s}{2} - \frac{\tau}{d} } .
    \end{align*}
    Note that $\gamma \beta_1 \geq \gamma \cH_{\kappa}^d(B_{\rho}) = \delta \geq 16$ and that $\rho = \rho(\gamma) \leq r(W)$ because $ \gamma \cH_\kappa^d(B_{\rho}) = \delta $ and $\gamma> \delta/\cH_\kappa^d(B_{r(W)})$.
    Thus, Theorem \ref{thm:RGGnonEuclidean} implies that for any 
    $t\geq 0$
    we have,  
    \begin{align*}
        \bP(|F_2^{(\tau)}-\bE F_2^{(\tau)}| \geq t)
        &\leq 2\exp \left( - \Cr{c_27}''' \left(\frac{t}{\beta_2} \right)^{\frac{1}{2}} \min \left( \left( \frac{t}{\beta_2 (\beta_1 \gamma)^2} \right)^{\frac{3}{2}} , 1 + \log_+ \left( \frac{t}{\beta_2 (\beta_1 \gamma)^2} \right) \right)  \right), 
    \end{align*}
    where 
    $\Cr{c_27}''' 
    := 2^{-45-2\tau} \cH^d_{\kappa}(B_{r(W)}) \cH_{\kappa}^d(W)^{-\frac{s}{2}-1} \cH^d_{\kappa}(B_{\rho/2})^2  \cH_{\kappa}^d(B_{\rho})^{-2+\frac{s}{2}} 
    \asymp \gamma^{ - {s\over 2} } $.
    Therefore, by picking $ c >0$ sufficiently small, we have
    \begin{align*}
        \bP(|F_2^{(\tau)}-\bE F_2^{(\tau)}| \geq t)
        &\leq 2\exp \left( -  c  \left(\frac{t}{\gamma^{\frac{1-s}{2} -\frac{\tau}{d} }} \right)^{\frac{1}{2}} \min \left( \left( \frac{t}{ \gamma^{\frac{1+s}{2} - \frac{\tau}{d} } } \right)^{\frac{3}{2}} , 1 + \log_+ \left( \frac{t}{ \gamma^{\frac{1+s}{2} - \frac{\tau}{d} } } \right) \right)  \right)
        \\& =\begin{cases}
            2\exp \left( -  c  t^2 \gamma^{\frac{2\tau}{d} -1-\frac{s}{2} }\right) ,
            & t \leq \gamma^{\frac{1+s}{2} -\frac{\tau}{d} } ,
            \\
            2\exp \left( -  c  t^{\frac{1}{2}} \gamma^{\frac{\tau}{2d} - \frac{1-s}{4} } \log \left( e t \gamma^{-\frac{1+s}{2} + \frac{\tau}{d} } \right)\right) ,
            & t \geq \gamma^{\frac{1+s}{2} -\frac{\tau}{d} } .
        \end{cases}
    \end{align*}
    The results follows by applying this bound with $s=0$ when $ t\leq \gamma^{\frac{1}{2}-\frac{\tau}{d}}$ and with $s=1$ when $t \geq \gamma^{ -\frac{\tau}{d} + 1} $.
\end{proof}

\section{Applications to Poisson hyperplane process}\label{sec:Applications}

\subsection{Euclidean case}

We start with introducing the model. Let $\bA(d,d-1)$ be the Grassmannian of all $(d-1)$-dimensional affine subspaces in $\bR^d$ and let $\mu_{d-1}$ be the rigid motions invariant Haar measure normalized as
$$
\mu_{d-1}\big(\{H\in \bA(d,d-1)\colon H\cap \bB^d\neq \emptyset\}\big)=2,
$$
where $\bB^d$ denotes the $d$-dimensional unit ball in $\bR^d$. Moreover, we denote by $\omega_d=2\pi^{d/2}/\Gamma(d/2)$ the surface area of the unit $(d-1)$-dimensional sphere $\bS^{d-1}$. A Poisson process $\eta$ on $\bA(d,d-1)$ with intensity measure $\gamma\mu_{d-1}$ is called a (stationary and isotropic) Poisson hyperplane process. 

The functional we are interested in is the Poisson $U$-statistic $F_m(f, \eta)$, $1\leq m\leq d$ with 
\[
f(H_1,\ldots,H_m)=\frac{1}{m!}V_i\left(H_1\cap\ldots\cap H_m \cap W\right){\bf 1}\{\dim (H_1\cap \ldots\cap H_m)=d-m\},
\]
where $W\subset \bR^d$ is some compact convex set and $V_i$ stands for the $i$-th intrinsic volume, $0\leq i\leq d$. Intrinsic volumes are continuous and positive functionals on the space of convex bodies (convex and compact sets), which are monotone under set inclusion. They describe the inner structure of a convex body $K$ and play an important role in stochastic and convex geometry. For exact definitions and further properties of intrinsic volumes we refer reader to \cite[Chapter 13]{SW}.
Further, for $\gamma>0$, we denote by
\begin{align*}
F_{m,i}(W,\gamma)\coloneqq &\frac{1}{m!}\sum_{(H_1,\ldots,H_m)\in \eta^m_{\neq}}V_i(H_1\cap\ldots\cap H_m\cap W){\bf 1}\{\dim (H_1\cap \ldots\cap H_m)=d-m\},
\end{align*}
where $1\leq m\leq d$ and $0\leq i\leq d-m$. We note that for convex body $K$ of dimension $d-m$ we have that $V_{d-m}(K)$ coincides with $(d-m)$-dimensional Hausdorff measure of $K$ and $V_0(K)$ is the Euler characteristic.
In particular $V_{0}(E\cap W)={\bf 1}\{E\cap W \neq \emptyset\}$ for $E\in \bA(d,d-m)$ and $F_{d,0}(W,\gamma)$ counts the number of intersection points of Poisson hyperplane process $\eta$ in a window $W$. Since the kernel $f$ and the measure $\mu_{d-1}$ are independent of $\gamma$, $F_{m,i}(W,\gamma)$ is an example of geometric $U$-statistic and the quantitative CLT for $F_{m,i}(W,\gamma)$ as $\gamma\to \infty$ has been established in \cite[Theorem 5.3]{RS13}. 
It should be noted that, instead of considering the case $\gamma\to\infty$, one could alternatively fix $\gamma$ and consider the sequence of Poisson $U$-statistics $F_{m,i}(rW,\gamma)$ in a growing window $rW$, $r>0$ as $r\to \infty$.
On the other hand due to the mapping properties of Poisson point processes it is easy to see that
$$
F_{m,i}(rW,\gamma)\overset{d}{=}r^iF_{m,i}(W,r\gamma),
$$
(see for example \cite[Corollary 6.4]{LPST14}) and both cases are equivalent.

In what follows we assume that $W$ is a fixed $d$-dimensional convex and compact subset of $\bR^d$ with $V_{d-m}(W)\in (0,\infty)$, and we use the short notation $\nu_i\coloneqq V_i(W)$, $1\leq i\leq d$.
We see that $F_{m,i}(W,\gamma)$ satisfies assumptions \eqref{assumptions} with $\bX = \{ H \in \bA(d,d-1) : H\cap W \neq \emptyset \}$ and 
\begin{align}  \label{paramsHyperplanes}
    \alpha_1 = \mu_{d-1} (\bX) = \frac{\omega_{d+1}}{\pi\omega_d} \nu_{1}, \quad 
    \alpha_2 = {1\over m!} \sup_{F \in \bA(d,d-m)} V_i(F\cap W) \leq {1\over m!} \nu_{i},
\end{align}
when $\gamma \geq \alpha_1^{-1} $.
The equality $\mu_{d-1} (\bX) = \frac{\omega_{d+1}}{\pi\omega_d} \nu_{1}$ follows from \cite[Theorem 5.1.1]{SW}, for example. The following theorem provides concentration bounds for $F_{m,i}(W,\gamma)$.

\begin{theorem}
\label{thm:EuclideanHyperplanes}
    The following inequality holds for any $1\leq m\leq d$, $0\leq i\leq d-m$, $t>0$ and $\gamma \geq 8m \alpha_1^{-1}$
    \[
     \bP(|F_{m,i}(W,\gamma)-\bE F_{m,i}(W,\gamma)| \geq t)
        \leq 2\exp \left( - \Cr{c_27} \left(\frac{t}{\alpha_2} \right)^{\frac{1}{m}} \min \left( \left( \frac{t}{\alpha_2 (\alpha_1 \gamma)^m} \right)^{2-\frac{1}{m}} , 1 + \log_+ \left( \frac{t}{\alpha_2 (\alpha_1 \gamma)^m} \right) \right)  \right), 
    \]
    where $\Cr{c_27}=2^{-17m-4}m^{-2m-3} \min\left( 1 , \left(\frac{ m\,\omega_{i+1}\nu_{i+m}}{ \alpha_2\omega_{i+m+1}(\pi\nu_1)^m} \right)^2\right) $ and $\alpha_1$, $\alpha_2$ are defined in \eqref{paramsHyperplanes}.
\end{theorem}

\begin{proof}
    We start by noting that by Crofton's formula \cite[Theorem 5.1.1]{SW} applied $m-1$ times we have
    \begin{align*}
        \|f_1\|^2_{L^2(\Lambda)}
        &=m^2\int_{\bA(d,d-1)}\Big(\int_{\bA(d,d-1)^{m-1}}V_i(H_1\cap\ldots\cap H_{m-1}\cap H\cap W)\mu^{\otimes (m-1)}_{d-1}(\dint (H_1,\ldots, H_{m-1})\Big)^2\mu_{d-1}(\dint H)\\
        &=m^2\Big({\omega_{d+1}\over \omega_d}\Big)^{2(m-1)}{\omega_{i+1}^2\over \omega_{i+m}^2}\int_{\bA(d,d-1)}V_{i+m-1}(W\cap H)^2\mu_{d-1}(\dint H)
    \end{align*}
    Applying Jensen's inequality and the Crofton formula one more time we get
    \begin{align*}
        \int_{\bA(d,d-1)}V_{i+m-1}(W\cap H)^2 \mu_{d-1}(\dint H)
        & \geq \alpha_1 \left( \int_{\bA(d,d-1)}V_{i+m-1}(W\cap H) \frac{\mu_{d-1}(\dint H)}{\alpha_1} \right)^2 
        \\& = \frac{1}{\alpha_1} \left(\frac{\omega_{d+1}\omega_{i+m}}{\omega_d\omega_{i+m+1}} V_{i+m}(W) \right)^2.
    \end{align*}
    Thus
    \begin{align*}
        \|f_1\|^2_{L^2(\Lambda)}&
        \geq \alpha_1^{-1}\left(\frac{ m\,\omega_{i+1}\nu_{i+m}}{ \omega_{i+m+1}}  \right)^2\Big({\omega_{d+1}\over \omega_d}\Big)^{2m}=\alpha_1^{-1+2m}\left(\frac{ m\,\omega_{i+1}\nu_{i+m}}{ \omega_{i+m+1}(\pi\nu_1)^m} \right)^2.
    \end{align*}
    Then, the result is given directly by Theorem \ref{tm:Summary} (in its extended version at the beginning of Section \ref{sec:discussion}) and Lemma \ref{lem:relAssumptions} since $F_{m,i}(W,\gamma)$ satisfies \eqref{assumptions} with $\alpha_1$ and $\alpha_2$ given by \eqref{paramsHyperplanes}.
\end{proof}

\subsection{Hyperbolic case} \label{sec:hyperbolic}

Let $\bH^d$ be a $d$-dimensional hyperbolic space, which is a simply connected Riemannian manifold of constant sectional curvature $-1$. The $(d-1)$-dimensional totally geodesic subspaces of $\bH^d$ are called hyperplanes and let $\bA_h(d,d-1)$ denotes the space of all hyperplanes in $\bH^d$. There is a unique (up to normalization) measure $\mu_{d-1}^h$ on $\bA_h(d,d-1)$ which is invariant under isometries of $\bH^d$. In this section we assume that $\mu_{d-1}^h$ is normalized in the standard way, namely as in \cite[Section 3.1]{HHT2019}. For $\gamma>0$ let $\eta^h$ be a Poisson process on $\bA_h(d,d-1)$ with intensity measure $\gamma\mu_{d-1}^h$, which is called a (hyperbolic) Poisson hyperplane process. For more information on hyperbolic geometry and, in particular, on hyperbolic Poisson hyperplane process we refer reader to \cite[Section 3]{HHT2019}. 

Let $\cH^s$, $s\ge 0$ denote the $s$-dimensional Hausdorff measure with respect to intrinsic metric on Riemannian manifold $\bH^d$ and let $B_r$ be a closed ball in $\bH^d$ or radius $r>0$. As in Euclidean case we consider a Poisson $U$-statistic of the following form
\begin{align*}
F_{m}^h(r,\gamma)\coloneqq&\frac{1}{m!}\sum_{(H_1,\ldots,H_m)\in (\eta^h)^m_{\neq}}\cH^{d-m}(H_1\cap\cdots\cap H_m\cap B_r){\bf 1}\{\dim (H_1\cap \cdots\cap H_m)=d-m\}.
\end{align*}
In particular $F_{d-1}^h(r,\gamma)$ is the total number of intersection points of Poisson hyperplane process $\eta^h$ and $F_{1}^h(r,\gamma)$ is the total surface content of the union of all hyperplanes of $\eta^h$ within a window $B_r$.

Note that, contrary to what happens in the Euclidean setting, scaling of the window size $r$ and of the intensity $\gamma$ have different effects on the random variable $F_m^h (r,\gamma)$. If one fixes $r$ and considers $\gamma\to\infty$, the assumptions \eqref{assumptions} hold with parameters $\alpha_1$ and $\alpha_2$ constant (they depend only on $d$ and $r$). 
Thus, we get good concentration bounds from Theorem \ref{tm:Summary}.
In this case we would get a similar result as in the Euclidean setting, see Theorem \ref{thm:EuclideanHyperplanes}. In this section, we are interested in deriving concentration bounds for $F_{m}^h(r,\gamma)$ with a particular focus on the situation when $r\to \infty$ and $\gamma = \gamma(r)$ (typically constant).
With this perspective, some computations show that it is not enough to restrict to the assumption \eqref{assumptions}, since the bounds derived under this assumption become meaningless because of the rapid growth of the parameters $\alpha_1$ and $\alpha_2$ as $r\to\infty$. For simplicity and in order to avoid the discussion to become too lengthy, we will mostly focus on the case $m=1$.
Results in Section \ref{sec:Wu} need, in any case, the assumption $m=1$, but we believe that, with additional efforts, the results of Section \ref{sec:appliOurResults} could be extended to $m\geq 2$.

The rest of the (sub)section is structured as follows.
In Section \ref{sec:CLT}, we describe CLT and non CLT results from \cite{HHT2019,KRT22} and derive concentration bounds from these results.
In Section \ref{sec:Wu}, we apply the result from Wu \cite{wu2000new} to derive other concentration bounds in the case $m=1$.
In Section \ref{sec:appliOurResults}, we show that $F_{1}^h(r,\gamma)$ satisfies $\eqref{eq:fBound}$ (with explicit parameters $\beta_0$, $\beta_1$ and $\beta_2$), derive bounds from our main result and compare them with the other bounds.
Finally, in Section \ref{sec:proofsHyperbolic}, we provide the proofs of the various theorems of Section \ref{sec:hyperbolic}.

\subsubsection{CLT and non CLT results and their consequences} \label{sec:CLT}

Herold, Hug and Th\"ale showed in \cite[Theorem 5]{HHT2019} that, for fixed $\gamma$ and $r\to\infty$, $F_m^h(r,\gamma)$ satisfies CLT if $d=2$ and $m=1,2$, or if $d=3$ and $m=1,2,3$.
In cases $d\ge 7$, $m\in \{1,\ldots, d\}$ and $d\ge 4$, $m=1$ they showed that CLT does not hold.
For the remaining cases, the situation is still open, but it is conjectured in \cite{betken2023} that there is no CLT as soon as $d\ge 4$, for any value of $m$.
From now on, we will focus on the case $m=1$.

As already mentioned above, for $d=2$ and $d=3$, $F_1^h(r, \gamma)$ satisfies CLT when $\gamma$ is fixed and $r\to\infty$.
Moreover the speed of convergence is of order $e^{-r/2}$ for $d=2$ and of order $r^{-1}$ for $d=3$, with respect to the Kolmogorov distance \cite[Theorem 5]{HHT2019}.
From this, similarly as in Section \ref{sec:AnalysingCLT} where we derived \eqref{eq:241a} from \eqref{eq:CLT}, we find that there exist $r_0>0$ and $C>0$ such that
\begin{align*}
\bP\left(|F_1^h(r, \gamma)-\bE F_1^h(r, \gamma)| \geq s\sqrt{\bV F_1^h(r, \gamma)}\right)\leq  2\exp( -s^2/2),
\quad
r \geq r_0, \quad
0\leq s \leq
\begin{cases}
    C \sqrt{r}, & d=2,\\
    C \sqrt{\log(r)}, & d=3 .
\end{cases}
\end{align*}
In Theorem \ref{tm:HyperbolicHyperplaneConcentrationWuCLT} below, we will recover Gaussian tails, with bounds of the form $\exp(-c s^2)$ for some constant $c>0$.
On one hand, they are not as tight because $c$ is smaller than $1/2$, but on the other hand, they apply to a larger range of value $s$, which is allowed to be of order at most $e^{r/2}$ for $d=2$ and at most $\sqrt{r}$ for $d=3$.

Now, we will consider the setting $d\geq 4$, $m=1$, $\gamma$ fixed and $r\to\infty$, for which we already mentioned that there is no CLT.
In the more recent work of Kabluchko, Rosen and Th\"ale \cite[Theorem 2.1]{KRT22} it was shows that
\begin{equation} \label{eq:241b}
    {2^{d-2}(d-2)\over \omega_{d-1}}{(F_{1}^h(r,\gamma)-\bE F_{1}^h(r,\gamma)\over e^{(d-2)r}} \overset{d}{\longrightarrow} Z_d,\qquad\qquad d\ge 4, \quad \gamma=1, \quad r\to\infty,
\end{equation}
where $Z_d$ is some given infinitely divisible non-Gaussian random variable. 
In particular the L\'evy measure of $Z_d$ is explicitly given and has support $(0,1)$, see \cite[Remark 2.3]{KRT22}. 
It should be also noted that the variance of $F_{1}^h(r,\gamma)$ has order of growth $e^{2(d-2)r}$ as $r\to\infty$ \cite[Lemma 19]{HHT2019}, and hence, the scaling coincides with the one used in CLT.
By \cite[Theorem 26.1]{Sato}, we have $\bP(Z_d \geq s) =  e^{-s\log(s)(1+o(1))}$ as $s\to\infty$, which suggests that, for large $r$, the random variable $e^{-(d-2)r}(F_{1}^h(r,\gamma)-\bE F_{1}^h(r,\gamma))$ has Poisson tail.
This will be confirmed by Corollary \ref{tm:HyperbolicHyperplaneConcentrationWuNoCLT} below.

\subsubsection{Concentration bounds following from the work of Wu} \label{sec:Wu}

Since we focus on the case $m=1$, we can use the result of Wu \cite{wu2000new}. This leads to the following theorem and corollaries, whose proofs are in Section \ref{sec:proofs}.

\begin{theorem}
\label{tm:HyperbolicHyperplaneConcentration Wu}
    For any $d\ge 2$, $r>0$ and $\gamma>0$ we have
    $$
    \bP(F_1^h(r, \gamma)-\bE F_1^h(r, \gamma) \geq t)
    \leq  \exp(-I_d(\gamma, r, t)),
    $$
    where 
    $$
    I_d(\gamma, r,t)=\begin{cases}
        \frac{t}{4r}\log\Big(1+{tr\over 32\gamma e^{r}}\Big),\qquad &d=2,\\
         \frac{t}{\omega_2e^r}\log\Big(1+{t\over 4\gamma\omega_2 re^{r}}\Big),\qquad &d=3,\\
         {2^{d-3}(d-2)t\over \omega_{d-1}e^{(d-2)r}}\log\Big(1+{(d-2)t\over 2^{d-1}\gamma\omega_{d-1} e^{(d-2)r}}\Big),\qquad &d\ge 4.
    \end{cases}
    $$
\end{theorem}

Next, we apply the above theorem to bound the deviation probability of the centred and normalized version of $F_1^h(r, \gamma)$.
We see that we get Gaussian bounds in any dimensions, but it is important to pay attention to the range of allowed values for $s$.
In any dimensions these ranges grow with $\gamma$, which is to be expected since we always have CLT when $\gamma\to\infty$ and $r$ is lower bounded, see \cite[Theorem 6, Remark 5]{HHT2019}.
If $\gamma$ is fixed and $r\to\infty$, then the ranges grow with $r$ only if $d=2$ or $d=3$, which is again to be expected since these are precisely the cases for which CLT holds.

\begin{corollary}[Gaussian tails] \label{tm:HyperbolicHyperplaneConcentrationWuCLT}
    For $r\ge 3$ and $\gamma\ge 1$ we have
    \begin{align*}
    \bP\left({F_1^h(r, \gamma)-\bE F_1^h(r, \gamma)\over\sqrt{\bV F_1^h(r, \gamma)}} \geq s\right)\leq 
    \begin{cases}
        \exp\left( -2^{-8} s^2\right),
        & d=2,\qquad 0<s\leq 2^5r^{-1}(\gamma e^{r})^{1/2},\\
        \exp\left( - 2^{-9} s^2\right), 
        & d= 3,\qquad 0<  s\leq 2^5(r\gamma)^{1/2},\\
         \exp\left( - 2^{-4d+3} s^2\right), 
        & d\ge 4,\qquad 0<  s\leq 2^{3d-3}\sqrt{d-2}\gamma^{1/2}.
    \end{cases}
    \end{align*}
\end{corollary}

For $d\geq 4$ and fixed $\gamma$, the last corollary does not provide bounds for large $s$.
The following corollary, provide bounds of a different quality for such case.
Note, that for convenience we normalize slightly differently, but this normalization is also of order $\sqrt{\bV F_1^h(r,\gamma)}$.
The following Poisson tail bounds are the best we could hope for since they apply to the limiting random variable \eqref{eq:241b}.
\begin{corollary}[Poisson tails] \label{tm:HyperbolicHyperplaneConcentrationWuNoCLT}
    For any $d\ge 4$, $r>0$ and $\gamma\ge 1$ we have
    \begin{align*}
        \bP\Big( {2^{d-2}(d-2)\over \omega_{d-1}}{(F_{1}^h(r,\gamma)-\bE F_{1}^h(r,\gamma)\over e^{(d-2)r}} \geq s\Big)&\leq \exp\left( - \Big(s+2^{2d-3}\gamma \Big)\log\Big(1+{s\over 2^{2d-3}\gamma}\right)+s\Big),\qquad s>0,\\
        &\leq \exp\left( - s\log(s)+s\log(2^{2d-3}e\gamma)\right), \qquad s> 2^{2d-3}e\gamma,\\
    &= \exp\left( - s\log(s)(1+o(1))\right),\qquad s\to\infty.
    \end{align*}
\end{corollary}

\subsubsection{Applications of Theorem \ref{tm:Summary}} \label{sec:appliOurResults}
Below, we show that $F_1^h(r, \gamma)$ satisfies \eqref{eq:fBound} which allows us to derive concentration bounds.
We will then comment on the fact that these bounds are comparable to the ones derived in Section \ref{sec:Wu} above.
We stress that, with additional work, our method can be extended to $m\geq 2$ whereas the results derived from the work of Wu are specific to $m=1$.

\begin{lemma} \label{lem:A2holds}
    For any $d\geq 2$, $F_1^h(r, \gamma)$ satisfies \eqref{eq:fBound} with $\beta_0=1$ and
    \begin{align*}
        q&=1, &\beta_1&=2e^r, &\beta_2&=4 , &\text{ if } &d=2;\\
        q&=0, &\beta_1&=2r, &\beta_2&=\omega_2e^{r}, &\text{ if } &d=3;\\
        q&=0, &\beta_1&=2, &\beta_2&={\omega_{d-1}\over d-2}e^{r(d-2)}, &\text{ if } &d\ge 4.
    \end{align*}
\end{lemma}
The proof of the lemma above is in Section \ref{sec:proofs}.
The next theorem follows immediately.
\begin{theorem}
\label{tm:ConcentrationHypHyperplane1}
For any $d\ge 2$ we have
$$
\bP(|F_1^h(r, \gamma)-\bE F_1^h(r, \gamma)|\geq t)\leq  2\exp(-I_d(\gamma, r, t)),
$$
where 
\[
I_d(\gamma, r,t)=2^{-4d-22}\omega_{d-1}^{-1}\frac{t}{e^{r(d-2)}}  \min \left( \frac{t}{\beta' \gamma}, 1 + \log_+ \left( \frac{t}{\beta' \gamma} \right)\right), \qquad d\ge 3,\, \gamma\ge 8\beta_1^{-1},
\]
with
\[
\beta'={\omega_{d-1}\over d-2}\begin{cases}
re^r,\qquad &d=3,\\
e^{r(d-2)},\qquad &d\ge 4,
\end{cases},\qquad\qquad \beta_1=\begin{cases}
2r,\qquad &d=3,\\
2,\qquad &d\ge 4,
\end{cases}
\]
and 
\[
I_2(\gamma, r,t)=
    \frac{t}{2^{31}}\min \left( \frac{t}{8e^{r}\gamma}, 1 \right), \qquad \gamma\ge 4e^{-r}.
\]

\end{theorem}

The concentration bounds from Theorem \ref{tm:ConcentrationHypHyperplane1} are of the same order as the bounds from Theorem \ref{tm:HyperbolicHyperplaneConcentration Wu}, up to constant factors in $I_d$ depending only on $d$.
To see this, note that the terms of the form $\log(1+x)$ in Theorem \ref{tm:HyperbolicHyperplaneConcentration Wu}, can be approximated by $x$ if $x\leq 1$ and by $1+\log_+(x)=1+\log(x)$ if $x\ge 1$.

\subsubsection{Proofs} \label{sec:proofs}
\label{sec:proofsHyperbolic}
    
\begin{proof}[Proof of Theorem \ref{tm:HyperbolicHyperplaneConcentration Wu}]
   Note that for any $H\in \bA_h(d,d-1)$ we have 
   $$
   D_H F_1^h(r, \gamma)= \cH^{d-1}(H\cap B_r)\leq \omega_{d-1}\int_0^r\sinh^{d-2}(x)\dint x\leq \begin{cases} 
   2r,\quad &d=2,\\
   {\omega_{d-1}\over 2^{d-2}(d-2)}e^{(d-2)r},\quad &d\ge 3.
   \end{cases}.
   $$ 
   On the other hand by Lemma \ref{lm:UpperBoundIntegral} we also have 
   \begin{align*}
   \gamma\int_{\bA_h(d,d-1)}|D_H F_1^h(r, \gamma)|^2\mu_{d-1}^h(\dint H)&= \gamma\int_{\bA_h(d,d-1)}\cH^{d-1}(H\cap B_r)^2\mu_{d-1}^h(\dint H)\\
   &\leq \gamma \begin{cases} 
   64e^r,\quad &d=2,\\
   2\omega_2^2re^{2r},\quad &d=3,\\
   2\omega_{d-1}^2(d-2)^{-2}e^{2(d-2)r},\quad &d\ge 4.
   \end{cases}
   \end{align*}
   Combining the above estimates and \cite[Proposition 3.1]{wu2000new} we obtain the result.
\end{proof}

\begin{proof}[Proof of Theorem \ref{tm:HyperbolicHyperplaneConcentrationWuCLT}]
    Let us start by considering the case $d=2$. Then we get
    $$
    \bP\left(F_1^h(r, \gamma)-\bE F_1^h(r, \gamma)\geq s\sqrt{\bV F_1^h(r, \gamma)}\right)\leq \exp\left(-{s\sqrt{\bV F_1^h(r, \gamma)}\over 4r}\log\left(1+{sr\sqrt{\bV F_1^h(r, \gamma)}\over 2^5\gamma e^{r}}\right)\right),\qquad s>0.
    $$
    Since by Lemma \ref{lm:Variance} we have $\bV F_1^h(r, \gamma)\ge \gamma e^{r}$ from the above inequality we obtain 
    $$
    \bP\left(F_1^h(r, \gamma)-\bE F_1^h(r, \gamma)\geq s\sqrt{\bV F_1^h(r, \gamma)}\right)\leq \exp\left(-{s(\gamma e^{r})^{1/2}\over 4r}\log\left(1+{sr\over 2^5(\gamma e^{r})^{1/2}}\right)\right),\qquad s>0.
    $$
    Further since $\log(1+x)>\frac{x}{1+x}>x/2$ for $x<1$ we get that, for $s<2^5r^{-1}(\gamma e^{r})^{1/2}$,
    $$
    \bP\left(F_1^h(r, \gamma)-\bE F_1^h(r, \gamma)\geq s\sqrt{\bV F_1^h(r, \gamma)}\right)\leq \exp\left(-{2^{-8}s^2}\right).
    $$
    Similar for $d=3$ using $\bV F_1^h(r, \gamma)\ge 2^{-6}\omega_2^2\gamma re^{2r}$ (see Lemma \ref{lm:Variance}) we get 
    $$
    \bP\left(F_1^h(r, \gamma)-\bE F_1^h(r, \gamma)\geq s\sqrt{\bV F_1^h(r, \gamma)}\right)\leq \exp\left(-{s(r\gamma)^{1/2}\over 2^{3}}\log\left(1+{s\over 2^{5}(r\gamma)^{1/2}}\right)\right),\qquad s>0,
    $$
    and for $s<2^5(r\gamma)^\frac{1}{2}$ we obtain
    $$
    \bP\left(F_1^h(r, \gamma)-\bE F_1^h(r, \gamma)\geq s\sqrt{\bV F_1^h(r, \gamma)}\right)\leq \exp\left(-{2^{-9}s^2}\right).
    $$
    Finally, for $d\ge 4$ using $\bV F_1^h(r, \gamma)\ge 2^{-4d+6}\omega_{d-1}^2(d-2)^{-3}\gamma e^{2r(d-2)}$ we have
    $$
    \bP\left(F_1^h(r, \gamma)-\bE F_1^h(r, \gamma)\geq s\sqrt{\bV F_1^h(r, \gamma)}\right)\leq \exp\left(-{s\gamma^{1/2}\over 2^d\sqrt{d-2}}\log\left(1+{s\over 2^{3d-4}\sqrt{d-2}\gamma^{1/2} }\right)\right),\qquad s>0,
    $$
    and for $s<2^{3d-4}\sqrt{d-2}\gamma^{1/2}$ we conclude
    $$
    \bP\left(F_1^h(r, \gamma)-\bE F_1^h(r, \gamma)\geq s\sqrt{\bV F_1^h(r, \gamma)}\right)\leq \exp\left(-{s^2\over 2^{4d-3}(d-2)}\right).
    $$
\end{proof}

\begin{proof}[Proof of Theorem \ref{tm:HyperbolicHyperplaneConcentrationWuNoCLT}]
    Recall that for any $H\in \bA_h(d,d-1)$ and $d\ge 4$ we have 
    $$
    D_H F_1^h(r, \gamma)\leq 
    {\omega_{d-1}\over 2^{d-2}(d-2)}e^{(d-2)r},
    $$ 
    and
    \begin{align*}
    \gamma\int_{A_h(d,d-1)}|D_H F_1^h(r, \gamma)|^2\mu_{d-1}^h(\dint H)\leq 2\gamma
    \omega_{d-1}^2(d-2)^{-2}e^{2(d-2)r}.
    \end{align*}
    Then by \cite[Proposition 3.1]{wu2000new} we get
    \begin{align*}
        \bP\left(F_1^h(r, \gamma)-\bE F_1^h(r, \gamma) \geq {\omega_{d-1}e^{(d-2)r}\over 2^{d-2}(d-2)}s\right)&\leq \exp\left( - \Big(s+2^{2d-3}\gamma
       \Big)\log\left(1+{s\over 2^{2d-3}\gamma}\right)+s\right).
    \end{align*}
    Moreover, for $s> 2^{2d-3}e\gamma$ using $\log(x+c)>\log(x)$ for any $c>0$ and $x>0$ we have that
    \begin{align*}
    \bP\left(F_1^h(r, \gamma)-\bE F_1^h(r, \gamma) \geq {\omega_{d-1}e^{(d-2)r}\over 2^{d-2}(d-2)}s\right)&\leq \exp\left( - s\log(s+2^{2d-3}\gamma)+s\log(2^{2d-3}e\gamma)\right)\\
    &\leq \exp\left( - s\log(s)+s\log(2^{2d-3}e\gamma)\right)\\
    &= \exp\left( - s\log(s)(1+o(1))\right).
    \end{align*}
\end{proof}

\begin{proof}[Proof of Lemma \ref{lem:A2holds}] 
    We start by showing that for any $\ell\in\bN$, $r>0$ and any $\sigma \in \Pi^{**}_{\ge 2}(1;\ell)$ we have $\|\sigma\|=\ell$ and, hence,
    \begin{equation} \label{eq:fd1}
        \int (f^{\otimes \ell})_{\sigma} \dint (\mu_{d-1}^h)^{\ell +|\sigma|-\|\sigma\|}=\int (f^{\otimes \ell})_{\sigma} \dint (\mu_{d-1}^h)^{|\sigma|}\leq
        \begin{cases}(2e^r)^{|\sigma|}4^{\ell}\prod_{k=1}^{|\sigma|}(|J_k|)!,& d=2,\\
        (2r)^{|\sigma|}(\omega_2e^{r})^{\ell},& d=3,\\
        2^{|\sigma|}({\omega_{d-1}\over d-2}e^{r(d-2)})^{\ell},& d\ge 4,
        \end{cases}
    \end{equation}
    where $f(H)=\cH^{d-1}(H\cap B_r)$.

    Recall the notation introduced in Section \ref{sec:subpartitions}. Consider a partition $\sigma\in \Pi_{\geq 2}^{**}(1;\ell)$ with blocks $J_1,\ldots,J_{|\sigma|}$. Note that according to definition we have $|J_1|+\ldots+|J_{|\sigma|}|=\|\sigma\|=\ell$ since each row $R_j$, $j\in[\ell]$, which in this case consists of a single element, is hit by at least one block $J_i$, $1\leq i\leq |\sigma|$ and as a consequence there are no uncovered elements. Further let $\tau_{\cdot}(j)\colon\Pi_\ell\to [\ell]$, $j\in[\ell]$ be the functions, such that $\tau_{\sigma}(j)=k$ if $j\in J_k$. Thus, we obtain
    \begin{align}
    T_{d,\ell}(\sigma)&\coloneqq \int (f^{\otimes \ell})_{\sigma} \dint (\mu_{d-1}^h)^{\otimes(|\sigma|)}\notag\\
    &= \int \prod_{j=1}^\ell \cH^{d-1}(H_{\tau_{\sigma}(j)}\cap B_r)\dint (\mu_{d-1}^h)^{\otimes(|\sigma|)}\notag\\
    &= \prod_{k=1}^{|\sigma|} \int_{\bA_h(d,d-1)} \cH^{d-1}(H\cap B_r)^{|J_k|}\mu_{d-1}^h(\dint H).\label{eq:T1}
    \end{align}
    The integrals of the above form have been investigated in \cite[Lemma 8]{HHT2019} and the following bounds have been obtained
    \begin{equation}\label{eq:MomentBoundsHyperbolic2}
    \int_{\bA_h(d,d-1)}\cH^{d-1}(H\cap B_r)^{k}\mu_{d-1}^h(\dint H)\leq \begin{cases}
        2\cdot 4^kk!e^{r}, \qquad &d=2,k\ge 2,\\
        2\cdot\omega_2^k re^{2r},\qquad &d=3, k=2,\\
        2\cdot\omega_{d-1}^k(d-2)^{-k}e^{rk(d-2)},\qquad &d\ge 4, k\ge 2\text{ or }d=3, k>2.
    \end{cases}
    \end{equation}
    It should be noted that in \cite[Lemma 8]{HHT2019} the above estimate has been stated exactly in terms of $r$, but the exact dependence on $k$ and $d$ has not been specified and only the existence of the constants, depending only on $k$ and $d$ has been claimed. Nevertheless these constants are easy to determine following the proof of Lemma 8 in \cite{HHT2019}. For reader's convenience of we put these computations in appendix.
    
    Then combining \eqref{eq:T1} with \eqref{eq:MomentBoundsHyperbolic2} we get
    $$ 
    T_{d,\ell}(\sigma)\leq  \begin{cases}
            (2e^r)^{|\sigma|}4^{\ell}\prod_{k=1}^{|\sigma|}(|J_k|)!, \qquad &d=2,\\
            2^{|\sigma|}\omega_2^{\ell}r^{t(\sigma)}e^{2rt(\sigma)}e^{r(\ell - 2t(\sigma))}\leq (2r)^{|\sigma|}(\omega_2e^{r})^{\ell},\qquad &d=3,\\
            2^{|\sigma|}(\omega_{d-1}(d-2)^{-1}e^{r(d-2)})^{\ell},\qquad& d\ge 4,
        \end{cases}
    $$
    where $t(\sigma)$ is the number of blocks of partition $\sigma$ having size $2$ and, hence, $t(\sigma)\leq |\sigma|$.
    This implies the estimate \eqref{eq:fd1}. Further note that according to \eqref{eq:fd1}, $F_1^h(r, \gamma)$ satisfies \eqref{eq:fBound} with $\beta_0=1$ and
    \begin{align*}
        q&=1, &\beta_1&=2e^r, &\beta_2&=4 , &\text{ if } &d=2;\\
        q&=0, &\beta_1&=2r, &\beta_2&=\omega_2e^{r}, &\text{ if } &d=3;\\
        q&=0, &\beta_1&=2, &\beta_2&={\omega_{d-1}\over d-2}e^{r(d-2)}, &\text{ if } &d\ge 4.
    \end{align*}
\end{proof}

\begin{proof}[Proof of Theorem \ref{tm:ConcentrationHypHyperplane1}]

	We start by noting that by Lemma \ref{lm:LowerBoundIntegral} we have
	$$
	\|f_1\|^2_{L^2(\Lambda)}=\int_{\bA_h(d,d-1)}\cH^{d-1}(H\cap B_r)^2\mu_{d-1}^h(\dint H)\ge \begin{cases}
        e^{r}, \quad &d=2,\\
        2^{-6}\omega_2^2 re^{2r},\quad &d=3,\\
        2^{-4d+6}\omega_{d-1}^2(d-2)^{-3}e^{2r(d-2)},\quad &d\ge 4.
    \end{cases}
	$$
	Since, according to Lemma \ref{lem:A2holds}, $F_1^h(r, \gamma)$ satisfies \eqref{eq:fBound} with $q=0$ if $d\ge 3$ and with $q=1$ if $d=2$, the bound above implies that
	\[
	\Cr{c_27} = 2^{-26}\min\left( 1 , {\|f_1\|_{L^2(\Lambda)}^2\over\beta_2^{2} \beta_1} \right) \ge 2^{-26}\begin{cases}
        2^{-5}, \quad &d=2,3,\\
        2^{-4d+6}(d-2)^{-1},\quad &d\ge 4.
    \end{cases}
	\]
	Now the bound follows directly from \eqref{onebound2}.
\end{proof}

\section*{Acknowledgements}

AG and GB were supported by the DFG priority program SPP 2265 \textit{Random Geometric Systems}.
AG was supported by the DFG under Germany's Excellence Strategy  EXC 2044 -- 390685587, \textit{Mathematics M\"unster: Dynamics - Geometry - Structure}.
GB would like to acknowledge the financial support of the CogniGron research center and the Ubbo Emmius Funds (Univ. of Groningen).

The authors would like to thank Matthias Schulte, Christoph Th\"ale and Zakhar Kabluchko for valuable comments and remarks regarding the applications.

\appendix

\section{Estimates for hyperbolic hyperplane process}

\begin{lemma}\label{lm:UpperBoundIntegral}
For any $k\ge 1$ and $r>0$ we have
\begin{equation}\label{eq:UpperBoundIntegral2}
    \int_{\bA_h(d,d-1)}\cH^{d-1}(H\cap B_r)^{k}\mu_{d-1}^h(\dint H)\leq \begin{cases}
        2\cdot 4^kk!e^{r}, \quad &d=2,k\ge 2,\\
        2\cdot \omega_2^k re^{2r},\quad &d=3, k=2,\\
        2\cdot \omega_{d-1}^k(d-2)^{-k}e^{rk(d-2)},\quad &d\ge 4, k\ge 2\text{ or }d=3, k>2.
    \end{cases}
\end{equation}
\end{lemma}

\begin{proof}
By \cite[Lemma 8]{HHT2019} we have
$$
I_d(k)\coloneqq \frac{1}{2}\int_{\bA_h(d,d-1)}\cH^{d-1}(H\cap B_r)^{k}\mu_{d-1}^h(\dint H)=\int_{0}^r\cosh^{d-1}(s)\cH^{d-1}(L_{d-1}(s)\cap B_r)^k\dint s,
$$
where $L_{d-1}(s)$ denotes the $(d-1)$-dimensional totally geodesic subspaces  such that the hyperbolic distance between $L_{d-1}(s)$ and the origin $o$ equals $s$. By \cite[Equation 19]{HHT2019} we have
$$
\cH^{d-1}(L_{d-1}(s)\cap B_r)=\omega_{d-1}\int_{0}^{{\rm arccosh}\big(\frac{\cosh(r)}{\cosh(s)}\big)}\sinh^{d-2}(x)\dint x.
$$
Moreover, by \cite[Lemma 6]{HHT2019} we get for any $r>0$ and $0\leq s\leq r$ that
\begin{equation}\label{eq:02.01.24_1}
{\rm arccosh}\Big(\frac{\cosh(r)}{\cosh(s)}\Big)\leq r-s+\log(2).
\end{equation}
Combining the above estimates with the trivial bounds $\cosh(x)\leq e^x$, $\sinh(x)\leq e^x/2$ and applying change of variables $u=r-s$ we obtain
\begin{align}
     I_d(k)&=\omega_{d-1}^k\int_{0}^r\cosh^{d-1}(r-u)\Big(\int_{0}^{{\rm arccosh}\big(\frac{\cosh(r)}{\cosh(r-u)}\big)}\sinh^{d-2}(x)\dint x\Big)^k\dint u\label{eq:17.12.23_eq1}\\
     &\leq \omega_{d-1}^k2^{-(d-2)k}\int_{0}^r\cosh^{d-1}(r-u)\Big(\int_{0}^{{\rm arccosh}\big(\frac{\cosh(r)}{\cosh(r-u)}\big)}e^{(d-2)x}\dint x\Big)^k\dint u\notag\\
     &\leq \omega_{d-1}^k2^{-(d-2)k}\int_{0}^re^{(d-1)(r-u)}\Big(\int_{0}^{u+\log(2)}e^{(d-2)x}\dint x\Big)^k\dint u.\notag
\end{align}
Now for $d=2$ we get for any $k\ge 1$ that
 \begin{align*}
     I_d(k)&\leq 2^ke^r\int_{0}^re^{-u}(u+\log(2))^k\dint u\\
     &\leq 2^{2k-1}e^r\int_{\log(2)}^{r}e^{-u}u^k\dint u+2^{2k-1}(\log(2))^ke^r\int_{0}^{\log(2)}e^{-u}\dint u\\
     &\leq  2^{2k-1}e^r\Gamma(k+1)+2^{2k-1}e^r\leq 4^kk!e^r,
\end{align*}
where in the second step we used that $a+b\leq 2^{k-1\over k}(a^k+b^k)^{1\over k}$ for any $a,b\ge 0$ as follows from H\"older inequality with $p=k$ and $q={k\over k-1}$. In case $d\ge 3$ we continue as follows
 \begin{align*}
     I_d(k)&\leq \omega_{d-1}^k2^{-(d-2)k}(d-2)^{-k}\int_{0}^re^{(d-1)(r-u)}(2^{(d-2)}e^{(d-2)u}-1)^k\dint u\\
     &\leq \omega_{d-1}^k(d-2)^{-k}\int_{0}^re^{(d-1)r}e^{(k(d-2)-(d-1))u}\dint u.
\end{align*}
Finally, for $d=3$ and $k=2$ we immediately obtain $I_2(3)\leq  \omega_{2}^2e^{2r}r$ and, otherwise,
\begin{align*}
     I_d(k)&\leq \omega_{d-1}^k(d-2)^{-k}(k(d-2)-(d-1))^{-1}e^{k(d-2)r}\leq \omega_{d-1}^k(d-2)^{-k}e^{k(d-2)r},
\end{align*}
which finishes the proof of \eqref{eq:UpperBoundIntegral2}.

\end{proof}

\begin{lemma}\label{lm:LowerBoundIntegral}
For any $r\ge 3$ have
$$
    \int_{\bA_h(d,d-1)}\cH^{d-1}(H\cap B_r)^{2}\mu_{d-1}^h(\dint H)\ge \begin{cases}
        e^{r}, \quad &d=2,\\
        2^{-6}\omega_2^2 re^{2r},\quad &d=3,\\
        2^{-4d+6}\omega_{d-1}^2(d-2)^{-3}e^{2r(d-2)},\quad &d\ge 4.
    \end{cases}
$$
\end{lemma}

\begin{proof}
    We start exactly like in the proof of the previous lemma and by \eqref{eq:17.12.23_eq1} together with $\cosh(x)\ge e^x/2$ we get
    \begin{align*}
        I_d&\coloneqq \frac{1}{2}\int_{\bA_h(d,d-1)}\cH^{d-1}(H\cap B_r)^{2}\mu_{d-1}^h(\dint H)\\
        &=\omega_{d-1}^2\int_{0}^r\cosh^{d-1}(r-u)\Big(\int_{0}^{{\rm arccosh}\big(\frac{\cosh(r)}{\cosh(r-u)}\big)}\sinh^{d-2}(x)\dint x\Big)^2\dint u\\
        &\ge 2^{-d+1}\omega_{d-1}^2\int_{0}^re^{(d-1)(r-u)}\Big(\int_{0}^{{\rm arccosh}\big(\frac{\cosh(r)}{\cosh(r-u)}\big)}\sinh^{d-2}(x)\dint x\Big)^2\dint u.
    \end{align*}
    Further we use the bound \cite[Lemma 6]{HHT2019} 
    $$
    {\rm arccosh}\Big(\frac{\cosh(r)}{\cosh(r-u)}\Big)\ge u,
    $$
    which holds for any $r>0$ and for $d=2$ with $r\ge 3$ we obtain
    \begin{align*}
        I_2&\ge 2e^{r}\int_{0}^re^{-u}u^2\dint u\ge 2e^{r}\int_{1}^2e^{-u}u^2\dint u\ge 2^{-1}e^r.
    \end{align*}
    For $d=3$ and $r\ge 3$ using the estimate $e^x/2\leq \cosh(x)\leq e^x$ we get
    \begin{align*}
        I_3&\ge 2^{-2}\omega_{2}^2\int_{0}^re^{2(r-u)}\Big(\int_{0}^{{\rm arccosh}\big(\frac{\cosh(r)}{\cosh(r-u)}\big)}\sinh(x)\dint x\Big)^2\dint u\\
        &= 2^{-2}\omega_{2}^2\int_{0}^re^{2(r-u)}\Big(\frac{\cosh(r)}{\cosh(r-u)}-1\Big)^2\dint u\\
        &\ge 2^{-2}\omega_{2}^2e^{2r}\int_{r/2}^r\big(1/2-e^{-u}\big)^2\dint u\\
        &\ge 2^{-7}\omega_{2}^2e^{2r}r.
    \end{align*}
    For $d\ge 4$ and $r\ge 3$ using additionally the estimate $\sinh(x)={1
    \over 2}(e^x-e^{-x})\ge {1\over 2\sqrt{2}}e^{x}$, which is valid for $x\ge \log(2+\sqrt{2})/2$ and $e^x\ge 1$ for $x\ge 0$ we finally obtain
    \begin{align*}
        I_d&\ge 2^{-d+1}\omega_{d-1}^2\int_{0}^re^{(d-1)(r-u)}\Big(\int_{0}^{u}\sinh^{d-2}(x)\dint x\Big)^2\dint u\\
         &\ge 2^{-4d+7}\omega_{d-1}^2\int_{r-1}^r\Big(\int_{u-1}^{u}e^{(d-2)x}\dint x\Big)^2\dint u\\
        &= 2^{-4d+7}\omega_{d-1}^2(d-2)^{-2}(1-e^{-(d-2)})^2\int_{r-1}^re^{2(d-2)u}\dint u\\
        &= 2^{-4d+6}\omega_{d-1}^2(d-2)^{-3}(1-e^{-(d-2)})^2(1-e^{-2(d-2)})e^{2(d-2)r}\\
        &\ge 2^{-4d+5}\omega_{d-1}^2(d-2)^{-3}e^{2(d-2)r},
    \end{align*}
    and the proof is finished.
\end{proof}

\begin{lemma}\label{lm:Variance}
For any $r\ge 3$ and $\gamma\ge 1$ we have
\begin{align*}
     \gamma e^r\leq &\bV F_1^h(r, \gamma)\leq 2^{6}\gamma e^r, \qquad &d=2,\\
    2^{-6}\omega_2^2 \gamma r e^{2r}\leq &\bV F_1^h(r, \gamma)\leq 2\omega_2^2 \gamma r e^{2r},\qquad &d=3,\\
    2^{-4d+6}\omega_{d-1}^2(d-2)^{-3}\gamma e^{2r(d-2)}\leq &\bV F_1^h(r, \gamma)\leq 2\omega_{d-1}^2(d-2)^{-2} \gamma e^{2r(d-2)},\qquad &d\ge 4,
    \end{align*} 
\end{lemma}

\begin{proof}
 First note that by \eqref{eq_UstatVar} we have
 $$
 \bV F_1^h(r, \gamma)=\gamma\int_{\bA_h(d,d-1)}\cH^{d-1}(H\cap B_r)^2\mu^h_{d-1}(\dint H).
 $$
 Thus, combining estimates from Lemma \ref{lm:UpperBoundIntegral} and Lemma \ref{lm:LowerBoundIntegral} we obtain the result.
\end{proof}

\bibliographystyle{abbrv}
\bibliography{biblio}

\begin{thebibliography}{10}

\bibitem{Ahle22}
T.~D. Ahle.
\newblock Sharp and simple bounds for the raw moments of the binomial and
  {P}oisson distributions.
\newblock {\em Statist. Probab. Lett.}, 182:Paper No. 109306, 5, 2022.

\bibitem{bachmann2016concentration}
S.~Bachmann and G.~Peccati.
\newblock Concentration bounds for geometric {P}oisson functionals: logarithmic
  {S}obolev inequalities revisited.
\newblock {\em Electron. J. Probab.}, 21:Paper No. 6, 44, 2016.

\bibitem{BR18}
S.~Bachmann and M.~Reitzner.
\newblock Concentration for {P}oisson {$U$}-statistics: subgraph counts in
  random geometric graphs.
\newblock {\em Stochastic Process. Appl.}, 128(10):3327--3352, 2018.

\bibitem{betken2023}
C.~Betken, D.~Hug, and C.~Thäle.
\newblock Intersections of poisson k-flats in constant curvature spaces.
\newblock {\em Stochastic Processes and their Applications}, 165:96--129, 2023.

\bibitem{BLM13}
S.~Boucheron, G.~Lugosi, and P.~Massart.
\newblock {\em Concentration inequalities}.
\newblock Oxford University Press, Oxford, 2013.
\newblock A nonasymptotic theory of independence, With a foreword by Michel
  Ledoux.

\bibitem{BHP07}
J.-C. Breton, C.~Houdr\'e, and N.~Privault.
\newblock Dimension free and infinite variance tail estimates on {P}oisson
  space.
\newblock {\em Acta Appl. Math.}, 95(3):151--203, 2007.

\bibitem{Chaf}
D.~Chafa\"{\i}.
\newblock Entropies, convexity, and functional inequalities: on
  {$\Phi$}-entropies and {$\Phi$}-{S}obolev inequalities.
\newblock {\em J. Math. Kyoto Univ.}, 44(2):325--363, 2004.

\bibitem{ErfctL}
S.-H. Chang, P.~C. Cosman, and L.~B. Milstein.
\newblock Chernoff-type bounds for the gaussian error function.
\newblock {\em IEEE Transactions on Communications}, 59:2939--2944, 2011.

\bibitem{eichelsbacher2015Moderate}
P.~Eichelsbacher, M.~Raič, and T.~Schreiber.
\newblock Moderate deviations for stabilizing functionals in geometric
  probability.
\newblock {\em Annales de l'Institut Henri Poincaré, Probabilités et
  Statistiques}, 51(1):89--128, 2015.

\bibitem{ET14}
P.~Eichelsbacher and C.~Th\"{a}le.
\newblock New {B}erry-{E}sseen bounds for non-linear functionals of {P}oisson
  random measures.
\newblock {\em Electron. J. Probab.}, 19:no. 102, 25, 2014.

\bibitem{GieringerLast}
F.~Gieringer and G.~Last.
\newblock Concentration inequalities for measures of a {B}oolean model.
\newblock {\em ALEA, Lat. Am. J. Probab. Math. Stat.}, 15:151--166, 2018.

\bibitem{NRP21}
N.~Gozlan, R.~Herry, and G.~Peccati.
\newblock Transport inequalities for random point measures.
\newblock {\em J. Funct. Anal.}, 281(9):Paper No. 109141, 45, 2021.

\bibitem{GST21}
A.~Gusakova, H.~Sambale, and C.~Th\"{a}le.
\newblock Concentration on {P}oisson spaces via modified {$\Phi$}-{S}obolev
  inequalities.
\newblock {\em Stochastic Process. Appl.}, 140:216--235, 2021.

\bibitem{Haight}
F.~A. Haight.
\newblock {\em Handbook of the {P}oisson distribution}, volume No. 11 of {\em
  Publications in Operations Research}.
\newblock John Wiley \& Sons, Inc., New York-London-Sydney, 1967.

\bibitem{HHT2019}
F.~Herold, D.~Hug, and C.~Th\"{a}le.
\newblock Does a central limit theorem hold for the $k$-skeleton of poisson
  hyperplanes in hyperbolic space?
\newblock {\em Probab. Theory Related Fields}, 179:889--968, 2021.

\bibitem{Hou02}
C.~Houdr\'{e}.
\newblock Remarks on deviation inequalities for functions of infinitely
  divisible random vectors.
\newblock {\em Ann. Probab.}, 30(3):1223--1237, 2002.

\bibitem{HP02}
C.~Houdr\'e and N.~Privault.
\newblock Concentration and deviation inequalities in infinite dimensions via
  covariance representations.
\newblock {\em Bernoulli}, 8(6):697--720, 2002.

\bibitem{HR-B03}
C.~Houdr\'e and P.~Reynaud-Bouret.
\newblock Exponential inequalities, with constants, for {U}-statistics of order
  two.
\newblock In {\em Stochastic inequalities and applications}, volume~56 of {\em
  Progr. Probab.}, pages 55--69. Birkh\"auser, Basel, 2003.

\bibitem{KRT22}
Z.~Kabluchko, D.~Rosen, and C.~Thäle.
\newblock Fluctuations of $\lambda$-geodesic {P}oisson hyperplanes in
  hyperbolic space.
\newblock {\em Israel Journal of Mathematics (to appear)}.

\bibitem{L-RP13b}
R.~Lachi\`eze-Rey and G.~Peccati.
\newblock Fine {G}aussian fluctuations on the {P}oisson space, {I}:
  contractions, cumulants and geometric random graphs.
\newblock {\em Electron. J. Probab.}, 18:no. 32, 32, 2013.

\bibitem{L-RP13}
R.~Lachi\`eze-Rey and G.~Peccati.
\newblock Fine {G}aussian fluctuations on the {P}oisson space {II}: rescaled
  kernels, marked processes and geometric {$U$}-statistics.
\newblock {\em Stochastic Process. Appl.}, 123(12):4186--4218, 2013.

\bibitem{LP}
G.~Last and M.~Penrose.
\newblock {\em Lectures on the {P}oisson process}, volume~7 of {\em Institute
  of Mathematical Statistics Textbooks}.
\newblock Cambridge University Press, Cambridge, 2018.

\bibitem{LPST14}
G.~Last, M.~D. Penrose, M.~Schulte, and C.~Th\"{a}le.
\newblock Moments and central limit theorems for some multivariate {P}oisson
  functionals.
\newblock {\em Adv. in Appl. Probab.}, 46(2):348--364, 2014.

\bibitem{Lee90}
A.~J. Lee.
\newblock {\em U-Statistics: Theory and Practice}.
\newblock Dekker, New York, 1990.

\bibitem{liu2023normal}
Q.~Liu and N.~Privault.
\newblock Normal approximation of subgraph counts in the random-connection
  model.
\newblock {\em ArXiv:2301.12145}, 2023.

\bibitem{oeis}
{OEIS Foundation Inc.}
\newblock The {O}n-{L}ine {E}ncyclopedia of {I}nteger {S}equences, 2023.
\newblock Published electronically at \url{http://oeis.org}.

\bibitem{PeccatiReitzner}
G.~Peccati and M.~Reitzner.
\newblock {\em Stochastic Analysis for Poisson Point Processes}, volume~7.
\newblock Bocconi \& Springer, 2016.

\bibitem{PT13}
G.~Peccati and C.~Th\"{a}le.
\newblock Gamma limits and {$U$}-statistics on the {P}oisson space.
\newblock {\em ALEA Lat. Am. J. Probab. Math. Stat.}, 10(1):525--560, 2013.

\bibitem{Penrose}
M.~Penrose.
\newblock {\em Random geometric graphs}, volume~5 of {\em Oxford Studies in
  Probability}.
\newblock Oxford University Press, Oxford, 2003.

\bibitem{R13}
M.~Reitzner.
\newblock Poisson point processes: large deviation inequalities for the convex
  distance.
\newblock {\em Electron. Commun. Probab.}, 18:no. 96, 7, 2013.

\bibitem{RS13}
M.~Reitzner and M.~Schulte.
\newblock {Central limit theorems for $U$-statistics of Poisson point
  processes}.
\newblock {\em The Annals of Probability}, 41(6):3879 -- 3909, 2013.

\bibitem{RST17}
M.~Reitzner, M.~Schulte, and C.~Th\"{a}le.
\newblock Limit theory for the {G}ilbert graph.
\newblock {\em Adv. in Appl. Math.}, 88:26--61, 2017.

\bibitem{SantaloBook}
L.~A. Santal\'{o}.
\newblock {\em Integral geometry and geometric probability}.
\newblock Cambridge Mathematical Library. Cambridge University Press,
  Cambridge, second edition, 2004.
\newblock With a foreword by Mark Kac.

\bibitem{Sato}
K.-i. Sato.
\newblock {\em L\'{e}vy processes and infinitely divisible distributions},
  volume~68 of {\em Cambridge Studies in Advanced Mathematics}.
\newblock Cambridge University Press, Cambridge, revised edition, 2013.
\newblock Translated from the 1990 Japanese original.

\bibitem{SW}
R.~Schneider and W.~Weil.
\newblock {\em Stochastic and {I}ntegral {G}eometry}.
\newblock Probability and its Applications (New York). Springer-Verlag, Berlin,
  2008.

\bibitem{Schu16}
M.~Schulte.
\newblock Normal approximation of {P}oisson functionals in {K}olmogorov
  distance.
\newblock {\em J. Theoret. Probab.}, 29(1):96--117, 2016.

\bibitem{ST23}
M.~Schulte and C.~Thaele.
\newblock Moderate deviations on poisson chaos.
\newblock {\em ArXiv:2304.00876}, 2023.

\bibitem{wu2000new}
L.~Wu.
\newblock A new modified logarithmic {S}obolev inequality for {P}oisson point
  processes and several applications.
\newblock {\em Probab. Theory Related Fields}, 118(3):427--438, 2000.

\end{thebibliography}

\end{document}